\documentclass[12pt]{amsart}
\usepackage{amssymb, amsthm, amsmath}
\usepackage[all]{xy}
\usepackage[margin=1.3in,marginparwidth=1in,marginparsep=.05in]{geometry}
\usepackage{xcolor,graphicx}
\usepackage{tcolorbox}
\tcbuselibrary{theorems}
\usepackage{enumitem}
\usepackage{tikz,tikz-cd}
\usepackage[normalem]{ulem}
\usetikzlibrary{shapes,arrows,positioning,decorations.pathreplacing}
\PassOptionsToPackage{hyphens}{url}
\usepackage{hyperref}

\usetikzlibrary{arrows,shapes,positioning}
\usetikzlibrary{decorations.markings}

\newcommand{\htop}{h_{\mathrm{top}}}
\newcommand{\ulim}{\varlimsup}

\newcommand{\ph}{\varphi}
\newcommand{\eps}{\epsilon}
\newcommand{\DDD}{\mathcal{D}}
\newcommand{\CCC}{\mathcal{C}}
\newcommand{\Cs}{\mathcal{C}^s}
\newcommand{\Cp}{\mathcal{C}^p}
\newcommand{\GGG}{\mathcal{G}}
\newcommand{\LLL}{\mathcal{L}}
\newcommand{\MMM}{\mathcal{M}}
\newcommand{\Mf}{\mathcal{M}_f}
\newcommand{\RR}{\mathbb{R}}
\newcommand{\ZZ}{\mathbb{Z}}

\newcommand{\NN}{\mathbb{N}}
\newcommand{\TT}{\mathbb{T}}

\newcommand{\zz}{\mathbf{z}}
\newcommand{\NE}{\mathrm{NE}}
\newcommand{\hexp}{h_{\mathrm{exp}}^\perp}
\newcommand{\Pexp}{P_{\mathrm{exp}}^\perp}
\newcommand{\hexpp}{h_{\mathrm{exp^+}}^\perp}

\DeclareMathOperator{\diam}{diam}
\DeclareMathOperator{\CAT}{CAT}

\newcommand{\PPP}{\mathcal{P}}
\newcommand{\SSS}{\mathcal{S}}
\newcommand{\BBB}{\mathcal{B}}

\newcommand{\JJJ}{\mathcal{J}}

\newcommand{\Sing}{\mathrm{Sing}}
\newcommand{\Reg}{\mathrm{Reg}}
\newcommand{\Dim}{\mathrm{Dim}}
\newcommand{\FFF}{\mathcal{F}}
\newcommand{\UUU}{\mathcal{U}}
\newcommand{\Per}{\mathrm{Per}}
\newcommand{\mureg}{\mu^{\mathrm{Reg}}}

\newcommand{\phc}{\ph^c}

\newcommand{\wM}{\widetilde{M}}

\theoremstyle{plain}
\newtheorem{theorem}{Theorem}[section]
\newtheorem{proposition}[theorem]{Proposition}
\newtheorem{lemma}[theorem]{Lemma}
\newtheorem{corollary}[theorem]{Corollary}

\newtcbtheorem[use counter*=theorem]
{mainthm}{Theorem}{colback=blue!5,colframe=blue!20!gray,fonttitle=\bfseries}{thm}

\theoremstyle{definition}

\newtheorem{definition}[theorem]{Definition}

\theoremstyle{remark}
\newtheorem{remark}[theorem]{Remark}
\newtheorem{exercise}[theorem]{Exercise}

\numberwithin{equation}{section}
\numberwithin{figure}{section}

\begin{document}
\title[Beyond Bowen's Specification Property]{Beyond Bowen's specification property}
\author{Vaughn Climenhaga and Daniel J.\ Thompson}

\begin{abstract}
A classical result in thermodynamic formalism is that for uniformly hyperbolic systems, every H\"older continuous potential has a unique equilibrium state. One proof of this fact is due to Rufus Bowen and uses the fact that such systems satisfy expansivity and specification properties. In these notes, we survey recent progress that uses generalizations of these properties to extend Bowen's arguments beyond uniform hyperbolicity, including applications to partially hyperbolic systems and geodesic flows beyond negative curvature. We include a new criterion for uniqueness of equilibrium states for partially hyperbolic systems with $1$-dimensional center.
\end{abstract}
\date{\today}
\thanks{V.C.\ is partially supported by NSF DMS-1554794. D.T.\ is partially supported by NSF DMS-1461163 and DMS-1954463.}
\maketitle
\setcounter{tocdepth}{1}
\tableofcontents

\section{Introduction}

We survey recent progress in the study of existence and uniqueness of measures of maximal entropy and equilibrium states in settings beyond uniform hyperbolicity using weakened versions of specification and expansivity. Our focus is a long-running joint project initiated by the authors in \cite{CT12}, and extended in a series of papers including \cite{CT16, BCFT18}. This approach is based on the fundamental insights of Rufus Bowen in the 1970's \cite{rB71, rB75}, who identified  and formalized three properties enjoyed by uniformly hyperbolic systems that serve as foundations for the equilibrium state theory: these properties are specification, expansivity, and a regularity condition now known as the Bowen property. We relax all three of these properties in order to study systems exhibiting various types of non-uniform structure. These notes start by recalling the basic mechanisms of Bowen, and then gradually build up in generality, introducing the ideas needed to move to non-uniform versions of Bowen's hypotheses. The generality is motivated by, and illustrated by, examples: we discuss applications in symbolic dynamics, to certain partially hyperbolic systems, and to wide classes of geodesic flows with non-uniform hyperbolicity. This survey has its roots in the authors' 6-part minicourse at the \emph{Dynamics Beyond Uniform Hyperbolicity} conference at CIRM in May 2019. 

Part \ref{part:1} describes Bowen's result for MMEs and the simplest case of our generalization. It begins by recalling the basic ideas of thermodynamic formalism (\S\ref{sec:therm}) and outlining Bowen's original argument in the simplest case: the measure of maximal entropy (MME) for a shift space with specification (\S\ref{sec:principles}). In \S\ref{sec:relax-spec}, we introduce the main idea of our approach, the use of \emph{decompositions} to quantify the idea of ``obstructions to specification'', and we give an application to $\beta$-shifts. Moving beyond the symbolic case requires the notion of expansivity, and in \S\ref{sec:MMEhomeos} we discuss the role this plays in Bowen's argument.

Part \ref{part:2} develops our general results for discrete-time systems. The notion of ``obstructions to expansivity'' is introduced in \S\ref{sec:weakexp}, and an application to partial hyperbolicity (the Ma\~n\'e example) is described in \S\ref{sec:DA}. Combining the notions of obstructions to specification and expansivity
leads to the general result for MMEs in discrete-time in \S\ref{sec:nonuniformMME}, which is applied in \S\ref{sec:ph} to the broader class of partially hyperbolic diffeomorphisms with one-dimensional center. The extension to equilibrium states for nonzero potential functions is given in \S\ref{sec:eq-st}.

Part \ref{part:3} is devoted to equilibrium states for geodesic flows, with particular emphasis on the case of non-positive curvature, which is one of the most widely studied examples of a non-uniformly hyperbolic flow. After recalling some geometric background in \S\ref{sec:geometry}, we give an introduction in \S\ref{sec:geodesic-ES} to the ideas in the paper \cite{BCFT18}, including the main ``pressure gap'' criterion for uniqueness, and how to decompose the space of orbit segments using a function $\lambda$ that measures curvature of horospheres. We also outline recent results for manifolds without conjugate points and $\CAT(-1)$ spaces.  In \S\ref{sec:Kproperty}, we discuss how to improve ergodicity of the equilibrium states in non-positive curvature to the much stronger Kolmogorov $K$-property. Finally, in \S\ref{sec:entropygap}, we describe our proof of Knieper's ``entropy gap'' for geodesic flow on a rank 1 non-positive curvature manifold.

To illustrate the broad utility of the specification-based approach to uniqueness, we mention the following applications of the machinery we describe, which go well beyond what we are able to discuss in detail in this survey.
\begin{itemize}
\item Measures of maximal entropy for symbolic examples: 
$\beta$-shifts, $S$-gap shifts, and their factors \cite{CT12}; 
certain shifts of quasi-finite type \cite{vC18};
S-limited shifts \cite{MS18}; 
 shifts with ``one-sided almost specification'' \cite{CP19};
$(-\beta)$-shifts \cite{SY20}; 
\item Equilibrium states for symbolic examples:
$\beta$-shifts in \cite{CT13}, their factors in \cite{vC18,CC} (in particular, \cite{CC} studies general conditions under which the ``pressure gap'' condition holds);
$S$-gap shifts in \cite{CTY}; certain $\alpha$-$\beta$ shifts \cite{CLR};
applications to Manneville--Pomeau and related interval maps \cite{CT13}.
\item 
Diffeomorphisms beyond uniform hyperbolicity:
Bonatti--Viana examples \cite{CFT18};
Ma\~n\'e examples \cite{CFT19}; Katok examples \cite{tW20};
certain partially hyperbolic attractors \cite{FO20}.
\item Geodesic flows:
non-positive curvature \cite{BCFT18};
no focal points \cite{CKP,CKP-2};
no conjugate points \cite{CKW};
$\CAT(-1)$ geodesic flows \cite{CLT20a}.
\end{itemize}
We also mention two related results: the machinery we describe has recently been used to prove ``denseness of intermediate pressures'' \cite{pS20}; an approach to uniqueness (and non-uniqueness) for equilibrium states using various weak specification properties has been developed by Pavlov \cite{rP16,rP19} for symbolic and expansive systems.

The current literature in the field is vibrant and continually growing. The scope of this article is restricted to the specification approach to equilibrium states, and we largely do not address the literature beyond that.  Other uses for the specification property that we do not discuss include large deviations properties, multifractal analysis, and universality constructions; see e.g.\ \cite{lY90,TV03,PS05, PS07, pV12, QS16, BV17} (among many others). Different variants of the specification property are sometimes more appropriate for these arguments; various definitions are surveyed in \cite{kY09,KLO16}.

We stress that we do not address the use of other techniques to study existence and uniqueness of equilibrium states. These approaches include transfer operator techniques, Margulis-type constructions, symbolic dynamics, and the Patterson-Sullivan approach. We suggest the following recent references as a starting point to delve into the literature: \cite{PPS,CP17,BCS18,FH19,CPZ,vC-dim}. Classic references include \cite{rB08, PP90, Keller}.

We also do not discuss the large and important area of statistical properties for equilibrium states.  If $f$ is a $C^{1+\alpha}$ Anosov diffeomorphism (or if $X$ is an Axiom A attractor) then the unique equilibrium state for the \emph{geometric potential} $\ph(x) = -\log |\det Df|_{E^u(x)}|$ is the physically relevant Sinai--Ruelle--Bowen (SRB) measure. This provides important motivation and application for thermodynamic formalism, and this general setting is one of the major approaches to studying the statistical properties of the SRB measure. References include \cite{rB08,PP90,BS93,vB00,vB00-b,lY02,BDV05,jC15}.

We sometimes adopt a conversational writing style. We hope that the informal style will be helpful for current purposes; we invite the reader to look at our original papers, particularly \cite{CT12, CT16, BCFT18} for a more precise account.

\part{Main ideas: uniqueness of the measure of maximal entropy}\label{part:1} 

We introduce our main ideas in the case of a discrete-time dynamical system $(X, f)$. In this section, we often consider the case when $(X, f)$ is a shift space. We also consider the general topological dynamics setting where $X$ is a compact metric space and $f\colon X \to X$ is continuous. In many of our examples of interest, $X$ is a smooth manifold and $f$ is a diffeomorphism.

\section{Entropy and thermodynamic formalism}\label{sec:therm}

For a \emph{probability vector} $\vec{p} = (p_1,\dots, p_N) \in [0,1]^N$, where $\sum p_i = 1$, the \emph{entropy} of $\vec{p}$ is $H(\vec{p}) = \sum_i -p_i \log p_i$.  The following is an elementary exercise:
\begin{itemize}
\item $\max_{\vec{p}} H(\vec{p}) = \log N$;
\item $H(\vec{p}) = \log N$\quad $\Leftrightarrow$\quad $p_i = \frac 1N$ for all $i$ \quad $\Leftrightarrow$\quad $p_i = p_j$ for all $i,j$.
\end{itemize}
These general principles lie at the heart of thermodynamic formalism for uniformly hyperbolic dynamical systems, with `probability vector' replaced by `invariant probability measure':
\begin{itemize}
\item 
there is a function called `entropy' that we wish to maximize;
\item it is maximized at a unique measure (variational principle and uniqueness);
\item that measure is characterized by an equidistribution (Gibbs) property.
\end{itemize}
Now we recall the formal definitions, referring to \cite{DGS76,pW82,kP89,VO16} for further details and properties.

Let $X$ be a compact metric space and $f\colon X\to X$ a continuous map.  This gives a discrete-time topological dynamical system $(X, f)$. Let $\Mf(X)$ denote the space of Borel $f$-invariant probability measures on $X$.  

When $f$ exhibits some hyperbolic behavior, $\Mf(X)$ is typically extremely large -- an infinite-dimensional simplex -- and it becomes important to identify certain ``distinguished measures'' in $\Mf(X)$.
This includes SRB measures, measures of maximal entropy, and more generally, equilibrium measures.

\begin{definition}[Measure-theoretic Kolmogorov--Sinai entropy]
Fix $\mu\in \Mf(X)$.
Given a countable partition $\alpha$ of $X$ into Borel sets, write
\begin{equation}\label{eqn:Hmu}
H_\mu(\alpha) := \sum_{A\in \alpha} -\mu(A) \log \mu(A)
= \int -\log \mu(\alpha(x)) \,d\mu(x)
\end{equation}
for the \emph{static entropy} of $\alpha$, where we write $\alpha(x)$ for the element of $\alpha$ containing $x$. One can interpret $H_\mu(\alpha)$ as the  expected amount of information gained by observing which partition element a point $x\in X$ lies in.  Given $j\leq k$, the corresponding \emph{dynamical refinement} of $\alpha$ records which elements of $\alpha$ the iterates $f^jx, \dots, f^kx$ lie in:
\begin{equation}\label{eqn:alpha-j-k}
\alpha_j^k = \bigvee_{i=j}^k f^{-i}\alpha
\quad\Leftrightarrow\quad
\alpha_j^k(x) = \bigcap_{i=j}^k f^{-i}(\alpha(f^i x)).
\end{equation}
A standard short argument shows that
\begin{equation}\label{eqn:H-subadd}
H_\mu(\alpha_0^{n+m-1}) \leq H_\mu(\alpha_0^{n-1}) + H_\mu(\alpha_n^{n+m-1})
= H_\mu(\alpha_0^{n-1}) + H_\mu(\alpha_0^{0+m-1}),
\end{equation}
so that the sequence $c_n = H_\mu(\alpha_0^{n-1})$ is subadditive: $c_{n+m} \leq c_n + c_m$.  Thus, by Fekete's lemma \cite{mF23},  $\lim \frac{c_n}n$ exists, and equals $\inf \frac{c_n}n$.  We can therefore define the \emph{dynamical entropy} of $\alpha$ with respect to $f$ to be
\begin{equation}\label{eqn:hmua}
h_\mu(f,\alpha) := \lim_{n\to\infty} \frac 1n H_\mu(\alpha_0^{n-1})
= \inf_{n\in\NN} \frac 1n H_\mu(\alpha_0^{n-1}).
\end{equation}
The \emph{measure-theoretic (Kolmogorov--Sinai) entropy} of $(X,f,\mu)$ is
\begin{equation}\label{eqn:hmu}
h_\mu(f) = \textstyle \sup_\alpha h_\mu(f,\alpha),
\end{equation}
where the supremum is taken over all partitions $\alpha$ as above for which $H_\mu(\alpha) < \infty$.
\end{definition}

The \emph{variational principle} \cite[Theorem 8.6]{pW82} states that
\begin{equation}\label{eqn:vp-0}
\sup_{\mu\in\Mf(X)} h_\mu(f) = \htop(X,f),
\end{equation}
where $\htop(X,f)$ is the \emph{topological entropy} of $f\colon X\to X$, which we will define more carefully below (Definition \ref{def:entropy}). Now we define a central object in our study.

\begin{definition}[MMEs]\label{def:mme}
A measure $\mu\in \Mf(X)$ is a \emph{measure of maximal entropy (MME)} for $(X,f)$ if $h_\mu(f) = \htop(X,f)$; equivalently, if  $h_\nu(f) \leq h_\mu(f)$ for every $\nu\in \Mf(X)$. 
\end{definition}

The following theorem on uniformly hyperbolic systems is classical.

\begin{theorem}[Existence and Uniqueness]\label{thm:classical}
Suppose one of the following is true.
\begin{enumerate}
\item $(X,f=\sigma)$ is a transitive shift of finite type (SFT).
\item $f\colon M\to M$ is a $C^1$ diffeomorphism and $X\subset M$ is a compact $f$-invariant topologically transitive locally maximal hyperbolic set.\footnote{In particular, this holds if $X=M$ is compact and $f$ is a transitive Anosov diffeomorphism.}
\end{enumerate}
Then there exists a unique measure of maximal entropy $\mu$ for $(X,f)$.  
\end{theorem}

\begin{remark}
The unique MME can be thought of as the `most complex' invariant measure for a system, and often encodes dynamically relevant information such as the distribution and asymptotic behavior of the set of periodic points.
\end{remark}

\section{Bowen's original argument: the symbolic case} \label{sec:principles}

\subsection{The specification property in a shift space}\label{sec:symbolic}

Following Bowen \cite{rB75}, we outline a proof of Theorem \ref{thm:classical} in the first case, when $(X,\sigma)$ is a transitive SFT.  The original construction of the MME in this setting is due to Parry and uses the transition matrix.  Bowen's proof works for a broader class of systems, which we now describe.

Fix a finite set $A$ (the \emph{alphabet}), let $\sigma\colon A^\NN\to A^\NN$ be the shift map $\sigma(x_1 x_2 \dots) = x_2 x_3 \dots$, and let $X\subset A^\NN$ be closed and $\sigma$-invariant: $\sigma(X)=X$.  Here $A^\NN$ (and hence $X$) is equipped with the metric $d(x,y) = 2^{-\min \{ n : x_n \neq y_n\}}$.  We refer to $X$ as a \emph{one-sided shift space}. One could just as well consider two-sided shift spaces by replacing $\NN$ with $\ZZ$ (and using $|n|$ in the definition of $d$); all the results below would be the same, with natural modifications to the proofs.  Note that so far we do not assume that $X$ is an SFT or anything of the sort.

Given $x\in A^\NN$ and $i<j$, we write $x_{[i,j]} = x_i x_{i+1} \cdots x_j$ for the \emph{word} that appears in positions $i$ through $j$.  We use similar notation to denote subwords of a word $w\in A^* := \bigcup_n A^n$.
Given $w\in A^n$, we write $|w| = n$ for the \emph{length} of the word, and $[w] = \{x\in X : x_{[1,n]} = w\}$ for the \emph{cylinder} it determines in $X$.  We write
\begin{equation}\label{eqn:L}
\LLL_n := \{ w\in A^n : [w]\neq \emptyset\},
\qquad\qquad
\LLL := \bigcup_{n\geq 0} \LLL_n,
\end{equation}
and refer to $\LLL$ as the \emph{language} of $X$.  

\begin{definition}
The \emph{topological entropy} of $X$ is $\htop(X) = \lim_{n\to\infty} \frac 1n \log \#\LLL_n$. We often write $h(X)$ for brevity.  The limit exists by Fekete's lemma using the fact that $\log\#\LLL_n$ is subadditive, which we prove in Lemma \ref{lem:counting} below.
\end{definition}

It is a simple exercise to verify that every transitive SFT has the following property: there is $\tau\in\NN$ such that for every $v,w\in \LLL$ there is $u\in \LLL$ with $|u|\leq \tau$ such that $vuw\in\LLL$.  Iterating this, we see that \begin{equation}\label{eqn:spec}
\begin{aligned}
&\text{for every }w^1,\dots,w^k\in\LLL
\text{ there are }u^1,\dots,u^{k-1}\in\LLL
 \\
&\text{such that }|u^i| \leq \tau \text{ for all } i,
\text{ and } w^1u^1w^2u^2 \cdots u^{k-1} w^k \in \LLL.
\end{aligned}
\end{equation}
We say that a shift space whose language satisfies \eqref{eqn:spec} has the \emph{specification property}. There are a number of different variants of specification in the literature:\footnote{The terminology in the literature for these different variants (weak specification, almost specification, almost weak specification, transitive orbit gluing, etc.) is not always consistent, and we make no attempt to survey or standardize it here. To keep our terminology as simple as possible, we just use the word \emph{specification} for the version of the definition which is our main focus. In places where a different variant is considered, we take care to emphasize this.}
for example, one might ask that the connecting words $u^i\in \LLL$ satisfy $|u^i| = \tau$, which implies topological mixing, not just transitivity (this stronger property holds for mixing SFTs).
The version in \eqref{eqn:spec} is sufficient for the uniqueness argument, which is the main goal of these notes.\footnote{For other purposes, and especially in the absence of any expansivity property, the difference between $\leq \tau$ and $=\tau$ can be quite substantial, see for example \cite{BTV17,pS20b}.}

\begin{theorem}[Shift spaces with specification]\label{thm:bowen-shift}
Let $(X,\sigma)$ be a shift space with the specification property.  Then there is a unique measure of maximal entropy on $X$.
\end{theorem}

In the remainder of this section, we outline the two main steps in the proof of Theorem \ref{thm:bowen-shift}: proving uniqueness using a Gibbs property (\S\ref{sec:lower-gibbs}), and building a measure with the Gibbs property using specification (\S\ref{sec:building-gibbs}).\footnote{The notes at \url{https://vaughnclimenhaga.wordpress.com/2020/06/23/specification-and-the-measure-of-maximal-entropy/} give a slightly more detailed version of this proof.}

\begin{remark}\label{rmk:history}
As mentioned above, the original proof that a transitive SFT has a unique MME is due to Parry \cite{wP64}.  Parry constructed the MME using eigendata of the transition matrix for the SFT, and proved uniqueness by showing that any MME must be a Markov measure, then showing that there is only one MME among Markov measures.

A different proof of uniqueness in the SFT case was given by Adler and Weiss, who gave a more flexible argument based on showing that if $\mu$ is the Parry measure, then every $\nu\perp \mu$ must have smaller entropy. The argument is described in \cite{AW67}, with full details in \cite{AW70}.
A key step in the proof is to consider an arbitrary set $E\subset X$ and relate $\mu(E)$ to the number of $n$-cylinders intersecting $E$.
In extending the uniqueness result to sofic shifts (factors of SFTs), Weiss \cite{bW73} clarified the crucial role of what we refer to below as the ``lower Gibbs bound'' in carrying out this step. This is essentially the proof of uniqueness that we use in all the results in this survey.

The crucial difference between Theorem \ref{thm:bowen-shift} and the results of Parry, Adler, and Weiss is the construction of the MME using the specification property rather than eigendata of a matrix. This is due to Bowen, as is the further generalization to non-symbolic systems and equilibrium states for non-zero potentials \cite{rB75}. Thus we often refer informally to the proof below as ``Bowen's argument''.
\end{remark}

\subsection{The lower Gibbs bound as the mechanism for uniqueness}\label{sec:lower-gibbs}

It follows from the Shannon--McMillan--Breiman theorem that if $\mu$ is an ergodic shift-invariant measure, then for $\mu$-a.e.\ $x$ we have
\begin{equation}\label{eqn:smb}
- \frac 1n \log \mu[x_{[1,n]}] \to h_\mu(\sigma)
\text{ as }n\to\infty.
\end{equation}
This can be rewritten as
\begin{equation}\label{eqn:smb-2}
\frac 1n \log \Big( \frac{\mu[x_{[1,n]}]}{e^{-nh_\mu(\sigma)}}\Big) \to 0
\quad\text{for $\mu$-a.e.\ $x$}.
\end{equation}
In other words, for $\mu$-typical $x$, the measure $\mu[x_{[1,n]}]$ decays like $e^{-nh_\mu(\sigma)}$ in the sense that $\mu[x_{[1,n]}] / e^{-nh_\mu(\sigma)}$ is ``subexponential in $n$". The mechanism for uniqueness in the Parry-Adler-Weiss-Bowen argument is to produce an ergodic measure for which this subexponential growth is strengthened to uniform boundedness\footnote{We will encounter this general principle multiple times: many of our proofs rely on obtaining uniform bounds (away from $0$ and $\infty$) for quantities that \emph{a priori} can grow or decay subexponentially.} and applies for all $x$.

The next proposition makes this \emph{Gibbs property} precise and explain how uniqueness follows; then in \S\ref{sec:building-gibbs} we describe how to construct such a measure. The following argument appears in \cite[Lemma 2]{bW73} (see also \cite{AW67,AW70}); see  \cite{rB74} for a version that works in the nonsymbolic setting, which we will describe in \S\ref{sec:adapted} below.

\begin{proposition}\label{prop:uniqueness}
Let $X \subset A^\NN$ be a shift space and $\mu$ an ergodic $\sigma$-invariant measure on $X$. Suppose that there are $K,h>0$ such that for every $x\in X$ and $n\in\NN$, we have the \emph{Gibbs bounds}
\begin{equation}\label{eqn:gibbs-0}
K^{-1} e^{-nh} \leq \mu[x_{[1,n]}] \leq K e^{-nh}.
\end{equation}
Then $h = h_\mu(\sigma) = \htop(X,\sigma)$, and $\mu$ is the unique MME for $(X,\sigma)$.
\end{proposition}
\begin{proof}
First observe that by the Shannon--McMillan--Breiman theorem, the upper bound in \eqref{eqn:gibbs-0} gives $h_\mu(\sigma) \geq h$, while the lower bound gives $h_\mu(\sigma) \leq h$.\footnote{This requires ergodicity of $\mu$; one can also give a short argument directly from the definition of $h_\mu(\sigma)$ that does not need ergodicity.} Moreover, summing \eqref{eqn:gibbs-0} over all words in $\LLL_n$ gives $K^{-1} e^{nh} \leq \#\LLL_n \leq K e^{nh}$, so $\htop(X,\sigma) = h$.

The remainder of the proof is devoted to using the lower bound to show that 
\begin{equation}\label{eqn:hnu}
h_\nu(\sigma) < h = h_\mu(\sigma) \text{ for all } \nu \in \MMM_\sigma(X) \text{ with } \nu\neq \mu.
\end{equation}
This will show that $\mu$ is the unique MME.

Given $\nu \in \MMM_\sigma(X)$, the Lebesgue decomposition theorem gives $\nu=t\nu_1 + (1-t)\nu_2$ for some $t\in [0,1]$ and  $\nu_1,\nu_2\in \Mf(X)$ with $\nu_1 \perp \mu$ and $\nu_2 \ll \mu$.
By ergodicity, $\nu_2 = \mu$, and thus if $\nu \neq \mu$ we must have $t > 0$.
Since $h_\nu(\sigma) = th_{\nu_1}(\sigma) + (1-t)h_{\nu_2}(\sigma)$ and $h_{\nu_2}(\sigma) = h_\mu(\sigma) \leq h$, we see that to prove \eqref{eqn:hnu}, it suffices to prove that $h_\nu(\sigma) < h$ whenever $\nu\perp \mu$.

Writing $\alpha$ for the (generating) partition into $1$-cylinders, we see that for any $\nu \in \MMM_\sigma(X)$ we have
\begin{equation}\label{eqn:nhnu-0}
nh_\nu(\sigma) = h_\nu(\sigma^n) = h_\nu(\sigma^n,\alpha_0^{n-1}) \leq H_\nu(\alpha_0^{n-1}) = \sum_{w\in \LLL_n} -\nu[w] \log \nu[w].
\end{equation}
When $\nu\perp\mu$, there is a Borel set $D\subset X$ such that $\mu(D)=1$ and $\nu(D)=0$.  Since cylinders generate the $\sigma$-algebra, there is $\DDD \subset \LLL(X)$ such that $\mu(\DDD_n)\to 1$ and $\nu(\DDD_n) \to 0$, where $\mu(\DDD_n) := \mu \big(\bigcup_{w\in \DDD_n} [w]\big)$.  
We break the sum in \eqref{eqn:nhnu-0} into two pieces, one over $\DDD_n$ and one over $\DDD_n^c = \LLL_n \setminus \DDD_n$.  Observe that
\begin{align*}
\sum_{w\in \DDD_n} -\nu[w] \log\nu[w] 
&= \sum_{w\in \DDD_n} -\nu[w] \Big(\log \frac{\nu[w]}{\nu(\DDD_n)} + \log \nu(\DDD_n)\Big) \\
&= \Big( \nu(\DDD_n) \sum_{w\in \DDD_n} -\frac{\nu[w]}{\nu(\DDD_n)} \log \frac{\nu[w]}{\nu(\DDD_n)} \Big) - \nu(\DDD_n) \log \nu(\DDD_n) \\
&\leq (\nu(\DDD_n) \log \#\DDD_n) + 1,
\end{align*}
where the last line uses the fact that $\sum_{i=1}^k -p_i \log p_i \leq \log k$ whenever $p_i\geq 0$, $\sum p_i = 1$, as well as the fact that $-t\log t \leq 1$ for all $t\in [0,1]$.  A similar computation holds for $\DDD_n^c$, and together with \eqref{eqn:nhnu-0} this gives
\begin{equation}\label{eqn:nhnu2}
nh_\nu(\sigma) \leq 2 + \nu(\DDD_n) \log \#\DDD_n + \nu(\DDD_n^c) \log \#\DDD_n^c.
\end{equation}
Using \eqref{eqn:gibbs-0} and summing over $\DDD_n$ gives
\[
\mu(\DDD_n) = \sum_{w\in \DDD_n} \mu[w] \geq K^{-1} e^{-nh} \#\DDD_n
\quad\Rightarrow\quad
\#\DDD_n \leq K e^{nh} \mu(\DDD_n),
\]
and similarly for $\DDD_n^c$, so \eqref{eqn:nhnu2} gives
\begin{align*}
nh_\nu(\sigma) &\leq 2 + \nu(\DDD_n)\big(\log K + nh + \log \mu(\DDD_n)\big) 
 + \nu(\DDD_n^c)\big(\log K + nh + \log \mu(\DDD_n^c)\big) \\
&= 2 + \log K + nh + \nu(\DDD_n) \log \mu(\DDD_n) + \nu(\DDD_n^c) \log \mu(\DDD_n^c).
\end{align*}
Rewriting this as
\[
n(h_\nu(\sigma) - h) \leq 2 + \log K + \nu(\DDD_n) \log \mu(\DDD_n) + \nu(\DDD_n^c) \log \mu(\DDD_n^c),
\]
we see that the right-hand side goes to $-\infty$ as $n\to\infty$, since $\nu(\DDD_n) \to 0$ and $\mu(\DDD_n)\to 1$, so the left-hand side must be negative for large enough $n$, which implies that $h_\nu(\sigma) < h$ and completes the proof.
\end{proof}

\subsection{Building a Gibbs measure}\label{sec:building-gibbs}

Now the question becomes how to build an ergodic measure satisfying the lower Gibbs bound.  There is a standard construction of an MME for a shift space, which proceeds as follows: let $\nu_n$ be any measure on $X$ such that $\nu_n[w]=1/\#\LLL_n$ for every $w\in \LLL_n$, and then consider the measures
\begin{equation}\label{eqn:build}
\mu_n := \frac 1n \sum_{k=0}^{n-1} \sigma_*^k \nu_n = \frac 1n \sum_{k=0}^{n-1} \nu_n \circ \sigma^{-k}.
\end{equation}
A general argument (which appears in the proof of the variational principle, see for example \cite[Theorem 8.6]{pW82}) shows that any weak* limit point of the sequence $\mu_n$ is an MME.  If the shift space satisfies the specification property, one can prove more.

\begin{proposition}\label{prop:build-mme}
Let $(X,\sigma)$ be a shift space with the specification property, let $\mu_n$ be given by \eqref{eqn:build}, and suppose that $\mu_{n_j} \to \mu$ in the weak* topology.  Then $\mu$ is $\sigma$-invariant, ergodic, and there is $K\geq 1$ such that $\mu$ satisfies the following \emph{Gibbs property}:
\begin{equation}\label{eqn:gibbs}
K^{-1} e^{-n\htop(X)} \leq \mu[w] \leq K e^{-n\htop(X)}
\text{ for all } w\in \LLL_n.
\end{equation}
\end{proposition}
Combining Propositions \ref{prop:uniqueness} and \ref{prop:build-mme} shows that there is a unique MME $\mu$, which is the weak* limit of the sequence $\mu_n$ from \eqref{eqn:build}. Thus to prove Theorem \ref{thm:bowen-shift} it suffices to prove Proposition \ref{prop:build-mme}. We omit the full proof, and highlight only the most important part of the associated counting estimates.

\begin{lemma}\label{lem:counting}
Let $(X,\sigma)$ be a shift space with the specification property, with gap size $\tau$.  Then for every $n\in\NN$, we have
\begin{equation}\label{eqn:counting}
e^{n\htop(X)} \leq \#\LLL_n \leq Q e^{n\htop(X)},
\quad\text{where } Q=(\tau+1) e^{\tau \htop(X)}.
\end{equation}
\end{lemma}
\begin{proof}
For every $m,n\in \NN$, there is an injective map $\LLL_{m+n} \to \LLL_m \times \LLL_n$ defined by $w\mapsto (w_{[1,m]}, w_{[m+1,m+n]})$, so $\#\LLL_{m+n} \leq \#\LLL_m \#\LLL_n$.  Iterating this gives
\[
\#\LLL_{kn} \leq (\#\LLL_n)^k\quad\Rightarrow\quad
\frac 1{kn} \log \#\LLL_{kn} \leq \frac 1n \log \#\LLL_n,
\]
and sending $k\to\infty$ we get $\htop(X) \leq \frac 1n \log \#\LLL_n$ for all $n$, which proves the lower bound.  For the upper bound we observe that specification gives a map $\LLL_m \times \LLL_n \to \LLL_{m+n+\tau}$ defined by mapping $(v,w)$ to $vuwu'$, where $u=u(v,w)\in \LLL$ with $|u|\leq \tau$ is the `gluing word' provided by the specification property, and $u'$ is \emph{any} word of length $\tau-|u|$ that can legally follow $vuw$.  This map may not be injective because $w$ can appear in different positions, but each word in $\LLL_{m+n}$ can have at most $(\tau+1)$ preimages, since $v,w$ are completely determined by $vuwu'$ and the length of $u$.  This shows that
\[
\#\LLL_{m+n+\tau} \geq \frac 1{\tau+1} \#\LLL_m \#\LLL_n
\quad\Rightarrow\quad
\#\LLL_{k(n+\tau)} \geq \Big( \frac{\#\LLL_n}{\tau+1}\Big)^k.
\]
Taking logs and dividing by $k(n+\tau)$ gives
\[
\frac 1{k(n+\tau)} \#\LLL_{k(n+\tau)} \geq \frac 1{n+\tau} \big( \log \#\LLL_n - \log(\tau+1)\big).
\]
Sending $k\to\infty$ and rearranging gives $\log \#\LLL_n \leq \log(\tau+1) + (n+\tau) \htop(X)$.  Taking an exponential proves the upper bound.
\end{proof}

With Lemma \ref{lem:counting} in hand, the idea of Proposition \ref{prop:build-mme} is to first prove the bounds on $\mu[w]$ by estimating, for each $n\gg |w|$ and $k\in \{1,\dots, n-|w|\}$, the number of words $u\in \LLL_n$ for which $w$ appears in position $k$; see Figure \ref{fig:nun-w}.
By considering the subwords of $u$ lying before and after $w$, one sees that there are at most $(\#\LLL_k)(\#\LLL_{n-k-|w|})$ such words, as in the proof of Lemma \ref{lem:counting}, and thus the bounds from that lemma give
\[
\nu_n(\sigma^{-k}[w]) \leq \frac{(\#\LLL_k)(\#\LLL_{n-k-|w|})}{\#\LLL_n}
\leq \frac{Q e^{k\htop(X)} Q e^{(n-k-|w|)\htop(X)}}{e^{n\htop(X)}}
= Q^2 e^{-|w| \htop(X,\sigma)};
\]
averaging over $k$ gives the upper Gibbs bound, and the lower Gibbs bound follows from a similar estimate that uses the specification property.

\begin{figure}[htbp]
\begin{tikzpicture}
\draw (0,0)--(6,0);
\foreach \x in {0,2,3,6}
{ \draw (\x,-.1)--(\x,.1); }
\begin{scope}[decoration={brace,amplitude=10,raise=4pt}]
\draw[decorate,blue] (6,0)--(0,0);
\draw[decorate,red] (0,0)--(2,0);
\draw[decorate,red] (3,0)--(6,0); \end{scope}
\node at (2.5,0)[above] {$w$};
\node[blue] at (3,-.8){$\#\LLL_n$};
\node[red] at (1,.8){$\#\LLL_k$};
\node[red] at (4.5,.8){$\#\LLL_{n-k-|w|}$};
\end{tikzpicture}
\caption{Estimating $\nu_n(\sigma^{-k}[w])$.}
\label{fig:nun-w}
\end{figure}
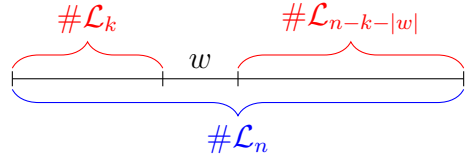

Next, one can use similar arguments to produce $c>0$ such that, for each pair of words $v,w$, there are arbitrarily large $j\in\NN$ such that $\mu([v] \cap \sigma^{-j}[w]) \geq c \mu[v] \mu[w]$; this is once again done by counting the number of long words that have $v,w$ in the appropriate positions.  

Since any measurable sets $V$ and $W$ can be approximated by unions of cylinders, one can use this to prove that
$\ulim_n \mu(V\cap \sigma^{-n}W) \geq c \mu(V) \mu(W)$.  Considering the case when $V=W$ is $\sigma$-invariant demonstrates that $\mu$ is ergodic.

\section{Relaxing specification: decompositions of the language}\label{sec:relax-spec}

\subsection{Decompositions}\label{sec:decomp}

There are many shift spaces that can be shown to have a unique MME despite not having the specification property; see \S\ref{sec:beta} below for the example that motivated the present work. We want to consider shift spaces for which the specification property holds if we restrict our attention to ``good words'', and will see that the uniqueness result in Theorem \ref{thm:bowen-shift} can be extended to this setting provided the collection of ``good words'' is ``large enough'' in an appropriate sense.

To make this more precise, let $X$ be a shift space on a finite alphabet, and $\LLL$ its language.  We consider the following more general version of \eqref{eqn:spec}.

\begin{definition}\label{def:symb-spec}
A collection of words $\GGG\subset \LLL$ has \emph{specification} if there exists $\tau\in \NN$ such that for every finite set of words $w^1,\dots, w^k \in \GGG$, there are $u^1,\dots, u^{k-1} \in \LLL$ with $|u^i| \leq \tau$ such that $w^1 u^1 w^2 u^2 \cdots u^{k-1} w^k \in \LLL$.  
\end{definition}

The only difference between this definition and \eqref{eqn:spec} is that here we only require the gluing property to hold for words in $\GGG$, not for all words.

\begin{remark}
In particular, $\GGG$ has specification if there is $\tau\in\NN$ such that for every $v,w\in \GGG$, there is $u\in \LLL$ with $|u|\leq \tau$ and $vuw\in\GGG$, because iterating this property gives the one stated above.  The property above, which is sufficient for our uniqueness results, is a priori more general because the concatenated word is not required to lie in $\GGG$.
\end{remark}

Now we need a way to say that a collection $\GGG$ on which specification holds is sufficiently large.

\begin{definition}\label{def:decomp}
A \emph{decomposition} of the language $\LLL$ consists of three collections of words $\Cp,\GGG,\Cs \subset \LLL$ with the property that
\begin{equation}\label{eqn:decomp}
\text{for every } w\in \LLL,\text{ there are } u^p\in \Cp, v\in\GGG, u^s\in \Cs \text{ such that } w=u^p v u^s.
\end{equation}
Given a decomposition of $\LLL$, we also consider for each $M\in\NN$ the collection of words
\begin{equation}\label{eqn:GM}
\GGG^M := \{ u^p v u^s \in \LLL : u^p \in \Cp, v\in \GGG, u^s \in \Cs, |u^p|, |u^s|\leq M \}.
\end{equation}
\end{definition}

If each $\GGG^M$ has specification, then the set $\Cp \cup\Cs$ can be thought of as the set of \emph{obstructions} to the specification property.

\begin{definition}
The \emph{entropy} of a collection of words $\CCC\subset \LLL$ is
\begin{equation}\label{eqn:entropy-C}
h(\CCC) = \ulim_{n\to\infty} \frac 1n \log \#\CCC_n.
\end{equation}
\end{definition}

\begin{theorem}[Uniqueness using a decomposition \cite{CT12}]\label{thm:symbolic}
Let $X$ be a shift space on a finite alphabet, and suppose that the language $\LLL$ of $X$ admits a decomposition $\Cp \GGG \Cs$ such that
\begin{enumerate}[label=\upshape{(\Roman{*})}]
\item\label{GM-spec} every collection $\GGG^M$ has specification, and
\item\label{hC<h} $h(\Cp \cup \Cs) < h(X)$.
\end{enumerate}
Then $(X,\sigma)$ has a unique MME $\mu$.  
\end{theorem}

\begin{remark}\label{rmk:pesin-sets}
Note that $\LLL = \bigcup_{M\in \NN} \GGG^M$; the sets $\GGG^M$ play a similar role to the regular level sets that appear in Pesin theory.\footnote{Since $\GGG^M$ corresponds to a collection of orbit segments rather than a subset of the space, the most accurate analogy might be to think of $\GGG^M$ as corresponding to orbit segments that start and end in a given regular level set.}
The gap size $\tau$ appearing in the specification property for $\GGG^M$ is allowed to depend on $M$, just as the constants appearing in the definition of hyperbolicity are allowed to depend on which regular level set a point lies in.  Similarly, for the unique MME $\mu$ one can prove that $\lim_{M\to\infty} \mu(\GGG^M) = 1$, which mirrors a standard result for hyperbolic measures and Pesin sets.
\end{remark}

\begin{remark}
In fact we do not quite need \emph{every} $w\in \LLL$ to admit a decomposition as in \eqref{eqn:decomp}.  It is enough to have $\Cp,\GGG,\Cs\subset \LLL$ such that $h(\LLL\setminus(\Cp\GGG\Cs)) < h(X)$, in addition to the conditions above \cite{vC18}.
\end{remark}

We outline the proof of Theorem \ref{thm:symbolic}.  The idea is to mimic Bowen's proof using Propositions \ref{prop:uniqueness} and \ref{prop:build-mme} by completing the following steps.
\begin{enumerate}
\item Prove uniform counting bounds as in Lemma \ref{lem:counting}.
\item Use these to establish the following \emph{non-uniform} Gibbs property for any limit point $\mu$ of the sequence of measures in \eqref{eqn:build}: there are constants $K,K_M\geq 1$ such that
\begin{equation}\label{eqn:nuh-gibbs}
K_M^{-1} e^{-|w|\htop(X)} \leq \mu[w] \leq K e^{-|w|\htop(X)} \text{ for all } M\in\NN \text{  and }w\in \GGG^M.
\end{equation}
We emphasize that the Gibbs property is non-uniform in the sense that the lower Gibbs constant depends on $M$.\footnote{The constant $K_M$ increases exponentially with the transition time in the specification property for $\GGG^M$, so we do not expect any explicit relationship between $M$ and $K_M$ in general. Examples of $S$-gap shifts (see Remark \ref{Sgap}) can be easily constructed to make the constants $K_M^{-1}$ decay fast.} The upper bound that we will obtain from our hypotheses is uniform in $M$. On a fixed $\GGG^M$, we have uniform Gibbs estimates.
\item Give a similar argument for ergodicity, and then prove that the non-uniform lower Gibbs bound in \eqref{eqn:nuh-gibbs} still gives uniqueness as in Proposition \ref{prop:uniqueness}.
\end{enumerate}
Once the uniform counting bounds are established, the proof of \eqref{eqn:nuh-gibbs} follows the same approach as before.  We do not discuss the third step at this level of generality except to emphasize that it follows the approach  given in Proposition \ref{prop:uniqueness}.

For the counting bounds in the first step, we start by observing that the bound $\#\LLL_n \geq e^{n\htop(X)}$ did not require any hypotheses on the symbolic space $X$ and thus continues to hold.  The argument for the upper bound in Lemma \ref{lem:counting} can be easily adapted to show that there is a constant $Q$ such that $\#\GGG_n \leq Q e^{n\htop(X)}$ for all $n$.  Then the desired upper bound for $\#\LLL_n$ is a consequence of the following.

\begin{lemma}\label{lem:most-in-G}
For any $r\in (0,1)$, there is $M$ such that $\#\GGG^M_n \geq r \#\LLL_n$ for all $n$.
\end{lemma}
\begin{proof}
Let $a_i = \#(\Cp_i \cup \Cs_i) e^{-i\htop(X)}$, so that in particular $\sum a_i < \infty$ by \ref{hC<h}.  Since any $w\in \LLL_n$ can be written as $w=u^p v u^s$ for some $u\in \Cp_i$, $v\in \GGG_j$, and $w\in \Cs_k$ with $i+j+k=n$, we have
\[
\#\LLL_n \leq \#\GGG^M_n + \sum_{\substack{i+j+k=n \\ \max(i,k) > M}} (\#\Cp_i) (\#\GGG_j) (\#\Cs_k)
\leq \#\GGG_n^M + \sum_{\substack{i+j+k=n \\ \max(i,k) > M}} a_i a_k Q e^{n\htop(X)},
\]
where the second inequality uses the upper bound $\#\GGG_j \leq Q e^{j\htop(X)}$.
Since $\sum a_i < \infty$, there is $M$ such that 
\[
\sum_{\substack{i+j+k=n \\ \max(i,k) > M}} a_i a_k Q e^{n\htop(X)} <(1-r) e^{n\htop(X)} \leq (1-r) \#\LLL_n,
\]
where the second inequality uses the lower bound $\#\LLL_n \geq e^{n\htop(X)}$.  Combining these estimates gives $\#\LLL_n \leq \#\GGG_n^M + (1-r)\#\LLL_n$, which proves the lemma.
\end{proof}

The same specification argument that gives the upper bound on $\#\GGG_n$ gives a corresponding upper bound on $\GGG_n^M$ (with a different constant), and thus we deduce the following consequence of Lemma \ref{lem:most-in-G}.

\begin{corollary}
There are constants $a,A>0$ and $M\in\NN$ such that
\[
e^{n\htop(X)} \leq \#\LLL_n \leq A e^{n\htop(X)}
\quad\text{and}\quad
\#\GGG_n^M \geq a e^{n\htop(X)} \text{ for all } n\in\NN.
\]
\end{corollary}

\begin{remark}
In fact, the proof of Lemma \ref{lem:most-in-G} can easily be adapted to show a stronger result: given any $\gamma>0$ and $r\in (0,1)$, there is $M$ such that if $\DDD_n\subset \LLL_n$ has $\#\DDD_n \geq \gamma e^{n\htop(X)}$, then $\#(\DDD_n \cap \GGG^M_n) \geq r \#\DDD_n$.  These types of estimates are what lie behind the claim in Remark \ref{rmk:pesin-sets} that the (non-uniform) Gibbs property implies $\mu(\GGG^M) \to 1$ as $M\to\infty$.
\end{remark}

\subsection{An example: beta shifts}\label{sec:beta}

Given a real number $\beta>1$, the corresponding $\beta$-transformation $f\colon [0,1) \to [0,1)$ is $f(x) = \beta x \pmod 1$.  Let $A = \{0,1,\dots, \lceil \beta \rceil - 1\}$; then every $x\in [0,1)$ admits a coding $y = \pi(x) \in A^\NN$ defined by $y_n = \lfloor \beta f^{n-1}(x) \rfloor$, and we have $\pi\circ f = \sigma\circ \pi$, where $\sigma\colon A^\NN \to A^\NN$ is the left shift.  Observe that $\pi(x)_n = a$ if and only if $f^{n-1}(x) \in I_a$, where the intervals $I_a$ are as shown in Figure \ref{fig:beta}.\footnote{Formally,
$I_a = \{ x\in [0,1) : \lfloor\beta x\rfloor = a\}$, so $I_a = [\frac a\beta, 1)$ if $a= \lceil\beta\rceil - 1$, and $[\frac a\beta, \frac{a+1}\beta)$ otherwise.}  
Given $n\in\NN$ and $w\in A^n$, let
\[
I(w) := \bigcap_{k=1}^{n} f^{-(k-1)}(I_{w_k})
\]
be the interval in $[0,1)$ containing all points $x$ for which the first $n$ iterates are coded by $w$.  The figure shows an example for which $f^n(I(w))$ is not the whole interval $[0,1)$; it is worth checking some other examples and seeing if you can tell for which words $f^n(I(w))$ \emph{is} equal to the whole interval.  Observe that if $\beta$ is an integer then this is true for every word.

\begin{figure}[htbp]
\begin{tikzpicture}[scale=4]
\def\b{2.525};
\draw (0,0) rectangle (1,1);
\draw[blue,dotted,thick] (0,0) -- (1,1);
\draw (0,0) -- ({1/\b},1);
\draw ({1/\b},0) -- ({2/\b},1);
\draw ({2/\b},0) -- (1,\b-2);
\foreach \i in {0,1,2}
{ \draw ({\i/\b},-.05) -- +(0,.1); }
\node at ({1/(2*\b)},0)[below] {$I_0$};
\node at ({3/(2*\b)},0)[below] {$I_1$};
\node at ({(1+2/\b)/2},0)[below] {$I_2$};
\pgfmathsetmacro{\x}{1};
\pgfmathfrac{\b*\x};
\pgfmathsetmacro{\y}{\pgfmathresult};
\foreach \i in {1,...,6}
{
\draw[red] (\x,\y) -- (\y,\y);
\xdef\x{\y};
\pgfmathfrac{\b*\x};
\pgfmathsetmacro{\yy}{\pgfmathresult};
\xdef\y{\yy};
\draw[red] (\x,\x) -- (\x,\y);
}
\end{tikzpicture}
\hspace{6em}
\begin{tikzpicture}[scale=4]
\def\b{2.525};
\draw (0,0) rectangle (1,1);
\draw[blue,dotted,thick] (0,0) -- (1,1);
\draw (0,0) -- ({1/\b},1);
\draw ({1/\b},0) -- ({2/\b},1);
\draw ({2/\b},0) -- (1,\b-2);
\foreach \i in {0,1,2}
{ \draw ({\i/\b},-.05) -- +(0,.1); }
\pgfmathsetmacro{\x}{1/((\b)^2) + 2/\b};
\draw[red] ({1/\b},0) -- ({1/\b},{1/\b})--(\x,{1/\b}) -- (\x,0);
\draw[red,ultra thick] (\x,0)--(1,0);
\node[red] at ({(1+\x)/2},0)[below] {$I(21)$};
\draw[red] (1,\b-2)--(\b-2,\b-2)--(\b-2,0);
\draw[red,ultra thick] ({1/\b},0)--({\b-2},0)
node[midway,below] {$f(I(21))$};
\pgfmathsetmacro{\x}{(\b)^2 - 2*\b - 1};
\draw[red] (\b-2,\x)--(\x,\x)--(\x,0);
\draw[red,ultra thick] (\x,0) -- (0,0)
node[midway,below left]{$f^2(I(21))$};
\end{tikzpicture}
\caption{Coding a $\beta$-transformation.}
\label{fig:beta}
\end{figure}
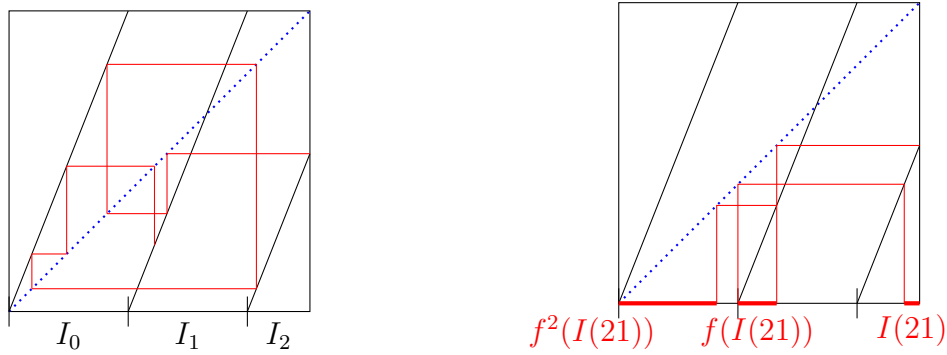

\begin{definition}
The \emph{$\beta$-shift} $X_\beta$ is the closure of the image of $\pi$, and is $\sigma$-invariant.  Equivalently, $X_\beta$ is the shift space whose language $\LLL$ is the set of all $w\in A^*$ such that $I(w) \neq \emptyset$; thus $y\in A^\NN$ is in $X_\beta$ if and only if $I(y_1 \cdots y_n) \neq\emptyset$ for all $n\in\NN$.
\end{definition}

For further background on the $\beta$-shifts, see \cite{aR57,wP60,fB89}.  We summarize the properties relevant for our purposes.

Write $\preceq$ for the lexicographic order on $A^\NN$ and observe that $\pi$ is order-preserving.  Let $\zz = \lim_{x\nearrow 1} \pi(x)$ denote the supremum of $X_\beta$ in this ordering.  It will be convenient to extend $\preceq$ to $A^*$, writing $v\preceq w$ if for $n=\min(|v|,|w|)$ we have $v_{[1,n]} \preceq w_{[1,n]}$.

\begin{remark}
Observe that on $A^* \cup A^\NN$, $\preceq$ is only a pre-order, because there are $v\neq w$ such that $v\preceq w$ and $w\preceq v$; this occurs whenever one of $v,w$ is a prefix of the other.
\end{remark}

The $\beta$-shift can be described in terms of the lexicographic ordering, or in terms of the following countable-state graph:
\begin{itemize}
\item the vertex set is $\NN_0 = \{0,1,2,3,\dots\}$;
\item the vertex $n$ has $1 + \zz_{n+1}$ outgoing edges, labeled with $\{0,1,\dots, \zz_{n+1}\}$; the edge labeled $\zz_{n+1}$ goes to $n+1$, and the rest go to the `base' vertex $0$.
\end{itemize}
Figure \ref{fig:beta-graph} shows (part of) the graph when $\zz = 2102001\dots$, as in Figure \ref{fig:beta}.

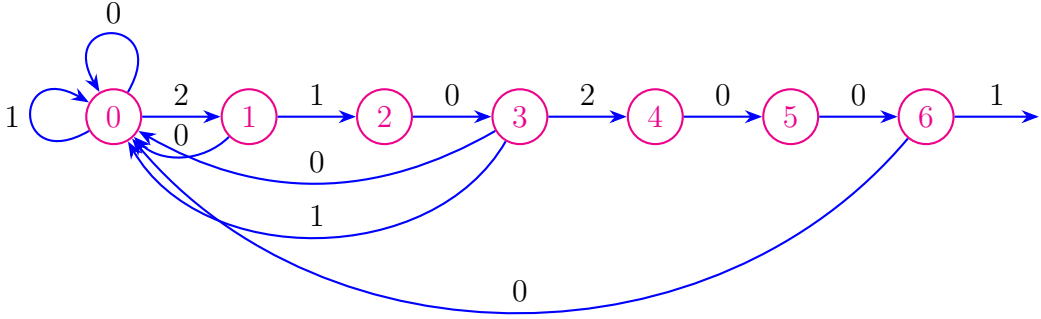
\begin{figure}[htbp]
\begin{tikzpicture}[scale=.6]
\begin{scope}[every node/.style={circle,thick,draw,color=magenta}]
	\def\x{3};
	\foreach \i in {0,1,...,6}
{	\node (\i) at ({\i*\x},0) {\i}; }
\end{scope}
\begin{scope}[>={Stealth},
              every node/.style={above},
              every edge/.style={draw=blue,thick}]
    \path [->] (0) edge node {$2$} (1);
    \path [->] (1) edge node {$1$} (2);
    \path [->] (2) edge node {$0$} (3);
    \path [->] (3) edge node {$2$} (4);
    \path [->] (4) edge node {$0$} (5);
    \path [->] (5) edge node {$0$} (6);
    \path [->] (6) edge node {$1$} (20.5,0);
    \path [->] (0) edge[out=60,in=120,looseness=8] node{$0$} (0); 
    \path [->] (0) edge[out=210,in=150,looseness=8] node[left]{$1$} (0); 
    \path [->] (1) edge[bend left=45] node {$0$} (0);
    \path [->] (3) edge[bend left=60] node {$1$} (0);  
    \path [->] (3) edge[bend left=30] node {$0$} (0);
	\path [->] (6) edge[bend left=50] node {$0$} (0);
\end{scope}
\end{tikzpicture}
\caption{A graph representation of $X_\beta$.}
\label{fig:beta-graph}
\end{figure}

\begin{proposition}
Given $n\in\NN$ and $w\in A^n$, the following are equivalent.
\begin{enumerate}
\item $I(w) \neq \emptyset$ (which is equivalent to $w\in \LLL(X_\beta)$ by definition).
\item $w_{[j,n]} \preceq \zz$ for every $1\leq j\leq n$.
\item $w$ labels the edges of a path on the graph that starts at the base vertex $0$.
\end{enumerate}
\end{proposition}
\begin{proof}[Idea of proof]
Using induction, check that the following are equivalent for every $n\in\NN$, $0\leq k\leq n$, and $w\in A^n$.
\begin{enumerate}
\item $f^n(I(w)) = f^k(I(\zz_{[1,k]})$, where we write $I(\zz_{[1,0]}) := [0,1)$.
\item $w_{[j,n]} \preceq \zz$ for every $1\leq j\leq n$, and $k$ is maximal such that $w_{[n-k+1,n]} = \zz_{[1,k]}$.
\item $w$ labels the edges of a path on the graph that starts at the base vertex $0$ and ends at the vertex $k$.\qedhere
\end{enumerate}
\end{proof}

\begin{corollary}
Given $x\in A^\NN$, the following are equivalent.
\begin{enumerate}
\item $x\in X_\beta$.
\item $\sigma^n(x) \preceq \zz$ for every $n$.
\item $x$ labels the edges of an infinite path of the graph starting at the vertex $0$.
\end{enumerate}
\end{corollary}

\begin{exercise}
Prove that $X_\beta$ has the specification property if and only if $\zz$ does not contain arbitrarily long strings of $0$s.
\end{exercise}

In fact, Schmeling showed \cite{jS97} that for Lebesgue-a.e.\ $\beta>1$, the $\beta$-shift $X_\beta$ does \emph{not} have the specification property.  Nevertheless, every $\beta$-shift has a unique MME.  This was originally proved by Hofbauer \cite{fH78} and Walters \cite{pW78} using techniques not based on specification.  Theorem \ref{thm:symbolic} gives an alternate proof: writing $\GGG$ for the set of words that label a path starting \emph{and ending} at the base vertex, and $\Cs$ for the set or words that label a path starting at the base vertex \emph{and never returning to it}, one quickly deduces the following.
\begin{itemize}
\item $\GGG \Cs$ is a decomposition of $\LLL$.
\item $\GGG^M$ is the set of words labeling a path starting at the base vertex and ending somewhere in the first $M$ vertices; writing $\tau$ for the maximum graph distance from such a vertex to the base vertex, $\GGG^M$ has specification with gap size $\tau$.
\item $\#\Cs_n = 1$ for every $n$, and thus $h(\Cs)=0 < \htop(X_\beta) = \log \beta$.
\end{itemize}
This verifies the conditions of Theorem \ref{thm:symbolic} and thus provides another proof of uniqueness of the MME.

\begin{remark}
Because the earlier proofs of uniqueness did not pass to subshift factors of $\beta$-shifts, it was for several years an open problem (posed by Klaus Thomsen) whether such factors still had a unique MME.  The inclusion of this problem in Mike Boyle's article ``Open problems in symbolic dynamics'' \cite{mB08} was our original motivation for studying uniqueness using non-uniform versions of the specification property, which led us to formulate the conditions in Theorem \ref{thm:symbolic};  these can be shown to pass to factors, providing a positive answer to Thomsen's question \cite{CT12}.
\end{remark}

\begin{remark} \label{Sgap}
Theorem \ref{thm:symbolic} can be applied to other symbolic examples as well, including $S$-gap shifts \cite{CT12}. The $S$-gap shifts are a family of subshifts of $\{0, 1\}^{\mathbb Z}$ defined by the property that the number of $0$'s that appear between any two $1$'s is an element of a prescribed set $S \subset \mathbb Z$. A specific example is the prime gap shift, where $S$ is taken to be the prime numbers. The theorem also admits an extension to equilibrium states for nonzero potential functions along the lines described in \S\ref{sec:eq-st} below, which has been applied to $\beta$-shifts \cite{CT13}, $S$-gap shifts \cite{CTY}, shifts of quasi-finite type \cite{vC18}, and $\alpha$-$\beta$ shifts (which code $x\mapsto \alpha + \beta x \pmod 1$) \cite{CLR}.
\end{remark}

\subsection{Periodic points}\label{sec:periodic-points}

It is often the case that one can prove a stronger version of specification, for example, when $X$ is a mixing SFT.

\begin{definition}
Say that $\GGG\subset \LLL$ has \emph{periodic strong specification} if there exists $\tau\in\NN$ such that for all $w^1,\dots, w^k\in \GGG$, there are $u^1,\dots, u^k \in \LLL_\tau$ such that $v := w^1 u^1 \cdots w^k u^k \in \LLL$, and moreover $x=vvvvv\dots \in X$.
\end{definition}

There are two strengthenings of specification, in the sense of \eqref{eqn:spec}, here: first, we assume that the gap size is equal to $\tau$, not just $\leq\tau$, and second, we assume that the ``glued word'' can be extended periodically after adding $\tau$ more symbols.

If we replace specification in Theorem \ref{thm:symbolic} with periodic strong specification for each $\GGG^M$, then the counting estimates in Lemma \ref{lem:counting} immediately lead to the following estimates on the number of periodic points: writing $\Per_n = \{x \in X : \sigma^n x = x\}$, we have
\begin{equation}\label{eqn:periodic}
C^{-1} e^{n\htop(X)}
\leq \#\Per_n
\leq C e^{n\htop(X)}.
\end{equation}
Using this fact and the construction of the unique MME given just before Proposition \ref{prop:build-mme}, one can also conclude that the unique MME $\mu$ is the limiting distribution of periodic orbits in the following sense: 
\begin{equation}\label{eqn:Pern}
\frac 1{\#\Per_n} \sum_{x\in \Per_n} \delta_x \xrightarrow{\text{weak*}} \mu
\text{ as } n\to\infty.
\end{equation}
This argument holds true in the classical Theorem \ref{thm:bowen-shift}, and for $\beta$-shifts. It also extends beyond the symbolic setting, and a natural analogue of the argument holds for regular closed geodesics on rank one non-positive curvature manifolds.

\section{Beyond shift spaces: expansivity in Bowen's argument} \label{sec:MMEhomeos}

Now we move to the non-symbolic setting and describe how Bowen's approach works for a continuous map on a compact metric space. In particular, his assumptions apply to and were inspired by the case when $X$ is a transitive locally maximal hyperbolic set for a diffeomorphism $f$.  First we recall some basic definitions.

\subsection{Topological entropy}\label{sec:topological-entropy}

\begin{definition}
Given $n\in\NN$, the $n$th \emph{dynamical metric} on $X$ is
\begin{equation}\label{eqn:dn}
d_n(x,y) := \max \{d(f^k x, f^k y) : 0\leq k < n\}.
\end{equation}
The \emph{Bowen ball of order $n$ and radius $\eps>0$ centered at $x\in X$} is
\begin{equation}\label{eqn:Bn}
B_n(x,\eps) := \{ y\in X : d_n(x,y) < \eps\}.
\end{equation}
A set $E\subset X$ is called \emph{$(n,\eps)$-separated} if $d_n(x,y) > \eps$ for all $x,y\in E$ with $x\neq y$; equivalently, if $y\notin B_n(x,\eps)$ for all such $x,y$.
\end{definition}

We define entropy in a more general way than is standard, reflecting our focus on the space of finite-length orbit segments $X\times \NN$ as the relevant object of study; this replaces the language $\LLL$ that we used in the symbolic setting.  We interpret $(x,n)\in X\times \NN$ as representing the orbit segment $(x,fx,f^2x,\dots, f^{n-1}x)$. Then the analogy is that a cylinder $[w]$ for a word in the language corresponds to a Bowen ball $B_n(x,\eps)$ associated to an orbit segment $(x,n) \in X\times \NN$. Given a collection of orbit segments $\DDD\subset X\times \NN$, for each $n\in\NN$ we write
\begin{equation}\label{eqn:Dn}
\DDD_n := \{x\in X : (x,n) \in \DDD\}
\end{equation}
for the collection of points that begin a length-$n$ orbit segment in $\DDD$.

\begin{definition}[Topological entropy]\label{def:entropy}
Given a collection of orbit segments $\DDD \subset X\times \NN$, for each $\eps>0$ and $n\in\NN$ we write
\begin{equation}\label{eqn:Lambda}
\Lambda(\DDD,\eps,n)
:= \max \{ \# E : E\subset \DDD_n \text{ is $(n,\eps)$-separated}\}.
\end{equation}
The \emph{entropy} of $\DDD$ at scale $\eps>0$ is
\begin{equation}\label{eqn:hDe}
h(\DDD,\eps) := \ulim_{n\to\infty} \frac 1n \log \Lambda(\DDD,\epsilon,n),
\end{equation}
and the entropy of $\DDD$ is
\begin{equation}\label{eqn:hD}
h(\DDD) := \lim_{\eps\to 0} h(\DDD,\eps).
\end{equation}
When $\DDD = Y\times \NN$ for some $Y\subset X$, 
we write $\Lambda(Y,\eps,n) = \Lambda(Y\times \NN,\eps,n)$, $\htop(Y,\eps) = h(Y\times \NN,\eps)$ and $\htop(Y) = \lim_{\eps \to 0} \htop(Y,\eps)$.
In particular, when $\DDD = X\times \NN$ we write
$\htop(X,f) = \htop(X) = h(X\times \NN)$ for the \emph{topological entropy} of $f\colon X\to X$.
\end{definition}

When different orbit segments in $\DDD$ are given weights according to their ergodic sum w.r.t.\ a given potential $\ph$, we obtain a notion of \emph{topological pressure}, which we will discuss in \S\ref{sec:eq-st}.

\begin{theorem}[Variational principle]
Let $X$ be a compact metric space and $f\colon X\to X$ a continuous map.  Then
\begin{equation}\label{eqn:vp}
\htop(X,f) = \sup_{\mu\in \Mf(X)} h_\mu(f).
\end{equation}
\end{theorem}

The following construction forms one half of the proof of the variational principle.

\begin{proposition}[Building a measure of almost maximal entropy]\label{prop:build}
With $X,f$ as above, fix $\eps>0$, and for each $n\in\NN$, let $E_n\subset X$ be an $(n,\eps)$-separated set.
Consider the Borel probability measures
\begin{equation}\label{eqn:mun}
\nu_n := \frac 1{\#E_n} \sum_{x\in E_n} \delta_x,
\quad
\mu_n := \frac 1n \sum_{k=0}^{n-1} f_*^k \nu_n
= \frac 1n \sum_{k=0}^{n-1} \nu_n \circ f^{-k}.
\end{equation}
Let $\mu_{n_j}$ be any subsequence that converges in the weak*-topology to a limiting measure $\mu$.  Then $\mu\in \Mf(X)$ and
\begin{equation}\label{eqn:entropy}
h_\mu(f)
\geq \ulim_{j\to\infty} \frac 1{n_j} \log \#E_{n_j}.
\end{equation} 
In particular, for every $\delta>0$ there exists $\mu\in \Mf(X)$ such that $h_\mu(f)\geq \htop(X,f,\delta)$.
\end{proposition}
\begin{proof}
See \cite[Theorem 8.6]{pW82}.
\end{proof}

\begin{corollary}\label{cor:existence}
Let $X,f$ be as above, and suppose that there is $\delta>0$ such that $\htop(X,f,\delta) = \htop(X,f)$. 
Then there exists a measure of maximal entropy for $(X,f)$.  Indeed, given any sequence $\{E_n\subset X\}_{n=1}^\infty$ of maximal $(n,\delta)$-separated sets, every weak*-limit point of the sequence $\mu_n$ from \eqref{eqn:mun} is an MME.
\end{corollary}

In our applications, it will often be relatively easy to verify that $\htop(X,f,\delta)=\htop(X,f)$ for some $\delta>0$, and so Corollary \ref{cor:existence} establishes existence of a measure of maximal entropy.  Thus the real challenge is to prove uniqueness, and this will be our focus.

\subsection{Expansivity}

In Bowen's general result, the assumption that $X$ is a shift space is replaced by the following condition.

\begin{definition}[Expansivity]
Given $x\in X$ and $\eps>0$, let
\begin{equation}\label{eqn:Gamma}
\Gamma_\eps^+(x) := \{ y \in X : d(f^n y, f^n x) < \eps \text{ for all } n\geq 0 \}
= \bigcap_{n\in\NN} B_n(x,\eps)
\end{equation}
be the forward infinite Bowen ball.
If $f$ is invertible, let
\begin{equation}\label{eqn:Gamma-1}
\Gamma_\eps^-(x) := \{ y \in X : d(f^n y, f^n x) < \eps \text{ for all } n\geq 0 \}
\end{equation}
be the backward infinite Bowen ball, and let
\begin{equation}\label{eqn:Gamma-2}
\Gamma_\eps(x) := \Gamma_\eps^+(x) \cap \Gamma_\eps^-(x) = \{ y \in X : d(f^n y, f^n x) < \eps \text{ for all } n\in \ZZ \}
\end{equation}
be the bi-infinite Bowen ball.
The system $(X,f)$ is \emph{positively expansive at scale $\eps>0$} if $\Gamma_\eps^+(x) = \{x\}$ for all $x\in X$, and \emph{(two-sided) expansive at scale $\eps>0$} if $\Gamma_\eps(x) = \{x\}$.  The system is \emph{(positively) expansive} if there exists $\eps>0$ such that it is (positively) expansive at scale $\eps$.
\end{definition}

It is an easy exercise to check that one-sided shift spaces are positively expansive.
A system $(X, f)$ is \emph{uniformly expanding} if there are $\epsilon,\lambda>0$ such that $d(fy, fx) \geq e^{\lambda} d(y,x)$ whenever $x,y\in X$ have $d(x,y) < \epsilon$. Iterating this property gives $\diam B_n(x,\epsilon) \leq \epsilon e^{-\lambda n}$ for all $n$, and thus $\Gamma_\epsilon^+(x) = \{x\}$, so $(X,f)$ is positively expansive.

Two-sided shift spaces can easily be checked to be (two-sided) expansive, and we also have the following.

\begin{proposition}\label{prop:uh-exp}
If $X$ is a hyperbolic set for a diffeomorphism $f$, then $(X,f)$ is expansive.
\end{proposition}
\begin{proof}[Sketch of proof]
Choose $\eps>0$ small enough that given any $x,y\in X$ with $d(x,y)<\eps$, the local leaves $W^s(x)$ and $W^u(y)$ intersect in a unique point $[x,y]$ (we do not require that this point is in $X$).  Write
\[
d^u(x,y) = d(x,[x,y])
\quad\text{and}\quad
d^s(x,y) = d(y,[x,y]).
\]
Passing to an adapted metric if necessary, hyperbolicity gives $\lambda>0$ such that 
\begin{gather}
\label{eqn:du}
d^u(f^nx,f^ny) \geq e^{\lambda n} d^u(x,y) \text{ if } d(f^kx,f^ky) < \eps \text{ for all } 0\leq k\leq n, \\
\label{eqn:ds}
d^s(f^{-n}x,f^{-n}y) \geq e^{\lambda n} d^s(x,y) \text{ if } d(f^{-k}x,f^{-k}y) < \eps \text{ for all } 0\leq k\leq n.
\end{gather}
In particular, if $y\in \Gamma_\eps(x)$ then $d^u(f^nx,f^ny)$ is uniformly bounded for all $n$, so $d^u(x,y) = 0$, and similarly for $d^s$, which implies that $x=[x,y]=y$.
\end{proof}

One important consequence of expansivity is the following.

\begin{proposition}\label{prop:appeared}
If $(X,f)$ is expansive at scale $\eps$, then $\htop(X,f,\eps) = \htop(X,f)$.
\end{proposition}
\begin{proof}[Two proof ideas]
We outline two proofs in the positively expansive case.

One argument uses a compactness argument to show that for every $0 < \delta < \eps$, there is $N\in\NN$ such that $B_N(x,\eps) \subset B(x,\delta)$ for all $x\in X$.  This implies that
$B_{n+N}(x,\eps) \subset B_n(x,\delta)$ for all $x$, and then one can show that the definition of topological entropy via $(n,\eps)$-separated sets gives the same value at $\delta$ as at $\eps$.

Another method, which is better for our purposes, is to observe that since $\eps$-expansivity gives $\bigcap_n B_n(x,\eps) = \{x\}$ for all $x$, one can easily show that for every $\nu\in\Mf(X)$, we have:
\begin{multline}\label{eqn:generating}
\text{if $\beta$ is a partition with $d_n$-diameter $<\eps$, then $\beta$ is generating for $(f^n,\nu)$.}
\end{multline}
Given a maximal $(n,\eps)$-separated set $E_n$, we can choose a partition $\beta_n$ such that each element of $\beta_n$ is contained in $B_n(x,\eps)$ for some $x\in E_n$, so $\beta_n$ has exactly $\#E_n$ elements.  Then we have
\begin{equation}\label{eqn:nhmu}
h_\mu(f) = \frac 1n h_\mu(f^n)
= \frac 1n h_\mu(f^n,\beta_n)
\leq \frac 1n H_\mu(\beta_n)
\leq \frac 1n \log \# E_n.
\end{equation}
Sending $n\to\infty$ gives $h_\mu(f) \leq \htop(X,f,\eps)$, and taking a supremum over all $\mu\in \Mf(X)$ proves that $\htop(X,f,\eps) = \htop(X,f)$.
\end{proof}

\subsection{Specification}\label{sec:specification}

The following formulation of the specification property is given for a collection of orbit segments $\DDD\subset X\times \NN$, and thus is not quite the classical one, but reduces to (a version of) the classical definition when we take $\DDD = X\times \NN$.  Observe that when $X$ is a shift space and we associate to each $(x,n)$ the word $x_{[1,n]} \in \LLL(X)$, the following agrees with the definition from \eqref{eqn:spec}.

\begin{figure}[htbp]
\includegraphics[width=\textwidth]{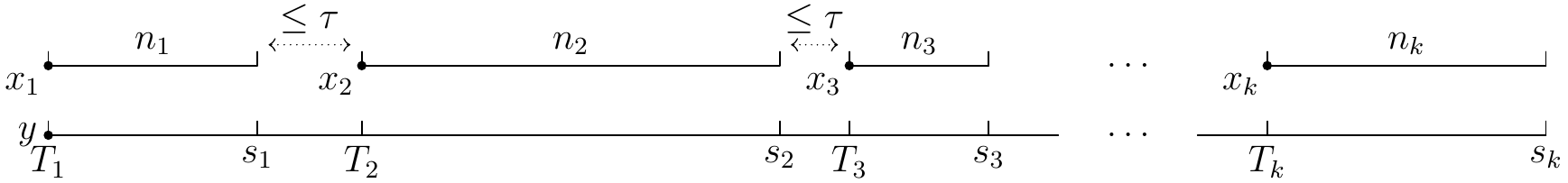}
\caption{Bookkeeping in the specification property.}
\label{fig:spec}
\end{figure}

\begin{definition}[Specification]\label{def:spec}
A collection of orbit segments $\DDD\subset X\times\NN$ has the \emph{specification property at scale $\delta>0$} if there exists $\tau\in\NN$ (the \emph{gap size} or \emph{transition time}) such that for every $(x_1,n_1),\dots, (x_k,n_k) \in \DDD$, there exist $0 = T_1 < T_2 < \cdots < T_k \in \NN$ and $y\in X$ such that
\[
f^{T_i}(y) \in B_{n_i}(x_i,\delta) 
\text{ and }
T_i - (T_{i-1} + n_{i-1}) \in [0,\tau]
\text{ for all } 1\leq i\leq k,
\]
see Figure \ref{fig:spec}. That is, starting from time $T_i$ the orbit of $y$ shadows the orbit of $x_i$, and moreover, writing $s_i = T_i + n_i$ for the time at which this shadowing ends, we have
\[
s_i \leq T_{i+1} \leq s_i + \tau \text{ for all } 1\leq i < k.
\]
We say that $\DDD$ has the \emph{specification property} if the above holds for every $\delta>0$.
We say that $(X,f)$ has the specification property if $X\times \NN$ does. We say that $\DDD$ has periodic specification if $y$ can be chosen to be periodic with period in $[s_k, s_k+\tau]$.
\end{definition}

First we explain how specification (for the whole system) is established in the uniformly hyperbolic case.
Recall from \eqref{eqn:spec} and the paragraph preceding it that in the symbolic case, one can establish specification by verifying it in the case $k=2$ and then iterating. In the non-symbolic case, the proof of specification usually follows this same approach, but one needs to verify a mildly stronger property for $k=2$ to allow the iteration step; one possible version of this property is formulated in the next lemma.

\begin{lemma}\label{lem:one-step-spec}
Given $f\colon X\to X$, suppose that $\delta_1>0$, $\delta_2\geq 0$, $\chi \in (0,1)$, and $\tau\in \NN$ are such that for every $(x_1,n_1), (x_2,n_2) \in X\times \NN$, there are $t\in \{0,1,\dots, \tau\}$ and $y\in X$ such that
\begin{equation}\label{eqn:one-step}
d(f^k y, f^k x_1) \leq \delta_1 \chi^{n_1-k} \text{ for all } 0\leq k < n_1
\text{ and }
d_{n_2}(f^{n_1+t} y, x_2) \leq \delta_2.
\end{equation}
Then $(X,f)$ has the specification property at scale $\delta_2 + \delta_1/(1-\chi)$ with gap size $\tau$.
\end{lemma}
\begin{proof}
Given $(x_1,n_1), \dots, (x_k, n_k) \in X\times \NN$, we will apply \eqref{eqn:one-step} iteratively to produce $y_1,\dots, y_k$ and $T_1,\dots, T_k \in \NN$ such that writing $\delta' = \delta_2 + \delta_1/(1-\chi)$, we have
\begin{equation}\label{eqn:fij}
f^{T_i}(y_j) \in B_{n_i}(x_i, \delta') \text{ for all } 1 \leq i \leq j \leq k.
\end{equation}
Once this is done, $y_k$ is the desired shadowing point. See Figure \ref{fig:spec-2} for an illustration of the following procedure and estimates.

Along with $y_i,T_i$, we will produce $s_i = T_i + n_i$ and $t_i \in \{0,\dots, \tau\}$ such that $T_{i+1} = s_i + t_i$.
Start by putting $y_1 = x_1$, $T_1 = 0$, and $s_1 = n_1$. Then apply \eqref{eqn:one-step} to $(y_1,s_1)$ and $(x_2,n_2)$ to get $y_2\in X$ and $t_1 \in \{0,1,\dots, \tau\}$ such that writing $T_2 = s_1 + t_1$, we have
\[
d(f^k y_2, f^k y_1) \leq \delta_1 \chi^{s_1-k} \text{ for all } 0\leq k < s_1
\text{ and }
d_{n_2}(f^{T_2} y_2, x_2) \leq \delta_2.
\]
In general, once $y_i, s_i$ are determined (with $T_i = s_i - n_i$), we apply \eqref{eqn:one-step} to $(y_i,s_i)$ and $(x_{i+1}, n_{i+1})$ to get $t_i \in \{0,1,\dots, \tau\}$ and $y_{i+1} \in X$ such that writing $T_{i+1} = s_i + t_i$, we have
\[
d(f^k y_{i+1}, f^k y_i) \leq \delta_1 \chi^{s_i-k} \text{ for all } 0\leq k < s_i
\text{ and }
d_{n_{i+1}}(f^{T_{i+1}} y_{i+1}, x_{i+1}) \leq \delta_2.
\]
Now we can verify \eqref{eqn:fij} by observing that for all $1\leq i\leq j \leq k$, we have
\begin{equation} \label{eqn:specdist}
d_{n_i}(f^{T_i}(y_j), f^{T_i}(y_i)) \leq \sum_{\ell=i}^{j-1} d_{n_i}(f^{T_i}(y_{\ell+1}), f^{T_i}(y_{\ell}))
\leq \sum_{\ell=i}^{j-1} \delta_1 \chi^{s_\ell - s_i} 
< \frac{\delta_1}{1-\chi},
\end{equation}
(the last inequality uses the fact that $s_\ell - s_i \geq \ell - i$), and also $d_{n_i}(f^{T_i}(y_i), x_i) \leq \delta_2$, so
\[
d_{n_i}(x_i,f^{T_i}y_j)
\leq d_{n_i}(x_i,f^{T_i}y_i) + d_{n_i}(f^{T_i}y_i,f^{T_i}y_j)
\leq \delta_2 + \frac{\delta_1}{1-\chi} = \delta'.\qedhere
\]
\end{proof}

\begin{figure}[htbp]
\includegraphics[width=.8\textwidth]{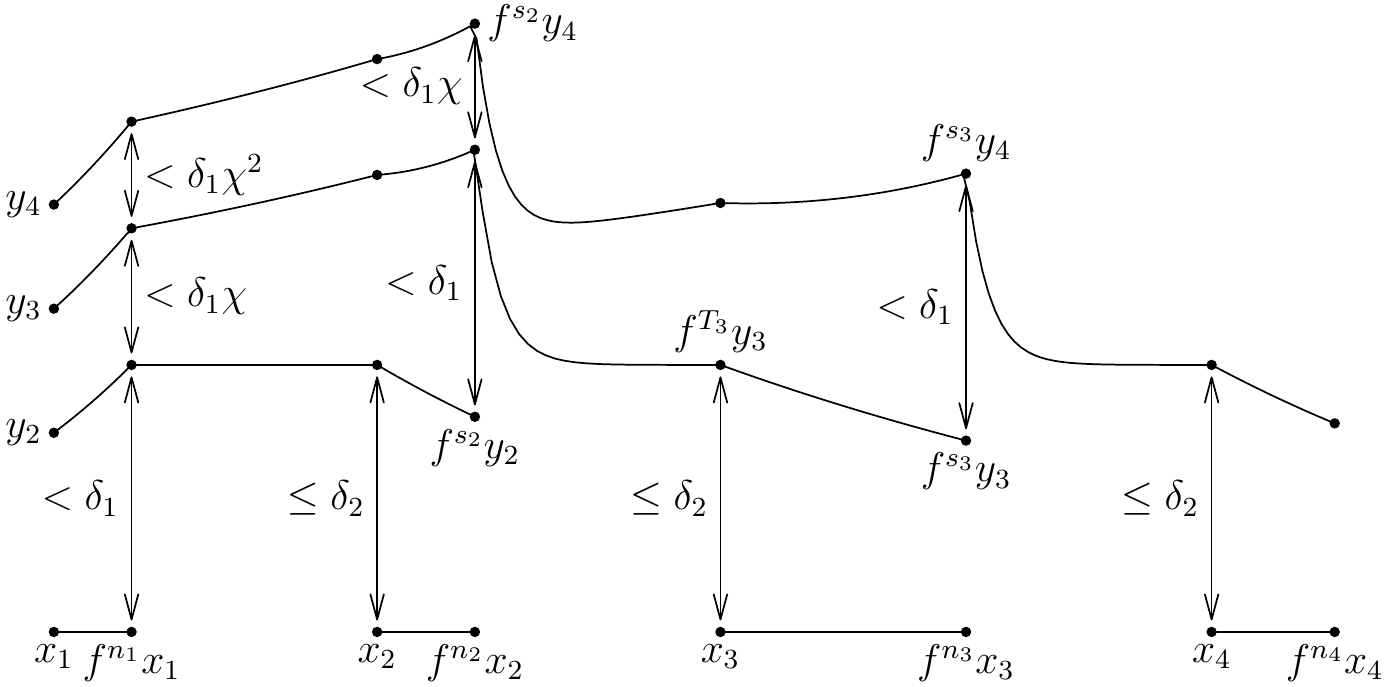}
\caption{Proving specification using a one-step property.}
\label{fig:spec-2}
\end{figure}

\begin{proposition}\label{prop:uh-spec}
If $X$ is a topologically transitive locally maximal hyperbolic set for a diffeomorphism $f$, then $(X,f)$ has the specification property.
\end{proposition}
\begin{proof}
By Lemma \ref{lem:one-step-spec}, it suffices to show that for every sufficiently small $\delta>0$, there are $\chi \in (0,1)$ and $\tau \in \NN$ such that for every $(x_1,n_1), (x_2, n_2) \in X\times \NN$, there are $t\in \{0,1,\dots, \tau\}$ and $y\in X$ such that \eqref{eqn:one-step} holds. To prove this, let $\delta,\rho>0$ be such that
\begin{itemize}
\item every $x\in X$ has local stable and unstable leaves $W_\delta^s(x)$ and $W_\delta^u(x)$ with diameter $<\delta$, and
\item for every $x,y\in X$ with $d(x,y) < \rho$, the intersection $W_\delta^s(x) \cap W_\delta^u(y)$ is a single point, which lies in $X$.
\end{itemize}
By topological transitivity and compactness, there is $\tau\in \NN$ such that for every $x,y\in X$ there is $t\in \{0,1,\dots, \tau\}$ with $d(f^t x, y) < \rho$, and thus
$f^t(W_\delta^u(x)) \cap W_\delta^s(y) \neq\emptyset$.   

Using this fact, given $(x_1,n_1), (x_2,n_2) \in X\times \NN$, we can let $t\in \{0,1,\dots, \tau\}$ be such that $f^t(W_\delta^u(f^{n_1}x_1))$ intersects $W_\delta^s(x_2)$. Choosing $z$ in this intersection and putting $y = f^{-(t+n_1)}(z)$, we see that $y$ satisfies \eqref{eqn:one-step} with $\delta_1 = \delta_2 = \delta$, and thus Lemma \ref{lem:one-step-spec} proves the proposition.
\end{proof}

\begin{remark}\label{rmk:Wcs-spec} Uniform contraction of $f$ along $W^s$ is not used; to prove specification at scale $\delta'$, it would suffice to know that if $x,y$ lie on the same local stable leaf and $d(x,y) \leq \delta_2$, then the same is true of $f(x),f(y)$, which still gives the second half of \eqref{eqn:one-step}. In particular, this follows as soon as $\|Df|_{E^s}\| \leq 1$. The same idea can also be applied to obtain specification on suitable collections $\GGG \subset X \times \NN$, and can be extended naturally to the continuous-time case. 

We also emphasize that the exponential contraction asked for in the first half of \eqref{eqn:one-step}, which is obtained from uniform backwards contraction along $W^u$, can be significantly weakened. What is really essential for the argument is backwards contraction in the local unstables by a fixed amount in each of the orbit segments (not necessarily proportional to length), and this is enough to obtain a uniform distance estimate analogous to \eqref{eqn:specdist}. We carried out the details of this argument in \cite[\S 4]{BCFT18} in the non-uniformly hyperbolic setting of rank one geodesic flow for the family of orbit segments $\CCC(\eta)$, which are defined in this survey in \S \ref{sec:bcftproofidea}.
\end{remark}

The following gives the corresponding result in the non-invertible case.

\begin{proposition}\label{prop:expanding}
Suppose that $f\colon X\to X$ is topologically transitive and has the following properties.
\begin{itemize}
\item \emph{Uniformly expanding:} $d(fx,fy) \geq e^{\lambda} d(x,y)$ whenever $d(x,y) < \delta$.
\item \emph{Locally onto:} 
For every $x\in X$, we have $f(B(x,\delta)) \supset B(fx,\delta)$.\footnote{In the symbolic setting, this corresponds to $X$ being a subshift of finite type.}
\end{itemize}
Then $(X,f)$ has the specification property at scale $\delta/(1-e^{-\lambda})$.
\end{proposition}
\begin{proof}
It suffices to verify \eqref{eqn:one-step} with $\delta_1 = \delta$, $\delta_2=0$, and $\chi = e^{-\lambda}$; then we can apply Lemma \ref{lem:one-step-spec}. We need the following consequence of the locally onto property: 
\begin{equation}\label{eqn:fnBn}
\text{for every $x\in X$ and $n\in\NN$, we have $f^n(B_n(x,\delta)) \supset B(f^n x, \delta)$}.
\end{equation}
As in the previous proposition, we use the following consequence of topological transitivity and compactness: given $\delta>0$, there is $\tau\in\NN$ such that for every $x,y\in X$ there is $t\in \{0,1,\dots, \tau\}$ with $f^t(x) \in B(y,\delta)$. 
Now given $(x_1,n_1), (x_2,n_2) \in X\times \NN$, there is $t\in \{0,1,\dots, \tau\}$ such that $f^t(f^{n_1}(x_1)) \in B(x_2,\delta)$, and thus \eqref{eqn:fnBn} gives
\[
f^t(f^{n_1} B_{n_1}(x_1,\delta)) \supset f^t B(f^{n_1} x_1, \delta)
\supset B(f^{n_1 + t} x_1, \delta) \ni x_2.
\]
Thus there is $y \in B_{n_1}(x_1,\delta)$ such that $f^{n_1 + t}(y) = x_2$, which verifies \eqref{eqn:one-step}; Lemma \ref{lem:one-step-spec} completes the proof.
\end{proof}

\subsection{Bowen's proof revisited}\label{sec:revisit}

Bowen's original uniqueness result \cite{rB75}, which we outlined in \S\ref{sec:principles}, was actually given not for shift spaces, but for more general expansive systems.

\begin{theorem}[Expansivity and specification (Bowen)] \label{thm:exp-spec} 
Let $X$ be a compact metric space and $f\colon X\to X$ a continuous map.  Suppose that $\eps>40\delta>0$ are such that $f$ has expansivity at scale $\eps$ and the specification property at scale $\delta$.  Then $(X,f)$ has a unique measure of maximal entropy.
\end{theorem}

\begin{remark}
Bowen's original paper assumed expansivity and periodic specification at all scales.  We relax the proof mildly so that it does not use periodic orbits and only uses specification at a fixed scale, small relative to an expansivity constant.\footnote{The statements in \cite{CT14} used $\eps\geq 28\delta$ but this must be corrected to $\eps>40\delta$; see \cite[\S5.7]{CT16}.}
We will see examples later where this additional generality is beneficial. 
\end{remark}

The proof of Theorem \ref{thm:exp-spec} extends the strategy in the symbolic case:
\begin{enumerate}
\item establish uniform counting bounds;
\item show that the usual construction of an MME gives an ergodic Gibbs measure;
\item prove that an ergodic Gibbs measure must be the unique MME.
\end{enumerate}
Now we examine the role played by expansivity.

\subsubsection{Uniform counting bounds}\label{sec:general-counting}

In the symbolic setting, the first step was to prove the \emph{counting bounds} on $\#\LLL_n$ given in \eqref{eqn:counting}. In the general setting, $\#\LLL_n$ is replaced with $\Lambda(X,\eps,n)$ from Definition \ref{def:entropy}, and mimicking the arguments in Lemma \ref{lem:counting} leads to the estimates
\begin{equation}\label{eqn:counting*}
e^{n\htop(X,f,6\delta)} \leq \Lambda(X,3\delta,n) \leq Q e^{n\htop(X,f,\delta)},
\quad\text{where } Q = (\tau + 1) e^{\tau \htop(X,f,\delta)}.
\end{equation}
Observe that the lower and upper bounds in \eqref{eqn:counting*} involve the entropy of $f$ at different scales, a phenomenon which did not appear in \eqref{eqn:counting}. 
To see why this occurs, recall that in the proof of Lemma \ref{lem:counting} we used an injective map
\begin{equation}\label{eqn:chop}
\LLL_{m+n} \to \LLL_m \times \LLL_n,
\qquad w \mapsto (w_{[1,m]}, w_{[m+1,m+n]}),
\end{equation}
as well as an at-most-$(\tau+1)$-to-1 map given by specification:
\begin{equation}\label{eqn:glue}
\LLL_m\times \LLL_n \to \LLL_{m+n+\tau},
\qquad (v,w) \mapsto vuwu'.
\end{equation}
In a general metric space, to generalize \eqref{eqn:chop} one might first attempt the following: 
\begin{itemize}
\item fixing $\rho>0$, let $E_k^\rho\subset X$ be a maximal $(k,\rho)$-separated set for each $k\in \NN$;
\item by maximality, for every $x\in X$ and $k\in \NN$ there is $\pi_k(x) \in E_k^\rho$ such that $x\in B_k(\pi_k(x),\rho)$;
\item then consider the map $E_{m+n}^\rho \to E_m^\rho \times E_n^\rho$ given by $x \mapsto (\pi_m(x), \pi_{n-m}(f^m x))$.
\end{itemize}
The problem is that injectivity may fail: there could be $z\in E_m$ such that $B_m(z,\rho)$ contains two distinct points $x,y\in E_{m+n}$, even though $d_m(x,y) \geq d_{m+n}(x,y) \geq \rho$. This possibility can be ruled out by considering a map $E_{m+n}^{2\rho} \to E_m^\rho \times E_n^\rho$; note the use of two different scales. With $\rho = 3\delta$, this leads to the lower bound in \eqref{eqn:counting*}. See \cite[\S3.1]{CT14} for details.

For the upper bound in \eqref{eqn:counting*}, one must play a similar game with \eqref{eqn:glue}. With $E_k^\rho$ as above, specification (at scale $\delta$) gives a ``gluing map'' $\pi\colon E_m^\rho \times E_n^\rho \to X$. As long as $\rho \geq \delta$, the multiplicity of this map is at most $\tau+1$ for the same reasons as in Lemma \ref{lem:counting}. However, since the gluing process in specification can move orbit segments by up to $\delta$, the image set $\pi(E_m^\rho\times E_n^\rho)$ can only be guaranteed to be $(\rho-2\delta,m+n+\tau)$-separated. Again, taking $\rho = 3\delta$ gives \eqref{eqn:counting*}; see \cite[\S3.2]{CT14} for details.

\begin{remark}\label{rmk:scales}
The reason that these issues do not arise in the symbolic setting is that there, if $\delta=\frac14$ and $y\in B_n(x,\delta)$, then $B_n(y,\delta) = [y_{[1,n]}] = [x_{[1,n]}] = B_n(x,\delta)$. In other words, in a shift space, each $d_n$ is an \emph{ultrametric}, for which the triangle inequality is strengthened to $d_n(x,z) \leq \max\{d_n(x,y),d_n(y,z)\}$.
In the non-symbolic setting, if $y\in B_n(x,\delta)$ then the most we can say is that $B_n(y,\delta) \subset B_n(x,2\delta)$, and vice versa. This leads to the ``changing scales'' aspect of the arguments above, which appears at several other places in the general proofs.
\end{remark}

\subsubsection{Construction of a Gibbs measure}\label{sec:build-gibbs}

With the counting bounds established as in \eqref{eqn:counting} and \eqref{eqn:counting*}, the next step in the symbolic proof was to consider measures $\nu_n$ giving equal weight to every $n$-cylinder, and prove a Gibbs property for any limit point of the measures $\mu_n = \frac 1n \sum_{k=0}^{n-1} \sigma_*^k \nu_n$.  
For non-symbolic systems, one replaces the collection of $n$-cylinders with a maximal $(n,\delta)$-separated set, and proves the following.

\begin{proposition}\label{prop:build-mme-2}
Let $X$ be a compact metric space and $f\colon X\to X$ a continuous map with the specification property at scale $\delta>0$ and expansivity at scale $\epsilon$, with $\epsilon>40 \delta$, and let $\rho \in (5\delta, \epsilon/8]$. Let $E_n\subset X$ be a maximal $(n,\rho-\delta)$-separated set for each $n$, and consider the measures
\begin{equation}\label{eqn:build-2}
\mu_n := \frac 1{\#E_n} \sum_{x\in E_n} \frac 1n \sum_{k=0}^{n-1} \delta_{f^k x}.
\end{equation}
Then there is $K\geq 1$ such that every weak* limit point $\mu$ of the sequence $\mu_n$ is $f$-invariant and satisfies the Gibbs property
\begin{equation}\label{eqn:gibbs2}
K^{-1} e^{-n\htop(X,f)}
\leq \mu(B_n(x,\rho)) \leq K e^{-n\htop(X,f)}
\quad\text{for all }x\in X, n\in\NN.
\end{equation}
\end{proposition}
This statement is a mild extension of the argument in \cite{rB75}, which is simplified by having periodic specification at all scales and constructing $\mu_n$ using periodic orbits. Proposition \ref{prop:build-mme-2} is proven, with the same level of detail on the choice of scales in \cite[\S6]{CT16}. For the purposes of this survey, the main point is simply that the expansivity scale is a suitably large multiple of the specification scale. However, we state the exact range of scales carefully for consistency with \cite[\S6]{CT16}. See also \cite{CT14} for a proof of the lower Gibbs bound. In that paper, many of the intermediate statements and bounds are given in terms of $\htop(X, f, c\rho)$, with $c\in \{1,2,3,4\}$. It is thus crucial that $\htop(X, f, c \rho)= \htop(X,f)$, which is provided in this statement by the expansivity assumption. This is the only way in which expansivity is used in the above proposition.


\subsubsection{Ergodicity}\label{sec:ergodic}

Observe that we have not yet claimed anything about ergodicity of the Gibbs measure $\mu$.  In the symbolic case, the argument for the Gibbs property can be used to deduce that there is $c>0$ and $k\in\NN$ such that for every $v,w\in\LLL$ and $\ell\geq |v|$, there is $j\in [\ell, \ell + k)$ such that
\[
\mu([v] \cap \sigma^{-j}[w]) \geq c\mu[v]\mu[w].
\]
Since any Borel set can be approximated (w.r.t.\ $\mu$) by unions of cylinders, this can be used to deduce that
\[
\ulim_{j\to\infty} \mu(V\cap \sigma^{-j} W) \geq \frac ck \mu(V)\mu(W)
\]
for all $V,W\subset X$, which gives ergodicity.  In the non-symbolic setting, one can still mimic the Gibbs argument to produce $c>0$ and $k\in\NN$ such that for every $(x,n),(y,m) \in X\times \NN$ and any $\ell \geq n$, there is $j\in [\ell, \ell + k)$ such that
\begin{equation}\label{eqn:partial-mix}
\mu(B_n(x,\rho) \cap f^{-j} B_m(y,\rho))
\geq c \mu(B_n(x,\rho)) \mu(B_m(y,\rho)).
\end{equation}
To establish ergodicity from this one needs to approximate arbitrary Borel sets by sets whose $\mu$-measure we control; this can be done by using a sequence of \emph{adapted partitions} $\beta_n$, for which each element of $\beta_n$ contains a Bowen ball $B_n(x,\rho)$ and is contained inside a Bowen ball $B_n(x,2\rho)$.  Expansivity implies that this sequence of partitions is generating w.r.t.\ $\mu$, so the rest of the argument goes through as before, and establishes ergodicity.  As we saw in the proof of Proposition \ref{prop:appeared}, this is also enough to guarantee that $\htop(X,f,\epsilon) = \htop(X,f)$.  We summarize our conclusions as follows.

\begin{proposition}\label{prop:ergodic}
Let $X,f,\delta,\mu$ be as in Proposition \ref{prop:build-mme-2}.  Suppose that $f$ is expansive at scale $40\delta$.  Then $\mu$ is ergodic and satisfies the Gibbs property \eqref{eqn:gibbs2}.
\end{proposition}

\subsubsection{Adapted partitions and uniqueness}\label{sec:adapted}

The proof that an ergodic Gibbs measure is the unique MME (Proposition \ref{prop:uniqueness}) has the following generalization to the non-symbolic setting.

\begin{proposition}\label{prop:uniqueness-2}
Let $X$ be a compact metric space, $f\colon X\to X$ a continuous map, and $\mu$ an ergodic $f$-invariant measure on $X$.  Suppose $\rho>0$ is such that
\begin{itemize}
\item $f$ is expansive (or positively expansive) at scale $4\rho$;
\item there are $K,h>0$ such that $\mu$ satisfies the Gibbs bound
\begin{equation}\label{eqn:gibbs-rho}
K^{-1} e^{-nh} \leq \mu(B_n(x,\rho)) \leq K e^{-nh}
\text{ for every } x\in X \text{ and } n\in\NN.
\end{equation}
\end{itemize}
Then $h=h_\mu(f) = \htop(X,f)$, and $\mu$ is the unique MME for $(X,f)$.
\end{proposition}
\begin{proof}[Outline of proof]
As before, one starts by using general arguments to prove that $h=h_\mu(f) = \htop(X,f)$ and to reduce to the case of considering an invariant measure $\nu\perp\mu$, for which we must show $h_\nu(f) < h_\mu(f)$; this is unchanged from the symbolic case.  The next step there was to choose $D\subset X$ with $\mu(D) = 1$ and $\nu(D)=0$, and approximate $D$ by a union of cylinders; then similar to \eqref{eqn:nhmu}, writing
\begin{equation}\label{eqn:nhnu}
nh_\nu(f) = h_\nu(f^n)
= h_\nu(f^n,\alpha_0^{n-1}) \leq H_\nu(\alpha_0^{n-1})
= \sum_{w\in \LLL_n} -\nu[w] \log\nu[w],
\end{equation}
and splitting the sum between cylinders in $\DDD_n$ and those in $\DDD_n^c$, one eventually proves that $h_\nu(f) < h_\mu(f)$ by using the Gibbs bound $\mu[w] \geq K^{-1} e^{-|w|\htop(X)}$.

In the non-symbolic setting, the approximation of $D$ follows just as in the paragraph after \eqref{eqn:partial-mix}.  Moreover, we can obtain an analogue of \eqref{eqn:nhnu} by replacing $\alpha_0^{n-1}$ with a partition $\beta_n$ such that every element of $\beta_n$ is contained in $B_n(x,2\rho)$ for some point $x$ in a maximal $(n,2\rho)$-separated set $E_n$.  Finally, as long as we also arrange that each element of $\beta_n$ contain $B_n(x,\rho)$, we can use the lower Gibbs bound to complete the proof just as in the symbolic case. 
\end{proof}

\begin{remark}
The partition $\beta_n$ which appears in the above proof is called an \emph{adapted partition} for $E_n$. Adapted partitions exist for any $(n, 2\rho)$-separated set of maximal cardinality since the sets $B_n(x, \rho)$ are disjoint and the sets $\overline B_n (x, 2 \rho)$ cover $X$.
\end{remark}

\begin{remark}
In the two-sided expansive case, the same argument works, provided we replace $d_n$ and $B_n$ with their two-sided versions. That is, we consider  balls in the metric $d_{[-n, n]}(x,y) = \max \{d(f^k x, f^k y) : -n \leq k \leq n \}$ in place of $B_n$. Then one uses adapted partitions and proceeds as in the positively expansive case.
\end{remark}

\part{Non-uniform Bowen hypotheses and equilibrium states}\label{part:2} In \S \ref{sec:weakexp}, we recall the role played by expansivity in Bowen's proof of uniqueness, and formulate a uniqueness result using a weaker version of expansivity. Then in \S\ref{sec:DA} we describe an explicit class of partially hyperbolic diffeomorphisms where expansivity fails but this result still applies. In \S\ref{sec:nonuniformMME} we combine the weakened versions of expansivity and specification to formulate our most general result on MMEs for discrete-time systems, which we apply in \S\ref{sec:ph} to a more general partially hyperbolic setting. Finally, in \S\ref{sec:eq-st} we describe how this theory extends to equilibrium states for nonzero potential functions. 

\section{Relaxing the expansivity hypothesis} \label{sec:weakexp}

In this section, we describe how we relax the expansivity property.  Our motivating examples are diffeomorphisms for which expansivity fails, but for which the failure of expansivity is ``invisible'' to the MME. In these examples, the failure of expansivity is a lower entropy phenomenon, and this  leaves room for us to develop a version of Bowen's argument for the MME.

As explained in the previous section, Bowen's proof of uniqueness uses expansivity to guarantee that certain sequences of partitions are generating with respect to every invariant $\nu$.  In fact, in every place where this property is used, it is enough to know that this holds for all $\nu$ with sufficiently large entropy.

More precisely, at the end of the proof, in (the analogue of) \eqref{eqn:nhnu}, it suffices to know that $\alpha_0^{n-1}$ is generating for $(f^n,\nu)$ when $\nu$ is an arbitrary MME, because if $\nu$ is not an MME then we already have $h_\nu < h_\mu$, which was the goal.  This is also sufficient for the approximation of $D$ by elements of the partitions $\beta_n$, and thus Proposition \ref{prop:uniqueness-2} remains true if we replace expansivity with the assumption that for every MME $\nu$, we have $\Gamma_\eps(x) = \{x\}$ for $\nu$-a.e.\ $x$.

In Proposition \ref{prop:build-mme-2}, the argument for ergodicity required a similar generating property.  Finally, in Proposition \ref{prop:appeared}, it suffices to have this generating property w.r.t.\ a family of measures $\nu$ over which $\sup_\nu h_\nu(f) = \htop(X,f)$.

With these observations in mind, we make the following definitions.

\begin{definition}[\cite{BF13}]\label{def:almost-expansive}
An $f$-invariant measure $\mu$ is \emph{almost expansive at scale $\eps$} if $\Gamma_\eps(x) = \{x\}$ for $\mu$-a.e.\ $x$; equivalently, if the \emph{non-expansive set}  $\NE(\eps) = \{x\in X : \Gamma_\eps(x) \neq \{x\}\}$ has $\mu(\NE(\eps))=0$.  Replacing $\Gamma_\eps$ by $\Gamma_\eps^+$ gives $\NE^+$ and a notion of \emph{almost positively expansive}.
\end{definition}

\begin{definition}[\cite{CT14}]\label{def:hexp}
The \emph{entropy of obstructions to expansivity at scale $\eps$} is
\begin{align*}
\hexp(X,f,\eps) &:= \sup \{h_\mu(f) : \mu\in\Mf^e(X) \text{ is not almost expansive at scale } \eps\} \\
&= \sup \{h_\mu(f) : \mu\in\Mf^e(X) \text{ and } \mu(\NE(\eps))>0\}.
\end{align*}
We write $\hexp(X,f) = \lim_{\eps\to0} \hexp(X,f,\eps)$ for the \emph{entropy of obstructions to expansivity}, without reference to scale.  The entropy of obstructions to positive expansivity $\hexpp$ is defined analogously.
\end{definition}

From the discussion after Proposition \ref{prop:build-mme-2}, we see that we can replace the assumption of expansivity with the assumption that $\hexp(X,f,\rho) < \htop(X,f)$, since then every ergodic $\nu$ with $h_\nu(f) > \hexp(X,f,\rho)$ is almost expansive, so the Proposition goes through.\footnote{See \cite[Proposition 2.7]{CT14} for a detailed proof that $\htop(X,f,\rho)=\htop(X,f)$ in this case.}  Similarly in Proposition \ref{prop:ergodic}and Proposition \ref{prop:uniqueness-2}, it suffices to assume that $\hexp(X,f,4\rho) < \htop(X,f)$.  

Now we have all the pieces for a uniqueness result using non-uniform expansivity.

\begin{theorem}[Unique MME with non-uniform expansivity \cite{CT14}]\label{thm:hexp}
Let $X$ be a compact metric space and $f\colon X\to X$ a continuous map.  Suppose that $\eps>40\delta>0$ are such that $\hexp(X,f,\eps) < \htop(X,f)$, and that $f$ has the specification property at scale $\delta$.  Then $(X,f)$ has a unique measure of maximal entropy.
\end{theorem}

\section{Derived-from-Anosov systems}\label{sec:DA}

We describe a class of smooth systems for which expansivity fails but the entropy of obstructions to expansivity is small. The following example is due to Ma\~n\'e \cite{rM78}; we primarily follow the discussion in \cite{CFT19}, and refer to that paper for further details and references.

\subsection{Construction of the Ma\~n\'e example}

Fix a matrix $A\in SL(3,\ZZ)$ with simple real eigenvalues $\lambda_u > 1 > \lambda_s > \lambda_{ss} > 0$, and corresponding eigenspaces $F^{u,s,ss} \subset \RR^3$.
Let $f_0\colon \TT^3\to \TT^3$ be the hyperbolic toral automorphism defined by $A$, and let $\mathcal{F}^{u,s,ss}$ be the corresponding foliations of $\TT^3$.
Define a perturbation $f$ of $f_0$ as follows.

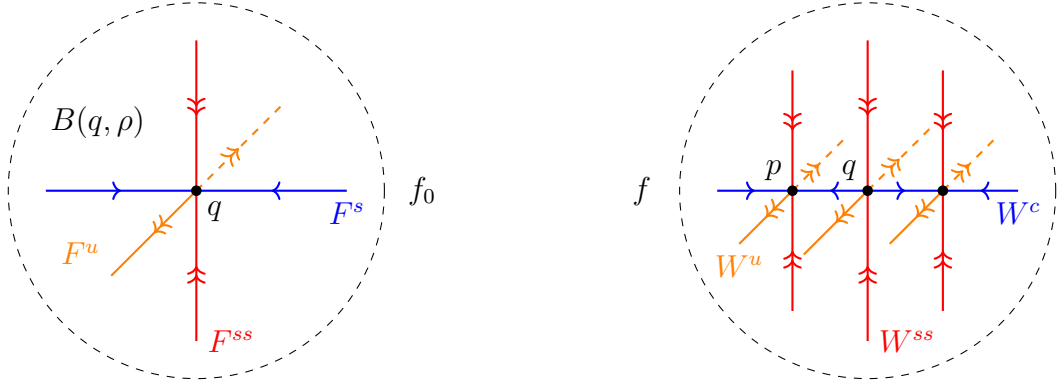
\begin{figure}[htbp]
\begin{tikzpicture}
\coordinate (q) at (0,0);
\def\r{2};
\draw[dashed] (q) circle (2.5);
\begin{scope}[thick,decoration={markings,
					mark=at position 0.5 with {\arrow{>}}}]
\draw[blue,postaction={decorate}] (0:\r) node[below]{$F^s$} -- (q);
\draw[blue,postaction={decorate}] (180:\r) -- (q);
\end{scope}
\begin{scope}[thick,decoration={markings,
					mark=at position 0.5 with {\arrow{>>}}}]
\draw[red,postaction={decorate}] (90:\r) -- (q);
\draw[red,postaction={decorate}] (270:\r) node[right]{$F^{ss}$} -- (q);
\draw[orange,postaction={decorate}] (q) -- (225:{.8*\r}) node[above left]{$F^u$};
\draw[orange,dashed,postaction={decorate}] (q) -- (45:{.8*\r});
\end{scope}
\fill (q) node[below right]{$q$} circle(2pt);
\draw (145:.8*\r) node{$B(q,\rho)$};
\draw (3,0) node{$f_0$};
\end{tikzpicture}
\hspace{5em}
\begin{tikzpicture}
\coordinate (q) at (0,0);
\draw[dashed] (q) circle (2.5);
\def\r{2};
\coordinate (s) at (-1,0);
\coordinate (t) at (1,0);
\begin{scope}[thick,decoration={markings,
					mark=at position 0.5 with {\arrow{>}}}]
\draw[blue,postaction={decorate}] (0:\r) node[below]{$W^c$} -- (t);
\draw[blue,postaction={decorate}] (180:\r) -- (s);
\draw[blue,postaction={decorate}] (q) -- (s);
\draw[blue,postaction={decorate}] (q) -- (t);
\end{scope}
\begin{scope}[thick,decoration={markings,
					mark=at position 0.5 with {\arrow{>>}}}]
\draw[red,postaction={decorate}] (90:\r) -- (q);
\draw[red,postaction={decorate}] (270:\r) node[right]{$W^{ss}$} -- (q);
\draw[orange,postaction={decorate}] (q) -- (225:{.6*\r});
\draw[orange,dashed,postaction={decorate}] (q) -- (45:{.6*\r});
\draw[red,postaction={decorate}] (s) +(90:.8*\r) -- (s);
\draw[red,postaction={decorate}] (t) +(90:.8*\r) -- (t);
\draw[red,postaction={decorate}] (s) +(270:.8*\r) -- (s);
\draw[red,postaction={decorate}] (t) +(270:.8*\r) -- (t);
\draw[orange,postaction={decorate}] (s) -- +(225:{.5*\r}) node[below]{$W^u$};
\draw[orange,dashed,postaction={decorate}] (s) -- +(45:{.5*\r});
\draw[orange,postaction={decorate}] (t) -- +(225:{.5*\r});
\draw[orange,dashed,postaction={decorate}] (t) -- +(45:{.5*\r});
\end{scope}
\fill (q) node[above left]{$q$} circle(2pt)  (t) circle(2pt) (s) node[above left]{$p$} circle(2pt);
\draw (-3,0) node{$f$};
\end{tikzpicture}
\caption{Ma\~n\'e's construction.}
\label{fig:mane}
\end{figure}

Fix $\rho>\rho'>0$ such that $f_0$ is expansive at scale $\rho$.  Let $q\in\TT^3$ be a fixed point of $f$, and set $f=f_0$ outside of $B(q,\rho)$.  Inside $B(q,\rho)$, perform a pitchfork bifurcation in the center direction as shown in Figure \ref{fig:mane}, in such a way that
\begin{itemize}
\item the foliation $W^c := \mathcal{F}^s$ remains $f$-invariant, and we write $E^c = TW^c$;
\item the cones around $F^u$ and $F^{ss}$ remain invariant and uniformly expanding for $Df$ and $Df^{-1}$, respectively, so they contain $Df$-invariant distributions $E^{u,ss}$ that integrate to $f$-invariant foliations $W^{u,ss}$;
\item $E^{cs} = E^c \oplus E^{ss}$ integrates to a foliation $W^{cs}$;
\item outside of $B(q,\rho')$, we have $\|Df|_{E^{cs}}\| \leq \lambda_s < 1$.
\end{itemize}
Thus $f$ is partially hyperbolic with $T \TT^3 = E^u \oplus E^c \oplus E^{ss} = E^u \oplus E^{cs}$.  Observe that
\begin{equation}\label{eqn:lcf}
\lambda_c(f) := \sup \{ \|Df|_{E^{cs}(x)}\| : x\in \TT^3 \} > 1
\end{equation}
because the center direction is expanding at $q$.

Now consider a diffeomorphism $g\colon \TT^3\to \TT^3$ that is $C^1$-close to $f$.  Such a $g$ remains partially hyperbolic, with
\begin{equation}\label{eqn:lcls}
\lambda_c(g) > 1 > \lambda_s(g) := \sup\{ \|Df|_{E^{cs}(x)}\| : x\in \TT^3\setminus B(q,\rho') \}.
\end{equation}
Existence of a unique MME was proved for such $g$ by Ures \cite{rU12} and by Buzzi, Fisher, Sambarino, and V\'asquez \cite{BFSV12}, using the fact that there is a semiconjugacy from $g$ back to the hyperbolic toral automorphism $f_0$.  We outline an alternate proof using Theorem \ref{thm:hexp}, which has the benefit of extending to a class of nonzero H\"older continuous potential functions \cite{CFT19}.

\subsection{Estimating the entropy of obstructions}
Although the map $g$ behaves as if it is uniformly hyperbolic outside of $B(q,\rho)$, the presence of fixed points with different indices inside this ball causes expansivity to fail.  Indeed, let $p$ denote one of the two fixed points created via the pitchfork bifurcation, and let $x$ be any point on the  leaf of $W^c$ that connects $p$ to $q$.  Then for every $\eps>0$, the bi-infinite Bowen ball $\Gamma_\eps(x)$ is a non-trivial curve in $W^c$, rather than a single point. However, we can give a simple mild criterion on the orbit of a point $x$ which rules out $\Gamma_\eps(x)$ being non-trivial, and we can argue that this criterion is satisfied for most points in our examples.

\begin{lemma}\label{lem:ph-ne}
Let $g$ be a partially hyperbolic diffeomorphism with a splitting $E^u \oplus E^c \oplus E^s$ such that $E^c$ is 1-dimensional and integrable.  Then there is $\eps_0>0$ such that $\Gamma_{\eps_0}(x) \subset W^c(x)$ for every $x$. Moreover, for every $\lambda>0$ there is $\eps>0$ such that 
\begin{equation}\label{eqn:lyap-exp}
\ulim_{n\to\infty} \frac 1n \log \|Dg^{-n}|_{E^c(x)}\| > \lambda
\quad\Rightarrow\quad
\Gamma_\eps(x) = \{x\}.
\end{equation}
\end{lemma}
\begin{proof}[Sketch of proof]
Following the argument for expansivity in the uniformly hyperbolic setting, we choose $\eps_0$ such that whenever $d(x,y) < \eps_0$, we can get from $x$ to $y$ by moving a distance $d^s$ along a leaf of $W^s$, then a distance $d^c$ along a leaf of $W^c$, then a distance $d^u$ along a leaf of $W^u$.  The argument given there shows that if $y\in \Gamma_{\eps_0}(x)$ then we must have $d^s(x,y) = d^u(x,y)=0$, which implies that $y\in W^c(x)$.  For \eqref{eqn:lyap-exp}, we observe that if the condition on $Dg^{-n}$ is satisfied, then there are arbitrarily large $n$ such that
\begin{equation}\label{eqn:back-lyap}
\|Dg^{-n}|_{E^c(x)}\| > c e^{\lambda n}.
\end{equation}
Choosing $\eps>0$ sufficiently small that $|\log \|Dg|_{E^c(z)}\| - \log\|Dg|_{E^c(z')}\|| < \lambda/2$ whenever $d(z,z')<\eps$, we see that any $y\in \Gamma_\eps(x)$ satisfies
\begin{equation}\label{eqn:center-hyp}
d(g^{-n}x, g^{-n}y) \geq c e^{\lambda n/2} d(x,y)
\end{equation}
for all $n$ satisfying \eqref{eqn:back-lyap}.  Since $n$ can become arbitrarily large, this implies that $d(x,y)=0$.
\end{proof}

\begin{remark}\label{rmk:back}
Replacing backwards time with forwards time, the analogous result for positive Lyapunov exponents is also true: $\ulim \frac 1n \log \|Dg^n|_{E^c(x)}\| > \lambda$ implies that $\Gamma_\eps(x) = \{x\}$.
\end{remark}

For the Ma\~n\'e examples, we can use \eqref{eqn:lcls} to control $\|Dg^{-n}|_{E^c(x)}\|$ in terms of how much time the orbit of $x$ spends outside $B(q,\rho)$; together with Lemma \ref{lem:ph-ne}, this allows us to estimate the entropy of $\NE(\eps)$.
To formalize this, we write $\chi = \mathbf{1}_{\TT^3 \setminus B(q,\rho)}$ and observe that by the definition of $\lambda_c(g)$ and $\lambda_s(g)$ in \eqref{eqn:lcf} and \eqref{eqn:lcls}, we have
\[
\|Dg^{-n}|_{E^c(x)}\|
\geq \lambda_s(g)^{-s_n(x)} \lambda_c(g)^{-(n-s_n(x))} 
\quad\text{ where } s_n(x) := \sum_{k=0}^{n-1} \chi(g^{-k}x).
\]
It follows that
\begin{equation}\label{eqn:ulimg}
\ulim_{n\to\infty} \frac 1n \log\|Dg^{-n}|_{E^c(x)}\|
\geq -(\overline{r}(x) \log \lambda_s(g) + (1-\overline{r}(x)) \log \lambda_c(g))
\end{equation}
where we write
\[
\overline{r}(x) = \ulim_{n\to\infty} \frac 1n s_n(x) = \ulim_{n\to\infty} \frac 1n \sum_{k=0}^{n-1} \chi(g^{-k}x).
\]
Fix $\lambda\in (0, -\log \lambda_s(g))$ and let $r>0$ satisfy $-(r\log\lambda_s(g) + (1-r)\log\lambda_c(g)) > \lambda$.  Then Lemma \ref{lem:ph-ne} and \eqref{eqn:ulimg} show that for a sufficiently small $\eps>0$, we have
\begin{equation}\label{eqn:NE-eps}
\NE(\eps) \subset \{ x : \overline{r}(x) < r \}.
\end{equation}
Since $f_0$ is Anosov, the uniform counting bounds in \eqref{eqn:counting*} give a constant $Q$ such that $\Lambda(X,f_0,\eps,n) \leq Q e^{n\htop(X,f_0)}$ for all $n$. Using this together with \eqref{eqn:NE-eps} one can prove the following.

\begin{lemma}[{\cite[\S3.4]{CFT18}}]\label{lem:hexp}
Writing $H(t) = -t\log t - (1-t)\log(1-t)$ for the usual bipartite entropy function, the Ma\~n\'e examples satisfy
\[
\hexp(g,\eps) < r(\htop(X,f_0) + \log Q) + H(2r).
\]
\end{lemma}
\begin{proof}[Idea of proof]
Given an ergodic measure $\mu$ that satisfies $\mu(\NE(\eps))$ and thus satisfies $\ulim \frac 1n S_n\chi(g^{-n} x) \leq r$ for $\mu$-a.e.\ $x$, the Katok entropy formula \cite{aK80} can be used to show that $h_\mu(f) \leq h(\CCC)$, where
\begin{equation}\label{eqn:C}
\CCC := \{(x,n) \in \TT^3\times \NN : S_n\chi(x) \leq rn \}.
\end{equation}
To estimate $h(\CCC)$, the idea is to partition an orbit segment $(x,n)\in\CCC$ into pieces lying entirely inside or outside of $B(q,\rho)$.  There can be at most $rn$ pieces lying outside, so the number of transition times between inside and outside is at most $2rn$.  The number of ways of choosing these transition times is thus at most
\[
\binom{n}{2rn} = \frac{n!}{(2rn)!((1-2r)n)!} \approx e^{H(2r)n},
\]
where the approximation can be made more precise using Stirling's formula or a rougher elementary integral estimate.  This contributes the $H(2r)$ term to the estimate; the remaining terms are roughly due to the observation that given a pattern of transition times for which the segments lying outside $B(q,\rho)$ have lengths $k_1,\dots, k_m$, the number of $\eps$-separated orbit segments in $\CCC$ associated to this pattern is at most
\[
\prod_{j=1}^m \Lambda(X,f_0,\eps,k_i)
\leq \prod_{j=1}^m Q e^{k_i \htop(X,f_0)}
\leq Q^m e^{rn\htop(X,f_0)}
\leq (Q e^{\htop(X,f_0)})^{rn},
\]
since no entropy is produced by the sojourns inside $B(q,\rho)$.
\end{proof}

Since there is a semi-conjugacy from $g$ to $f_0$, we have $\htop(X,g) \geq \htop(X,f_0)$.  Thus we have $\hexp(g) < \htop(g)$ whenever $r$ satisfies
\begin{equation}\label{eqn:rh}
r(\htop(X,f_0) + \log Q) + H(2r) < \htop(X,f_0).
\end{equation}
Recall that $r$ must be chosen large enough such that $\lambda_s(g)^r \lambda_c(g)^{1-r} < 1$.  Equivalently, for a given value of $r$, the perturbation must be chosen small enough for this to hold (that is, $\lambda_c$ must be close enough to $1$).
Thus given $f_0$, we can find $r$ small enough such that \eqref{eqn:rh} holds, and then for any sufficiently small perturbation the above argument guarantees that $\hexp(X,g) < \htop(X,g)$.

\begin{remark}\label{rmk:BV}
Since $\Gamma_\eps(x) \subset W^c(x)$, which is one-dimensional, it is not hard to show that $\htop(W^c(x)) = 0$, and thus $\htop(\Gamma_\eps(x))=0$ \cite{CY05,CFT19}; in other words, $f$ is \emph{entropy expansive}.
Entropy expansivity implies that $\htop(X,f,\epsilon) = \htop(X,f)$ \cite{rB72-h}, which for systems with (coarse) specification is sufficient for the construction of a Gibbs measure in Proposition \ref{prop:build-mme-2}.  However, there does not seem to be any way to use entropy expansivity to carry out the arguments for ergodicity and uniqueness. The issue is that we need to use Bowen balls to construct adapted partitions which approximate Borel sets. When $\Gamma_\eps(x)$ is a point, the two-sided Bowen ball at $x$ is a neighborhood of the point, which is key to the approximation argument. The analysis is significantly more difficult even when $\Gamma_\eps(x) \neq \{x\}$ has a simple explicit characterization, see \S \ref{sec:flow-unique} for more details in the flow case. If all we know about $\Gamma_\eps(x)$ is that $h(\Gamma_\eps(x))=0$ it is unclear how to proceed. On the other hand, for the Bonatti--Viana examples introduced in \cite{BV00}, entropy expansivity can fail \cite {BF13} even while the condition $\hexp < \htop$ is satisfied \cite{CFT18}. The Bonatti--Viana examples are 4-dimensional analogues of the Ma\~n\'e examples that involve two separate perturbations and have a dominated splitting $T\TT^4 = E^{cu} \oplus E^{cs}$ but are not partially hyperbolic. We were able to study their thermodynamic formalism in \cite{CFT18} despite these difficulties.
\end{remark}

\subsection{Specification for Ma\~n\'e examples}
In order to apply Theorem \ref{thm:hexp} to the Ma\~n\'e examples, one must investigate the specification property. Globally, specification at all scales certainly fails.  
Two approaches to deal with this are possible, and it is instructive to consider both -- our choice is to work with a \emph{coarse} specification property globally, or specification at all scales on a `good collection of orbit segments'. 

The key ingredient we are missing from the uniformly hyperbolic case is uniform contraction along $W^{cs}$, which is replacing $W^s$.  We explain why we can obtain coarse specification globally. As explained in Remark \ref{rmk:Wcs-spec}, uniform contraction is not needed for the proof of specification; it suffices to know that 
\begin{equation}\label{eqn:Wcs-in-Bowen}
W_\delta^{cs}(x) \subset B_n(x,\delta) \text{ for all } x.
\end{equation}
Since contraction in $W^{cs}$ can fail for the Ma\~n\'e example only in $B(q,\rho')$, one can easily show that \eqref{eqn:Wcs-in-Bowen} continues to hold as long as $\delta > 2\rho'$, and thus $g$ has specification at these scales.  Choosing $\rho'$ to be small enough relative to $\rho$, Theorem \ref{thm:hexp} applies and establishes existence of a unique MME.

To see that the Ma\~n\'e example does not have the specification property at all scales, we sketch a short argument which appears in much greater generality in \cite{SVY16}. Observe that for sufficiently small $\delta>0$, the forward infinite Bowen ball $\Gamma_\delta^+(q)$ is the 1-dimensional local stable leaf $W_\delta^{ss}(q)$.  Suppose that $g$ has specification at scale $\delta$ with gap size $\tau$, and let $x$ be any point whose orbit never enters $B(q,\rho)$.  Specification gives $y\in W^u_\delta(x)$ and $0\leq k\leq \tau$ such that $f^k(y) \in W^{ss}_\delta(q)$;\footnote{Use specification to get $y_n \in f^n(B_n(x,\delta)) \cap f^{-k_n}(B_n(q,\delta))$ for $0\leq k_n\leq \tau$, choose $k$ such that $k_n = k$ for infinitely many values of $n$, and let $y$ be a limit point of the corresponding $y_n$.}
In other words, $f^{-\tau}(W^{ss}_\delta(q))$ intersects \emph{every} local unstable leaf associated to an orbit that avoids $B(q,\rho)$.
But this is impossible because the dimensions are wrong.\footnote{Note that $f^{-\tau}(W^{ss}_\delta(q))$ intersects a local leaf of $W^{cu}$ in at most finitely many points, and thus thus intersects at most finitely many of the corresponding local leaves of $W^u$; however, there are uncountably many of these corresponding to points that never enter $B(q,\rho)$.}

Thus, if we want a global specification property, we must work at a fixed coarse scale, as described above. We explore the other option of returning to the ideas from \S\ref{sec:relax-spec} and recovering specification at all scales by restricting to a ``good collection of orbit segments'' in the next section.

\section{The general result for MMEs in discrete-time} \label{sec:nonuniformMME}

Now we formulate a general result that combines the symbolic result using decompositions with Theorem \ref{thm:hexp} by allowing both expansivity and specification to fail, provided the obstructions have small entropy.  This allows us to cover some new classes of examples, as we will see later, and is also important in dealing with nonzero potential functions.

Recall from \S\ref{sec:relax-spec} that a decomposition of the language $\LLL$ of a shift space consists of $\Cp,\GGG,\Cs\subset \LLL$ such that every $w\in \LLL$ can be written as $w=u^pvu^s$ where $u^p\in\Cp$, $v\in\GGG$, and $u^s\in\Cs$.  As discussed in \S\ref{sec:topological-entropy}, for non-symbolic systems we replace $\LLL$ with the \emph{space of orbit segments} $X\times\NN$, where $(x,n)$ corresponds to the orbit segment $x, f(x), f^2(x), \dots, f^{n-1}(x)$.

\begin{definition}\label{def:decomp-2}
A \emph{decomposition} for $X\times \NN$ consists of three collections $\Cp, \GGG, \Cs \subset X\times \NN_0$ for which there exist three functions $p,g,s\colon X\times \NN \to \NN_0$ such that for every $(x,n) \in X\times \NN$, the values $p=p(x,n)$, $g=g(x,n)$, and $s=s(x,n)$ satisfy $p+g+s=n$, and
\[
(x,p) \in \Cp,
\quad (f^p x,g) \in \GGG,
\quad (f^{p+s} x, s) \in \Cs.
\]
Given a decomposition, for each $M\in\NN$ we write
\[
\GGG^M := \{(x,n) \in X\times \NN : p(x,n) \leq M
\text{ and } s(x,n) \leq M \}.
\]
\end{definition}

\begin{theorem}[Non-uniform Bowen hypotheses for maps (MME case)]\label{thm:hexp-decomp}
Let $X$ be a compact metric space and $f\colon X\to X$ a continuous map.  Suppose that $\eps>40\delta>0$ are such that $\hexp(X,f,\eps) < \htop(X,f)$, and that the space of orbit segments $X\times \NN$ admits a decomposition $\Cp\GGG\Cs$ such that
\begin{enumerate}[label=\upshape{(\Roman{*})}]
\item\label{GM-spec-2} every collection $\GGG^M$ has specification at scale $\delta$, and
\item\label{hC<h-2} $h(\Cp \cup \Cs,\delta) < \htop(X,f)$.
\end{enumerate}
Then $(X,f)$ has a unique measure of maximal entropy.
\end{theorem}

The proof of Theorem \ref{thm:hexp-decomp} requires an extension of the counting arguments for decompositions (\S\ref{sec:decomp}) to the general metric space setting, following similar ideas to those outlined in \S\ref{sec:general-counting}. Similarly, the construction of a Gibbs measure in \S\ref{sec:build-gibbs} and the proofs of ergodicity and uniqueness in \S\S\ref{sec:ergodic}--\ref{sec:adapted} must be modified to reflect the fact that uniform lower bounds can only be obtained on $\GGG^M$. As in \S\ref{sec:decomp}, we omit further discussion of these more technical aspects, referring to \cite{CT14,CT16} for complete details.

\begin{remark}
If $\GGG$ has specification at all scales, then a short continuity argument \cite[Lemma 2.10]{CT16} proves that every $\GGG^M$ does as well, which establishes \ref{GM-spec-2}.
\end{remark}

\section{Partially hyperbolic systems with one-dimensional center}\label{sec:ph}

Theorem \ref{thm:hexp-decomp} can be applied to a broad class of partially hyperbolic systems, which includes the Ma\~n\'e examples. This result has not previously appeared elsewhere. We give an outline of the proof. Further details are analogous to the case of the Ma\~n\'e examples, and we emphasize the key new points.

\begin{theorem} \label{thm:phd}
Let $f\colon M\to M$ be a partially hyperbolic diffeomorphism with $TM = E^u \oplus E^c \oplus E^s$.  Assume that $\dim E^c=1$ and that every leaf of the foliations $W^s$ and $W^u$ is dense in $M$.

Let $\phc(x) = \log \|Df|_{E^c(x)}\|$, and given $\mu\in \Mf^e(M)$, let $\lambda^c(\mu) = \int\phc\,d\mu$ be the \emph{center Lyapunov exponent} of $\mu$.  Consider the quantities
\begin{equation}\label{eqn:h+h-}
\begin{aligned}
h^+ &:= \sup \{h_\mu(f) : \mu\in\Mf^e(M), \lambda^c(\mu) \geq 0\}, \\
h^- &:= \sup \{h_\mu(f) : \mu\in\Mf^e(M), \lambda^c(\mu) \leq 0\}. 
\end{aligned}
\end{equation}
Suppose that $h^+ \neq h^-$.  Then $f$ has a unique MME.
\end{theorem}

\begin{remark}
Since $\htop(X,f) = \max(h^+,h^-)$, the condition $h^+ \neq h^-$ is equivalent to the condition that either $h^+ < \htop(X,f)$ or $h^- < \htop(X,f)$. It would be interesting to investigate how typical this condition is. The only way for this condition to fail is if there is an ergodic MME with $\lambda^c=0$, or if there are (at least) two ergodic MMEs for which $\lambda^c$ takes both signs. See \S\ref{sec:ph-pressure} for an interpretation of this condition in terms of topological pressure, and an extension of Theorem \ref{thm:phd} to equilibrium states for nonzero potentials.
\end{remark}

\begin{remark}
For 3-dimensional partially hyperbolic diffeomorphisms homotopic to Anosov, Ures \cite{rU12} showed that there is a unique measure of maximal entropy. In this setting, Crisostomo and Tahzibi \cite{Crisostomo-Tahzibi} gave some interesting criteria for uniqueness (and in some case finiteness) of equilibrium states. We note that our setting is a complementary regime to that of \cite{RHRHTU12}, which assumes compact center leaves, and in which non-uniqueness of the MME is typical.
\end{remark}

First observe that arguments similar to those given for the Ma\~n\'e example in Lemma \ref{lem:ph-ne} and Remark \ref{rmk:back} 
show that $\hexp(f) \leq \min(h^+,h^-)$, so the condition $\hexp(f) < \htop(f)$ is satisfied whenever $h^+\neq h^-$.

\begin{remark}
The upper bound on $\hexp$ for the Ma\~n\'e examples in Lemma \ref{lem:hexp} is actually an upper bound on $h^+$ in that setting, verifying that $h^+(g) < \htop(g)$ whenever the perturbation is small enough.
Moreover, the leaves of $W^u$ are all dense for these examples \cite{PS06}, so Theorem \ref{thm:phd} applies to the Ma\~n\'e examples.
\end{remark}

The rest of the proof of Theorem \ref{thm:phd} consists of finding a decomposition $\Cp,\GGG,\Cs$ for $X\times \NN$ such that $\GGG$ has specification at all scales and $h(\Cp \cup \Cs) < \htop(X,f)$.  We describe the general argument in the case when $h^+ < \htop(f)$, so intuitively, all of the large entropy parts of the system have negative central Lyapunov exponents.

\subsection{A small collection of obstructions}\label{sec:Cp-small}

We take $\Cs=\emptyset$.  To describe $\Cp$, we first observe that the condition $h^+ < \htop(f)$ implies that
\[
\sup \{h_\mu(f) : \mu\in \Mf, \lambda^c(\mu) \geq 0 \} < \htop(f),
\]
where the difference is that now the supremum allows non-ergodic measures as well, and then a weak*-continuity argument gives $r>0$ such that
\begin{equation}\label{eqn:hmu<r}
\sup \{h_\mu(f) : \mu\in \Mf, \lambda^c(\mu) \geq -r \} < \htop(f).
\end{equation}
We can relate the left-hand side of \eqref{eqn:hmu<r} to $h(\Cp)$, where
\[
\Cp := \{(x,n) \in M\times \NN : S_n\phc(x) \geq -rn \}.
\]
One relationship between these was mentioned 
when we bounded $\hexp$ for the Ma\~n\'e example
(though the function being summed there was different).  Here we want to go the other way and obtain an upper bound on $h(\Cp)$.  For this we observe that if we let $E_n\subset \Cp_n$ be any $(n,\eps)$-separated set, $\nu_n$ the equidistributed atomic measure on $E_n$, and $\mu_n = \frac 1n \sum_{k=0}^{n-1} f_*^k \nu_n$, then 
half of the proof of the variational principle \cite[Theorem 8.6]{pW82} 
shows that any limit point of $\mu_n$ is $f$-invariant and has
\[
h_\mu(f) \geq h(\Cp,\eps).
\]
Moreover, $\lambda^c(\mu) = \int\phc\,d\mu(x) \geq -r$ by weak*-convergence and the definition of $\Cp$.  Together with \eqref{eqn:hmu<r}, we conclude that $h(\Cp) < \htop(f)$.

\subsection{A good collection with specification}\label{sec:ph-good}
We now describe a `good' collection of orbit segments $\GGG$, and define a decomposition. To this end, take an arbitrary orbit segment $(x,n)\in M\times \NN$, and remove the longest possible element of $\Cp$ from its beginning.  That is, let $p=p(x,n)$ be maximal with the property that $(x,p) \in \Cp$.  Then we have
\[
S_p \phc(x) \geq -rp
\text{ and }
S_k \phc(x) < -rk \text{ for all } p<k\leq n.
\]
Subtracting the first from the second gives
\[
S_{k-p} \phc(f^px)
= S_k \phc(x) - S_p\phc(x) < -r(k-p),
\]
which we can rewrite as
\[
S_j \phc(f^jx) < -rj \text{ for all } 0\leq j \leq n-p.
\]
In other words, as shown in Figure \ref{fig:Ec-decomposition}, we have\footnote{There is a clear analogy between what we are doing here and the notion of \emph{hyperbolic time} introduced by Alves \cite{jA00}, and developed by Alves, Bonatti and Viana \cite{ABV00}.}
\begin{equation}\label{eqn:ph-G}
(f^px,n-p) \in \GGG := \{(y,m) : S_j \phc(y) < -rj \text{ for all } 0\leq j \leq m\}.
\end{equation}

\begin{figure}[htbp]
\includegraphics{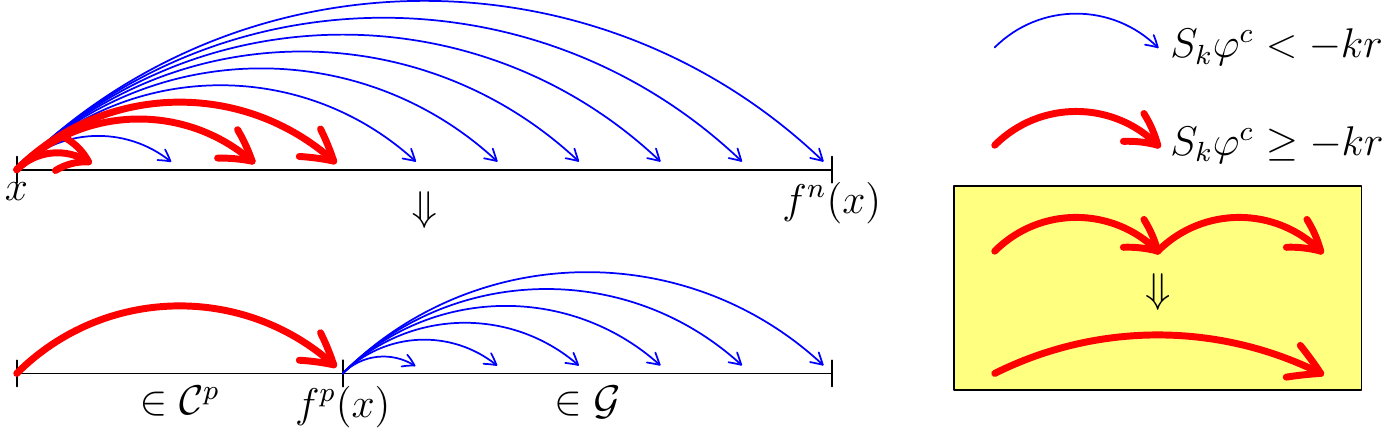}
\caption{A decomposition $\Cp \GGG$ of the space of orbit segments.}
\label{fig:Ec-decomposition}
\end{figure}

Moreover, by choosing $\delta>0$ sufficiently small that $|\phc(y)-\phc(z)| < r/2$ whenever $d(y,z)<\delta$, we see that if $(y,m) \in \GGG$ and $z\in B_m(y,\delta)$, then
\begin{equation}\label{eqn:Dfcs}
\|Df^j|_{E^{cs}(z)}\| \leq e^{-rj/2}
\text{ for all } 0\leq j\leq m.
\end{equation}
This is enough to prove the specification property for $\GGG$.  If $E^{cs}$ is integrable, then one can simply use the proof from the uniformly hyperbolic case verbatim, using \eqref{eqn:Dfcs} to guarantee that
\begin{equation}\label{eqn:Wcs}
W_\delta^{cs}(x) \subset B_n(x,\delta) \text{ whenever }(x,n)\in \GGG.
\end{equation}
Since questions of integrability in partial hyperbolicity can be subtle \cite{RHRHU16}, we point out that one can still establish the specification property without assuming integrability of $E^{cs}$. To do this, fix $\theta>0$ and consider the \emph{center-stable cone}
\[
K^{cs}(x) := \{ v + w : v\in E^{cs}, w\in E^u, \|w\| < \theta \|v\| \} \subset T_x M;
\]
then when establishing the ``one-step specification'' property in \eqref{eqn:one-step}, one can take an \emph{admissible} manifold $W \ni f^{n_2}(x_2)$ that has $T_y W \subset K^{cs}(x)$ at each $y\in W$, and replace $W_\delta^{cs}(x)$ with $f^{-n_2}(W) \cap B(x_2,\delta)$ in the argument. As long as $\theta>0$ is sufficiently small, there will still be enough contraction along $(x_2,n_2)$ for vectors in $K^{cs}$ to guarantee that \eqref{eqn:Wcs} holds.

\section{Unique equilibrium states}\label{sec:eq-st}

For the sake of simplicity, we have so far restricted our attention to measures of maximal entropy. However, the entire apparatus developed above works equally well for equilibrium states associated to ``sufficiently regular'' potential functions.

\subsection{Topological pressure}\label{sec:topological-pressure}

First we recall the notion of \emph{topological pressure}. As with topological entropy in \S\ref{sec:topological-entropy}, we give a more general definition than is standard, defining pressure for collections of orbit segments $\DDD \subset X\times \NN$; our definition reduces to the standard one when $\DDD = X\times \NN$.

\begin{definition}\label{def:pressure}
Given a continuous potential function $\ph\colon X\to \RR$ and a collection of orbit segments $\DDD \subset X\times \NN$, for each $\eps>0$ and $n\in \NN$ we consider the \emph{partition sum}
\begin{equation}\label{eqn:L-pressure}
\Lambda(\DDD,\ph,\eps,n)
:= \sup \Big \{ \sum_{x\in E} e^{S_n \ph(x)} : E \subset \DDD_n \text{ is $(n,\eps)$-separated} \Big \},
\end{equation}
where $S_n\ph(x) = \sum_{k=0}^{n-1} \ph(f^k x)$ is the $n$th Birkhoff sum. The \emph{pressure} of $\ph$ on the collection $\DDD$ at scale $\eps>0$ is
\begin{equation}\label{eqn:P-eps}
P(\DDD,\ph,\eps) := \ulim_{n\to\infty} \frac 1n \log \Lambda(\DDD,\ph,\eps,n),
\end{equation}
and the pressure of $\ph$ on the collection $\DDD$ is
\begin{equation}\label{eqn:P}
P(\DDD,\ph) := \lim_{\eps\to 0} P(\DDD,\ph,\eps).
\end{equation}
As with entropy, in the case when $\DDD = Y\times \NN$ we write $\Lambda(Y,\ph,\eps,n)$, etc.
\end{definition}

The variational principle for topological pressure states that
\begin{equation}\label{eqn:var-princ}
P(X,\ph) = \sup_{\mu \in \MMM_f(X)} \Big( h_\mu(f) + \int \ph\,d\mu \Big).
\end{equation}
A measure that achieves the supremum is called an \emph{equilibrium state} for $(X,f,\ph)$.

As was the case with the MME, there is a standard construction from the proof of the variational principle that establishes existence of an equilibrium state in many cases: we have the following generalization of Proposition \ref{prop:build} and Corollary \ref{cor:existence}.

\begin{proposition}[Building approximate equilibrium states]\label{prop:build-es}
With $X,f,\ph$ as above, fix $\eps>0$, and for each $n\in\NN$, let $E_n\subset X$ be an $(n,\eps)$-separated set.
Consider the Borel probability measures
\begin{equation}\label{eqn:mun-es}
\nu_n := \frac 1{\sum_{x\in E^n} e^{S_n\ph(x)}} \sum_{x\in E_n} \delta_x e^{S_n\ph(x)},
\quad
\mu_n := \frac 1n \sum_{k=0}^{n-1} f_*^k \nu_n
= \frac 1n \sum_{k=0}^{n-1} \nu_n \circ f^{-k}.
\end{equation}
Let $\mu_{n_j}$ be any subsequence that converges in the weak*-topology to a limiting measure $\mu$.  Then $\mu\in \Mf(X)$ and
\begin{equation}\label{eqn:pressure}
h_\mu(f) + \int \ph\,d\mu
\geq \ulim_{j\to\infty} \frac 1{n_j} \log \sum_{x\in E_{n_j}} e^{S_{n_j}\ph(x)}.
\end{equation} 
In particular, for every $\delta>0$ there exists $\mu\in \Mf(X)$ such that $h_\mu(f) + \int \ph\,d\mu \geq P(X,f,\ph,\delta)$.
\end{proposition}
\begin{proof}
See \cite[Theorem 9.10]{pW82}.
\end{proof}

\begin{corollary}\label{cor:existence-es}
Let $X,f$ be as above, and suppose that there is $\delta>0$ such that $P(X,\ph,\delta) = P(X,\ph)$. 
Then there exists an equilibrium state for $(X,f,\ph)$.  Indeed, given any sequence $\{E_n\subset X\}_{n=1}^\infty$ of maximal $(n,\delta)$-separated sets, every weak*-limit point of the sequence $\mu_n$ from \eqref{eqn:mun-es} is an equilibrium state.
\end{corollary}

There is an analogue of Proposition \ref{prop:appeared} for pressure: if $(X,f)$ is expansive at scale $\eps$, then $P(X,\ph,\eps) = P(X,\ph)$, so Corollary \ref{cor:existence-es} establishes existence of an equilibrium state, as well as a way to construct one. Then the goal becomes to prove uniqueness.

\subsection{Regularity of the potential function: the Bowen property}

Even for uniformly hyperbolic systems, one should not expect every continuous potential function to have a unique equilibrium state. Indeed, for the full shift it is possible to show that given any finite set $E$ of ergodic measures, there is a continuous potential function $\ph$ whose set of equilibrium states is precisely the convex hull of $E$; see \cite[p.\ 117]{rI79} and \cite[p.\ 52]{dR78-book}.

For expansive systems $(X,f)$ with specification, uniqueness of the equilibrium state can be guaranteed by the following regularity condition on the potential.

\begin{definition}\label{def:Bowen}
A continuous function $\ph\colon X\to \RR$ has the \emph{Bowen property} at scale $\eps>0$ if there is a constant $V>0$ such that for every $(x,n) \in X\times \NN$ and $y\in B_n(x,\eps)$, we have $|S_n \ph(y) - S_n\ph(x)| \leq V$.
\end{definition}

The following generalization of Theorems \ref{thm:bowen-shift} and \ref{thm:exp-spec} is the full statement of Bowen's original result from \cite{rB75}, with the slight modification that we make the scales explicit.

\begin{theorem}\label{thm:general-bowen}
Let $X$ be a compact metric space and $f\colon X\to X$ a continuous map. Suppose that there are $\eps > 40\delta > 0$ such that $f$ is expansive or positively expansive at scale $\eps$ and has the specification property at scale $\delta$. Then every continuous potential function $\ph\colon X\to \RR$ with the Bowen property at scale $\eps$ has a unique equilibrium state.
\end{theorem}

The proof of Theorem \ref{thm:general-bowen} follows the argument outlined earlier for Theorems \ref{thm:bowen-shift} and \ref{thm:exp-spec} in \S\ref{sec:principles} and \S\ref{sec:revisit}. The main difference is that now the computations involve Birkhoff sums. For example, if we consider the symbolic setting for a moment and recall the motivation from \S\ref{sec:lower-gibbs} for the Gibbs bound as the mechanism for uniqueness, we see that in addition to the use of the Shannon--McMillan--Breiman theorem in \eqref{eqn:smb}, it is natural to use the Birkhoff ergodic theorem and get
\[
h_\mu(\sigma) + \int \ph\,d\mu = \lim_{n\to\infty} \frac 1n \big( -\log \mu [x_{[1,n]}] + S_n\ph(x)\big).
\]
For an equilibrium state, the left-hand side is $P(\ph)$, and this can be rewritten as
$P(\ph) + \lim_{n\to\infty} \frac 1n (\log \mu[x_{[1,n]}] - S_n\ph(x)) = 0$, or equivalently,
\[
\lim_{n\to\infty} \frac 1n \log \Big( \frac{ \mu[x_{[1,n]}] } { e^{-nP(\ph) + S_n\ph(x)}}\Big) = 0.
\]
As with the Gibbs property for the MME, uniqueness of the equilibrium state can be guaranteed by requiring that the quantity inside the logarithm be bounded away from $0$ and $\infty$.\footnote{Observe that this is impossible if $\ph$ does not satisfy the Bowen property.}
Generalizing to arbitrary compact metric spaces by replacing cylinders with Bowen balls, we say that a measure $\mu$ has the \emph{Gibbs property} for a potential $\ph$ at scale $\eps$ if there are constants $K>0$ and $P\in\RR$ such that for every $x\in X$ and $n\in \NN$, we have
\begin{equation}\label{eqn:general-gibbs}
K^{-1} e^{-nP + S_n\ph(x)} \leq \mu(B_n(x,\eps)) \leq K e^{-nP + S_n\ph(x)}.
\end{equation}
If it is known that every equilibrium measure is almost expansive at scale $\eps$ (recall Definition \ref{def:almost-expansive}) -- in particular, if $(X,f)$ is expansive at scale $\eps$ -- and if $\mu$ is an ergodic Gibbs measure for $\ph$, then the analogue of Proposition \ref{prop:uniqueness} holds: we have $P = P(\ph) = h_\mu(f) + \int \ph\,d\mu$, and $\mu$ is the unique equilibrium state for $(X,f,\ph)$. The proof is essentially the same, although now the computations involve Birkhoff sums.

Similarly, in the proof of the uniform counting bounds and the construction of an ergodic Gibbs measure using the procedure in Proposition \ref{prop:build-es}, one encounters multiple steps where a Birkhoff sum $S_n\ph(x)$ must be replaced with $S_n\ph(y)$ for some $y$ in the Bowen ball around $x$, and the Bowen property is required at these steps to guarantee ``bounded distortion'' in the estimates. 

Recalling that topologically transitive locally maximal hyperbolic sets have expansivity and specification, it is natural to ask which potential functions have the Bowen property: how much does Theorem \ref{thm:general-bowen} extend Theorem \ref{thm:classical}?

\begin{proposition}\label{prop:holder-bowen}
If $X$ is a locally maximal hyperbolic set for a diffeomorphism $f$, then every H\"older continuous function $\ph\colon X\to\RR$ has the Bowen property at scale $\eps$, where $\eps$ is the scale of the local product structure.
\end{proposition}
\begin{proof}
Recalling the estimates \eqref{eqn:du} and \eqref{eqn:ds} in the proof of Proposition \ref{prop:uh-exp}, we see that for every $y\in B_n(x,\eps)$ and every $k\in \{0,1,\dots, n-1\}$, we have
\[
d^u(f^k x, f^k y) \leq e^{-\lambda(n-k)} \eps
\quad\text{and}\quad
d^s(f^k x, f^k y) \leq e^{-\lambda k} \eps.
\]
Writing $C$ for the H\"older constant and $\gamma$ for the H\"older exponent, we obtain
\begin{align*}
|\ph(f^k x) - \ph(f^k y)|
&\leq C d(f^k x, f^k y)^\gamma
\leq C \big(2\max (d^u(f^k x, f^k y), d^s(f^k x, f^k y)) \big)^\gamma \\
&\leq C (2\eps)^\gamma \max( e^{-\lambda(n-k)\gamma}, e^{-\lambda k\gamma}),
\end{align*}
and summing over $0\leq k < n$ gives
\begin{align*}
|S_n\ph(x) - S_n\ph(y)| &\leq \sum_{k=0}^{n-1} C (2\eps)^\gamma \max( e^{-\lambda(n-k)\gamma}, e^{-\lambda k\gamma}) \\
&\leq C(2\eps)^\gamma \sum_{k=0}^{n-1} e^{-\lambda\gamma(n-k)} + e^{-\lambda\gamma k}
\leq 2C(2\eps)^\gamma \sum_{k=0}^\infty e^{-\lambda\gamma k} =:V.
\end{align*}
This last quantity is finite and independent of $x,y,n$, which establishes the Bowen property for $\ph$.
\end{proof}

\begin{remark}
The theorem ``H\"older potentials for uniformly hyperbolic systems have unique equilibrium states'' is well-entrenched enough that it is worth stressing the following point: it is the dynamical Bowen property (bounded distortion), rather than the metric H\"older property, that is truly important here. In particular, if we consider a non-uniformly hyperbolic system that is conjugate to a uniformly hyperbolic one, such as the Manneville--Pomeau interval map or Katok map of the torus, then every potential with the Bowen property continues to have a unique equilibrium state, but there may be H\"older potentials with multiple equilibrium states. However, determining which potentials have the Bowen property may be a nontrivial task.
\end{remark}

\subsection{The most general discrete-time result}

Recalling the weakened versions of expansivity and specification used in Theorem \ref{thm:hexp-decomp}, it is natural to ask for a uniqueness result for equilibrium states that uses a weakened version of the Bowen property. Observe that the Bowen property can be formulated for a collection of orbit segments (rather than the entire system) by replacing $X\times \NN$ in Definition \ref{def:Bowen} with $\GGG \subset X\times \NN$.

\begin{definition}\label{def:Bowen-G}
A continuous function $\ph\colon X\to \RR$ has the \emph{Bowen property} at scale $\eps>0$ on a collection of orbit segments $\GGG \subset X\times \NN$ if there is a constant $V>0$ such that for every $(x,n) \in \GGG$ and $y\in B_n(x,\eps)$, we have $|S_n \ph(y) - S_n\ph(x)| \leq V$.
\end{definition}

To formulate our most general discrete-time result on uniqueness of equilibrium states, we replace the entropy of obstructions to expansivity from Definition \ref{def:hexp} with the \emph{pressure of obstructions to expansivity at scale $\eps$}:
\begin{equation}\label{eqn:Pexp}
\Pexp(\phi,\eps) := \sup \Big \{ h_\mu(f) + \int \ph\,d\mu : \mu \in \MMM_f^e(X) \text{ and } \mu(\NE(\eps)) > 0 \Big \}.
\end{equation}

\begin{theorem}[{\cite[Theorem 5.6]{CT16}}]\label{thm:general-f}
Let $X$ be a compact metric space, $f\colon X\to X$ a homeomorphism, and $\ph\colon X\to \RR$ a continuous potential function. Suppose that there are $\eps > 40\delta > 0$ such that $\Pexp(\ph,\eps) < P(\ph)$ and there exists a decomposition $(\Cp,\GGG,\Cs)$ for $X\times \NN$ with the following properties:
\begin{enumerate}[label=\upshape{(\Roman{*})}]
\item every collection $\GGG^M$ has specification at scale $\delta$,
\item $\ph$ has the Bowen property on $\GGG$ at scale $\eps$, and
\item\label{pressure-gap} $P(\Cp \cup \Cs,\ph,\delta) < P(\ph)$.
\end{enumerate}
Then $(X,f,\ph)$ has a unique equilibrium state.
\end{theorem}

\begin{remark}\label{rmk:bowen-G}
In applications to non-uniformly hyperbolic systems, it is very often the case that there is a natural collection of orbit segments $\GGG$ along which the dynamics is uniformly hyperbolic; this is the most common way of establishing specification for $\GGG$, as we saw in \S\ref{sec:DA}. In this case the proof of Proposition \ref{prop:holder-bowen} shows that every H\"older potential $\ph$ has the Bowen property on $\GGG$. Then the question of uniqueness boils down to determining which H\"older potentials have the pressure gap properties \ref{pressure-gap} and $\Pexp(\ph,\eps) < P(\ph)$. It is often the case that one or both of these conditions fails for some H\"older potentials, as in the Manneville--Pomeau example.
\end{remark}

\subsection{Partial hyperbolicity}\label{sec:ph-pressure}

For partially hyperbolic systems with one-dimensional center as in \S\ref{sec:ph}, Theorem \ref{thm:general-f} can be used to extend Theorem \ref{thm:phd}.

\begin{theorem}\label{thm:ph-es}
Let $M,f,\phc$ be as in Theorem \ref{thm:phd}. Given a H\"older continuous potential function $\ph \colon M\to \RR$, consider the quantities
\begin{align*}
P^+ &:= \sup \Big \{ h_\mu(f) + \int \ph \,d\mu : \mu \in \MMM_f^e(M), \lambda^c(\mu) \geq 0 \Big \}, \\
P^- &:= \sup \Big \{ h_\mu(f) + \int \ph \,d\mu : \mu \in \MMM_f^e(M), \lambda^c(\mu) \leq 0 \Big \}.
\end{align*}
If $P^+ \neq P^-$, then $(M,f,\ph)$ has a unique equilibrium state.
\end{theorem}

Beyond the properties from \S\ref{sec:ph}, the only additional ingredient required for Theorem \ref{thm:ph-es} is the fact that $\ph$ has the Bowen property on the collection of orbit segments $\GGG$ defined in \eqref{eqn:ph-G}, which follows from Remark \ref{rmk:bowen-G} and the hyperbolicity estimate in \eqref{eqn:Dfcs}; then uniqueness follows from Theorem \ref{thm:general-f}.

It is worth noting that the condition $P^+ \neq P^-$ (and thus the condition $h^+ \neq h^-$) can be formulated in terms of the topological pressure function. The function $t\mapsto P(\ph + t\phc)$ is convex, being the supremum of the affine functions
\[
P_\mu \colon t \mapsto h_\mu(f) + \int \ph \,d\mu + t\lambda^c(\mu)
\]
over all $\mu \in \MMM_f^e(M)$. Some of its possible shapes are shown in Figure \ref{fig:possible-graphs}. 

\begin{figure}[htbp]
\includegraphics[width=\textwidth]{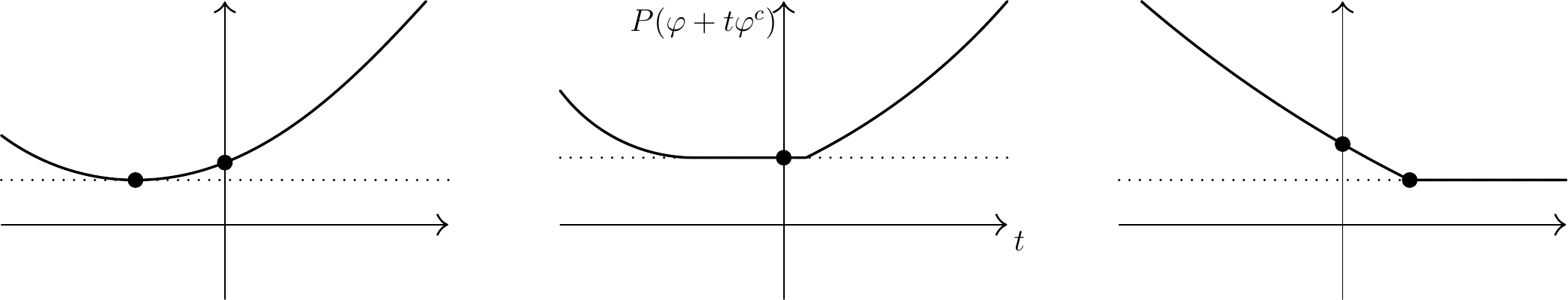}
\caption{Some possible graphs of $t\mapsto P(\ph+t\phc)$.}
\label{fig:possible-graphs}
\end{figure}

Suppose there is $t>0$ such that $P(\ph + t\phc) < P(\ph)$, as in the third graph in Figure \ref{fig:possible-graphs}. Then given any $\mu \in \MMM_f^e(M)$ with $\lambda^c(\mu) \geq 0$, we have
\begin{equation}\label{eqn:Pmu-leq}
h_\mu(f) + \int\ph\,d\mu = P_\mu(0) \leq P_\mu(t) \leq P(\ph + t\phc) < P(\ph),
\end{equation}
and taking a supremum over all such $\mu$ gives $P^+ \leq P(\ph + t\phc) < P(\ph)$, so that the condition of Theorem \ref{thm:ph-es} is satisfied and $(M,f,\ph)$ has a unique equilibrium state, which has negative center Lyapunov exponent. 

A similar argument holds if there is $t<0$ such that $P(\ph + t\phc) < P(\ph)$, as in the first graph in Figure \ref{fig:possible-graphs}; \eqref{eqn:Pmu-leq} applies to all $\mu \in \MMM_f^e(M)$ with $\lambda^c(\mu) \leq 0$, so that $P^- < P(\ph) = P^+$, and there is a unique equilibrium state, which has positive center Lyapunov exponent.

We see that the only way to have $P^+ = P^-$ is if the function $t\mapsto P(\ph + t\phc)$ has a global minimum at $t=0$. Thus one could restate the last line of Theorem \ref{thm:ph-es} as the conclusion that $(M,f,\ph)$ has a unique equilibrium state if there is $t\neq 0$ such that $P(\ph + t\phc) < P(\ph)$. In particular, returning to Theorem \ref{thm:phd},  $f$ has a unique MME if there is $t\neq 0$ such that $P(t\phc) < P(0) = \htop(f)$.

\part{Geodesic flows}\label{part:3}
In this part, we focus on our geometric applications. In \S \ref{sec:geometry}, we introduce some geometric background, and in \S\ref{sec:geodesic-ES} we describe the main results and some of the key ideas from the paper \cite{BCFT18}. In \S\ref{sec:Kproperty}, we discuss our approach to the Kolmogorov $K$-property. In \S\ref{sec:entropygap}, we give the main ideas of proof  for the ``pressure gap'' for a wide class of potentials for  geodesic flow on a rank 1 non-positive curvature manifold.

\section{Geometric preliminaries}\label{sec:geometry}

\subsection{Overview}

Let $M= (M^n,g)$ be a closed connected $C^\infty$ Riemannian manifold with dimension $n$, and $F = (f_t)_{t\in\mathbb{R}}$ denote the geodesic flow on the unit tangent bundle $X=T^1M$. The geodesic flow is defined by picking a point and a direction (i.e. an element of $T^1M$), and walking at unit speed along the geodesic determined by that data. More precisely, $f_t(v) = \dot{c}_v(t)$, where $c_v\colon \RR\to M$ is the unique unit speed geodesic with $\dot{c}_v(0) = v$. Geodesic flows are of central importance in the theory of dynamical systems, and encode many important features of the geometry and topology of the underlying manifold $M$. For general background on geodesic flows, we refer to \cite{jL18, BG05}. 

If all sectional curvatures of $M$ are negative at every point, then $F$ is a transitive Anosov flow. In particular, the thermodynamic formalism is very well understood. To go beyond negative curvature, one generally needs the tools of non-uniform hyperbolicity. There are three further classes of manifolds that generally exhibit some kind of non-uniformly hyperbolic behaviour: nonpositive curvature; no focal points; and no conjugate points.  The relationships are as follows: 
\[
\text{negative curv.}
\Rightarrow
\text{nonpositive curv.}
\Rightarrow
\text{no focal points}
\Rightarrow
\text{no conjugate points}.
\]
The reverse implications all fail in general.

The definition of nonpositive curvature is easy: all sectional curvatures are $\leq 0$ at every point.  No focal points and no conjugate points are defined in terms of Jacobi fields, which we will introduce shortly, but can be understood in terms of the growth of distance between geodesics which pass through the same point. If we work in the universal cover $\wM$ and consider arbitrary geodesics $c_1,c_2$ with $c_1(0) = c_2(0)$, then non-positive curvature implies that $t\mapsto d(c_1(t), c_2(t))$ is convex, while no focal points is equivalent to the condition that $t\mapsto d(c_1(t), c_2(t))$ be nondecreasing for all such $c_1,c_2$, and no conjugate points is equivalent to the condition that this function never vanish for $t>0$; in other words, there is at most one geodesic connecting any two points in $\wM$. In \S \ref{s.cat}, we will also briefly discuss  geodesic flow on some classes of spaces beyond the Riemannian case: namely, $\CAT(-1)$ spaces (which generalize negative curvature) and $\CAT(0)$ spaces (which generalize non-positive curvature).

For intuition, negative curvature has the effect of spreading out geodesics which pass through the same point (think of a saddle), while positive curvature has the effect of bringing them back together after a finite amount of time (think of a sphere). As described in \cite{rG75}, one can imagine starting with a negatively curved surface and then ``raising a bump of positive curvature''; at first the positive curvature effect is weak enough that the geodesic flow remains Anosov, but eventually the Anosov property is destroyed, and raising the bump far enough creates conjugate points.

In these notes, we focus on the case of equilibrium states for manifolds with nonpositive curvature using specification-based techniques as in \cite{BCFT18}; this relies on a continuous-time version of Theorem \ref{thm:general-f}, which we formulate in \S\ref{sec:flow-unique}.
This approach has been extended to manifolds without focal points by Chen, Kao, and Park \cite{CKP, CKP-2}\footnote{Another specification-based proof of uniqueness of the MME on surfaces without focal points was given by Gelfert and Ruggiero \cite{GR}}.  We also state and sketch recent results by the first-named author, Knieper and War for the MME to surfaces with no conjugate points, and survey some relevant recent results for $\CAT(-1)$ and $\CAT(0)$ spaces.

In the remainder of this section we collect some geometric preliminaries. Some of the definitions are taken verbatim from \cite{BCFT18} for notational consistency. For more details, we recommend recent works \cite{BCFT18, GS14}, and more classical references \cite{wB95, pE99, Ebe:96}.

\subsection{Surfaces}\label{sec:surfaces}

For purposes of exposition, we will often think about the surface case $n=2$, although our approach applies in higher dimension too. By the Gauss--Bonnet theorem, the sphere has no metric of nonpositive curvature, and the only such metrics on the torus are flat everywhere; it can be easily verified that the corresponding geodesic flows have zero topological entropy and are not topologically transitive. Thus we are interested in studying surfaces of genus at least $2$.

As a first example, we can think about a surface of genus $2$ with an embedded flat cylinder, and negative curvature elsewhere. We could also consider the case where the flat cylinder collapses to a single closed geodesic on which the curvature vanishes, with strictly negative curvature elsewhere.  In higher dimensions, much more complicated examples exist, such as the $3$-dimensional \emph{Gromov example} that we describe in \S\ref{sec:entropygap}.

Geodesic flow in non-positive curvature is a primary example of non-uniform hyperbolicity. The basic example of a surface containing a flat cylinder illustrates the primary difficulty: the co-existence of trajectories displaying hyperbolic behavior (geodesics in the negatively curved part of the surface) with trajectories displaying non-hyperbolic behavior (geodesics in the flat cylinder). More precisely, given a surface $M$ of genus at least $2$ with non-positive curvature, we let $K\colon M\to (-\infty,0]$ be the Gaussian curvature, and $\pi\colon T^1M \to M$ the natural projection of a tangent vector to its footpoint. Then we define the \emph{singular set} to be
\begin{equation}\label{eqn:sing}
\Sing := \{v \in T^1M : K(\pi(f_t v)) = 0 \text{ for all } t\in \RR \}.
\end{equation}
That is, $\Sing$ is the set of $v$ for which the corresponding geodesic $\gamma_v$ experiences $0$ curvature for all time. All other vectors are called \emph{regular}:
\begin{equation}\label{eqn:reg}
\Reg := T^1M \setminus \Sing = \{v\in T^1M : K(\pi(f_t v)) < 0 \text{ for some } t\in \RR \}.
\end{equation}
Although the negative curvature encountered along regular geodesics guarantees some expansion/contraction, this may be arbitrarily weak because the geodesic can be arranged to experience $0$ curvature for a long time (e.g., wrapping round an embedded flat cylinder) before hitting any negative curvature.

The set $\Sing$ is closed and flow-invariant, while the set $\Reg$ is open. The regular set is nonempty because $M$ has genus at least $2$, and in fact $\Reg$ is dense in $T^1M$.

In higher dimensions one has a similar dichotomy between singular and regular vectors, which we will describe in the next section. This gives a partition of $T^1M$ as $\Reg \sqcup \Sing$, where $\Sing$ is closed and flow-invariant. As with surfaces, we will restrict our attention to the case when $\Reg \neq \emptyset$; this \emph{rank 1} assumption rules out examples such as direct products, and is the typical situation, as demonstrated by the higher rank rigidity theorem of Ballmann and Burns--Spatzier \cite{wB85,BS87a,BS87b}.

\subsection{Invariant foliations via horospheres}

Now let the dimension of $M$ be any $n\geq 2$. We describe invariant stable and unstable foliations $W^s$ and $W^u$ of $X = T^1M$ that are tangent to invariant subbundles $E^s$ and $E^u$ in $TX = TT^1M$ along which we will eventually obtain the contraction and expansion estimates necessary to study uniqueness of equilibrium states.

We must be a little careful in defining these foliations: we cannot ask that $W^s(v)$ is the set of $w\in T^1M$ so that $d(f_tv, f_tw) \to 0$ as $t \to \infty$ like we can in the uniformly hyperbolic setting. We must allow points that stay bounded distance apart (in the universal cover) for all forward time. However, this does not work as the definition of $W^s$ because it does not distinguish the stable from the flow direction. To do things properly, there are two approaches.
\begin{itemize}
\item \emph{Local approach:} Use stable and unstable orthogonal Jacobi fields to define $E^s$ and $E^u$ locally; see \S\ref{sec:jacobi-fields} below.
\item \emph{Global approach:}
Define stable and unstable horospheres $H^s$ and $H^u$ in the universal cover $\wM$ (this is typically done using Busemann functions) and use these to get $W^s,W^u$.
\end{itemize}
We outline this second approach here. Given $v\in T^1M$, let $\tilde v \in T^1\wM$ be a lift of $v$, and construct $H^s(\tilde v)$ as follows: for each $r>0$ let
\[
S^r(\tilde v, +) = \{ x \in \wM : d_{\wM}(x, \pi(f_r \tilde v)) = r\}
\]
denote the set of points at distance $r$ from $\pi(f_r \tilde v) = c_{\tilde v}(r)$, and let $H^s(\tilde v)$ be the limit of $S^r(v,+)$ as $r\to\infty$. 
This defines a hypersurface that contains the point $\pi \tilde v$. Writing $W^s(\tilde v)$ for the unit normal vector field to $H^s(\tilde v)$ on the same side as $\tilde v$, the stable manifold $W^s(v)$ is the image of $W^s(\tilde v)$ under the canonical projection $T^1 \wM\to T^1M$.

The unstable horosphere $H^u(\tilde v)$ and the unstable manifold $W^u(v)$ are defined analogously, replacing $S^r(\tilde v, +)$ with
\[
S^r(\tilde v, -) = \{ x \in \wM : d_{\wM}(x, \pi(f_{-r} \tilde v)) = r\}.
\] 
The horospheres are $C^2$ manifolds, so $W^s(v)$ and $W^u(v)$ are $C^1$ manifolds, and we can define  the stable and unstable subspaces $E^s(v), E^u(v) \subset T_vT^1 M$ to be the tangent spaces of $W^s(v), W^u(v)$ respectively.  The bundles $E^s, E^u$, which are both globally defined in this way, are respectively called the stable and unstable bundles. 
They are invariant and depend continuously on $v$; see \cite{pE99, GW99}.  

The following is equivalent to the standard definition of the regular set via Jacobi fields, which we will give in the next section.

\begin{definition}\label{def:sing-reg} A vector $v\in T^1M$ is \emph{regular} if $E^s(v) \cap E^u(v)$ is trivial (contains only the $0$ vector in $T_v T^1M$), and \emph{singular} otherwise. Write $\Reg \subset T^1M$ for the set of regular vectors, and $\Sing \subset T^1M$ for the set of singular vectors.
\end{definition}

On $\Reg$, we obtain the expected splitting $T_vT^1M = E^s(v) \oplus E^u(v) \oplus E^c(v)$, where $E^c(v)$ is the flow direction. This splitting degenerates on $\Sing$.

\begin{definition}\label{def:rank-1}
The manifold $M$ is \emph{rank 1} if $\Reg \neq \emptyset$.
\end{definition}

Finally, we define a function which is of great importance in thermodynamic formalism. The \emph{geometric potential} is the function that measures infinitesimal volume growth in the unstable distribution:
\[
\ph^u(v)=-\lim_{t\to 0} \frac{1}{t}\log \det(df_t|_{E^u(v)})
= -\frac d{dt}\Big|_{t=0} \log \det(df_t|_{E^u(v)}).
\]
The potential $\ph^u$ is continuous and globally defined. When $M$ has dimension $2$, $\ph^u$ is H\"older along unstable leaves \cite{GW99}. It is not known whether $\ph^u$ is H\"older along stable leaves.  In higher dimensions, it is not known whether $\ph^u$ is H\"older continuous on either stable or unstable leaves. An advantage of our approach is that we sidestep the question of H\"older regularity for $\ph^u$.

\subsection{Jacobi fields and local construction of stables/unstables}\label{sec:jacobi-fields}

Now we give an alternate description of the stable and unstable subbundles and foliations, which can be shown to agree with the definitions in the previous section.

A \emph{Jacobi field} along a geodesic $\gamma$ is a vector field along $\gamma$ obtained by taking a one-parameter family of geodesics that includes $\gamma$ and differentiating in the parameter coordinate; equivalently, it is a vector field along $\gamma$ satisfying
\begin{equation}\label{eqn:Jacobi}
J''(t) + R(J(t), \dot{\gamma}(t))\dot{\gamma}(t)=0,
\end{equation}
where $R$ is the Riemannian curvature tensor on $M$ and $'$ represents covariant differentiation along $\gamma$.

We often  want to remove the variations through geodesics in the flow direction from consideration. If $J(t)$ is a Jacobi field along a geodesic $\gamma$ and both  $J(t_0)$ and $J'(t_0)$ are orthogonal to $\dot{\gamma}(t_0)$ for some $t_0$, then $J(t)$ and $J'(t)$ are orthogonal to $\dot{\gamma}(t)$ for all $t$.  Such a Jacobi field is an \emph{orthogonal Jacobi field}.

A Jacobi field $J(t)$ along a geodesic $\gamma$ is \emph{parallel at $t_0$} if $J'(t_0)=0$. A Jacobi field $J(t)$ is parallel if it is parallel for all $t \in \RR$.

\begin{definition}\label{def:sing-reg-2}
A geodesic $\gamma$ is \emph{singular} if it admits a nonzero parallel orthogonal Jacobi field, and \emph{regular} otherwise.
\end{definition}

If $\gamma$ is singular in the sense of Definition \ref{def:sing-reg-2}, then every $\dot\gamma(t) \in T^1M$ is singular in the sense of Definition \ref{def:sing-reg}, and similarly for regular.

We  write $\JJJ(\gamma)$ for the space of orthogonal Jacobi fields for $\gamma$; given $v\in T^1 M$ there is a natural isomorphism $\xi \mapsto J_\xi$ between $T_vT^1M$ and $\JJJ(\gamma_v)$, which has the property that
\begin{equation} \label{compare}
\|df_t(\xi)\|^2= \|J_\xi(t)\|^2+\|J'_\xi(t)\|^2.
\end{equation}
An orthogonal Jacobi field $J$ along a geodesic $\gamma$ is \emph{stable} if $\|J(t)\|$ is bounded for $t\geq 0$, and \emph{unstable} if it is bounded for $t\leq 0$.  The stable and the unstable Jacobi fields each form linear subspaces of 
$\JJJ(\gamma)$, which we denote by $\JJJ^s(\gamma)$ and $\JJJ^u(\gamma)$, respectively.
The corresponding stable and unstable subbundles of $TT^1M$ are
\begin{align*}
E^u(v)&=\{ \xi \in T_v(T^1M) : J_\xi \in \JJJ^u(\gamma_v) \}, \\
E^s(v)&=\{ \xi \in T_v(T^1M) : J_\xi \in \JJJ^s(\gamma_v) \}.
\end{align*}
The bundle $E^c$ is spanned by the vector field that generates the flow $F$.  We also write $E^{cu} = E^c\oplus E^u$ and $E^{cs} = E^c\oplus E^s$.
The subbundles have the following properties (see \cite{pE99} for details):
\begin{itemize}
\item $\dim(E^u)=\dim(E^s)= n-1$, and $\dim(E^c)=1$;
\item the subbundles are invariant under the geodesic flow;
\item the subbundles depend continuously on $v$, see \cite{pE99, GW99};
\item $E^u$ and $E^s$ are both orthogonal to $E^c$;
\item $E^u$ and $E^{s}$ intersect non-trivially if and only if $v \in \Sing$;
\item $E^\sigma$ is integrable to a foliation $W^\sigma$ for each $\sigma\in \{u,s,cs,cu\}$.
\end{itemize}
It is proved in \cite[Theorem 3.7]{wB82} that the foliation $W^s$ is minimal in the sense that $W^s(v)$ is dense in $T^1M$ for every $v \in T^1M$. Analogously, the foliation $W^u$ is also minimal.

\section{Equilibrium states for geodesic flows}\label{sec:geodesic-ES}

\subsection{The general uniqueness result for flows}\label{sec:flow-unique}

We recall the general definitions of topological pressure, variational principle, and equilibrium states for flows, which are analogous to the discrete-time definitions from \S\ref{sec:topological-pressure}.

Given a compact metric space $X$ and a continuous flow $F = (f_t)$ on $X$, we write $\MMM_F(X) = \bigcap_{t\in \RR} \MMM_{f_t}(X)$ for the space of flow-invariant Borel probability measures on $X$, and $\MMM_F^e(X) \subset \MMM_F(X)$ for the set of ergodic measures.

For $\epsilon>0$, $t>0$, and $x\in X$, the \emph{Bowen ball of radius $\epsilon$ and order $t$} is
\[
B_t(x,\epsilon)= \{ y\in X\mid d(f_s x, f_s y)<\epsilon\text{ for all }0\leq s\leq t\}.
\]
A set $E\subset X$ is \emph{$(t, \epsilon)$-separated} if for all distinct $x, y\in E$ we have $y\notin \overline{B_t(x, \epsilon)}$.

Given a continuous potential function $\ph\colon X\to \RR$, we write $\Phi(x,t) = \int_{0}^{t} \varphi(f_sx)\,ds$ for the integral of $\varphi$ along an orbit segment of length $t$.  We interpret $\DDD \subset X\times [0,\infty)$ as a collection of finite-length orbit segments by identifying $(x,t)$ with the orbit segment starting at $x$ and lasting for time $t$. Writing $\DDD_t := \{x\in X : (x,t) \in \DDD\}$, the partition sums associated to $\DDD$ and $\ph$ are
\begin{equation}\label{eqn:Lambda-sep}
\Lambda(\DDD,\varphi,\epsilon, t) = \sup
\Big\{ \sum_{x\in E} e^{\Phi(x, t)} : E\subset \DDD_t \text{ is $(t,\epsilon)$-separated} \Big\}.
\end{equation}
The pressure of $\ph$ on the collection $\DDD$ is given by \eqref{eqn:P-eps}--\eqref{eqn:P}, replacing $n$ with $t$:
\[
P(\DDD,\ph) = \lim_{\eps\to 0} P(\DDD,\ph,\eps),
\qquad
P(\DDD,\ph,\eps) = \ulim_{t\to\infty} \frac 1t \log \Lambda(\DDD,\ph,\eps,t).
\]
We continue to write $P(Y,\ph) = P(Y\times [0,\infty),\ph)$ for $Y\subset X$, and often abbreviate $P(\ph) = P(X,\ph)$.
The \emph{variational principle for pressure} states that
\[
P(\ph) = \sup_{\mu \in \MMM_F(X)} \Big( h_\mu(f_1) + \int \ph\,d\mu \Big).
\]
A measure that achieves the supremum is an \emph{equilibrium state} for $(X,f,\ph)$. When $\ph=0$, we recover the topological entropy $h(F)$, and an equilibrium state for $\varphi=0$ is called a \emph{measure of maximal entropy}.

\begin{remark}\label{rmk:usc}
As in the discrete-time case, if the entropy map $\mu \mapsto h_\mu$ is upper semi-continuous then equilibrium states exist for each continuous potential function.
Geodesic flows in non-positive curvature are entropy-expansive due to the flat strip theorem \cite{gK98}; this guarantees upper semi-continuity and thus existence.
\end{remark}

In light of Remark \ref{rmk:usc}, the real question is once again uniqueness. Our main tool will be a continuous-time analogue of Theorem \ref{thm:general-f}, which gives non-uniform versions of specification, expansivity, and the Bowen property that are sufficient to give uniqueness.

The main novelty compared with the discrete-time case is the expansivity condition.
For an expansive map, the set of points that stay close to $x$ for all time is only the point $x$ itself.
For an expansive flow, this set is an orbit segment of $x$. Our set of non-expansive points for a flow is defined accordingly. For $x\in X$ and $\epsilon>0$, we let the \emph{bi-infinite Bowen ball} be
\[
\Gamma_\epsilon(x)=\{y\in X \, :\, d(f_tx, f_ty)\leq \epsilon \text{ for all }t\in\mathbb{R}\}.
\]
The \emph{set of non-expansive points at scale $\eps$} is (compare this to Definition \ref{def:almost-expansive})
\begin{equation}\label{eqn:NE-flow}
\NE(\epsilon, F):=\{ x\in X \mid \Gamma_\epsilon(x)\not\subset  f_{[-s,s]}(x) \text{ for any }s>0 \},
\end{equation}
where $f_{[a,b]}(x) = \{f_tx : a \leq t \leq b\}$.\footnote{We note that the original formulation of expansivity for flows by Bowen and Walters \cite{BW72} allows reparametrizations, which suggests that one might consider a potentially larger set in place of $\Gamma_\eps$ for expansive flows. The main motivation for allowing reparametrizations is to give a definition that is preserved under orbit equivalence. However, this is not relevant for our purposes.  
In our setup, the natural notion of expansivity would be to ask that there exists $\epsilon$ so that $\NE(\epsilon, \FFF)= \emptyset$. This definition is sufficient for the uniqueness results, and strictly weaker than Bowen--Walters expansivity, although it is not an invariant under orbit equivalence. See the discussion of \emph{kinematic expansivity} in \cite{FH19}.}
The \emph{pressure of obstructions to expansivity} is
\[
\Pexp(\varphi):=\lim_{\eps\to 0} P^\perp_{\mathrm{exp}}(\varphi, \eps),
\]
where
\[
\Pexp(\varphi, \epsilon)=\sup_{\mu\in \MMM^e_F(X)}\Big\{
h_\mu(f_1) + \int\varphi\, d\mu\, :\, \mu(\NE(\eps, \FFF))=1\Big\}.
\]

\begin{remark}\label{rmk:Pexp-sing}
For rank $1$ geodesic flow, a simple argument using the flat strip theorem guarantees that $\NE(\eps,F) \subset \Sing$, so we have $\Pexp(\varphi) \leq P(\Sing, \ph)$.
\end{remark}

Our definitions of specification and the Bowen property are completely analogous to Definitions \ref{def:spec} and \ref{def:Bowen-G} from the discrete-time case. The specification property for flows was defined by Bowen in \cite{rB72}, and was used to prove uniqueness of equilibrium states by Franco \cite{eF77}.

\begin{definition}\label{def:spec-flow}
A collection of orbit segments $\GGG \subset X\times [0,\infty)$ \emph{has the specification property at scale $\delta>0$} if there exists $\tau>0$ such that for every $(x_1,t_1), \dots, (x_k, t_k) \in \GGG$, there exist $0=T_1 < T_2 < \cdots < T_k$ and $y\in X$ such that $f^{T_i}(y) \in B_{t_i}(x_i,\delta)$ for all $i$, and moreover, writing $s_i = T_i + t_i$, we have $s_i \leq T_{i+1} \leq s_i + \tau$ for all $i$.

We say that $\GGG$ \emph{has the specification property} if it has the specification property at scale $\delta$ for every $\delta>0$.
\end{definition}

\begin{definition}\label{def:Bowen-flow}
A continuous function $\ph\colon X\to \RR$ has the \emph{Bowen property at scale $\eps>0$} on a collection of orbit segments $\GGG \subset X\times [0,\infty)$ if there is $V>0$ such that for every $(x,t) \in \GGG$ and $y\in B_t(x,\eps)$, we have $|\Phi(y,t) - \Phi(x,t)| \leq V$.

We say that $\ph$ has the \emph{Bowen property on $\GGG$} if there exists $\eps>0$ such that $\ph$ has the Bowen property at scale $\eps$ on $\GGG$.
\end{definition}

An argument following the proof of Proposition \ref{prop:holder-bowen} shows that for uniformly hyperbolic flows, any H\"older continuous function has the Bowen property. More generally, Remark \ref{rmk:bowen-G} applies here as well: if the flow is uniformly hyperbolic along a collection of orbit segments $\GGG \subset X\times [0,\infty)$, then every H\"older $\ph$ has the Bowen property on $\GGG$.

As in Definition \ref{def:decomp-2} for discrete time, a \emph{decomposition for $X\times [0, \infty)$} consists of three collections $\PPP, \GGG, \SSS\subset X\times [0, \infty)$ for which there exist three functions $p, g, s\colon X\times [0, \infty)\rightarrow [0, \infty)$ such that for every $(x,t)\in X\times [0, \infty)$, the values $p=p(x,t)$, $g=g(x,t)$, and $s=s(x,t)$ satisfy $t=p+g+s$, and 
\[
(x,p)\in \PPP,\quad
(f_p(x), g)\in \GGG,\quad
(f_{p+g}(x), s)\in \SSS.
\]
The conditions we are interested in depend only on the collections $(\PPP, \GGG, \SSS)$ rather than the functions $p$, $g$, $s$. However, we work with a fixed choice of $(p,g,s)$ for the proof of the abstract theorem to apply.

One small difference from the discrete-time case is that we need to ``fatten up'' $\PPP$ and $\SSS$ slightly before imposing the smallness condition in the general uniqueness theorem. To this end, for a collection $\DDD \subset X\times [0,\infty)$, we define
\[
[\mathcal{D}]:=
\{(x,k)\in X\times \mathbb{N} : (f_{-s}x, k+s+t)\in \mathcal{D}\textrm{ for some }s, t\in [0,1]\}.
\]

\begin{theorem}[Non-uniform Bowen hypotheses for flows \cite{CT16}] \label{thm:ESflows}
Let $(X, F)$ be a continuous flow on a compact metric space, and $\varphi\colon X\rightarrow \mathbb{R}$ be a continuous potential function.  Suppose that $\Pexp(\varphi)<P(\varphi)$ and $X\times [0, \infty)$ admits a decomposition $(\PPP, \GGG, \SSS)$ with the following properties:
\begin{enumerate}[label=\textup{{(\Roman{*})}}]
\item $\GGG$ has specification;\label{cond:spec}
\item $\varphi$ has the Bowen property on $\GGG$;
\item $P([\PPP]\cup [\SSS], \varphi)<P(\varphi).$\label{cond:P-gap}
\end{enumerate}
Then $(X, F, \varphi)$ has a unique equilibrium state $\mu_\ph$.
\end{theorem}

\begin{remark}
The reason that in general we control the pressure of $[\PPP]\cup[\SSS]$ rather than the collection $\PPP \cup \SSS$ is a consequence of a technical step in the proof of the abstract result in \cite{CT16} that required a passage from continuous to discrete time. This distinction does not matter for the \emph{$\lambda$-decompositions} described in the next section, which cover all the applications we discuss here; see \cite[Lemma 3.5]{CT19}.
\end{remark}

\subsection{Geodesic flows in non-positive curvature}

Now we return to the specific setting of geodesic flow in non-positive curvature. In \S\ref{sec:no-gap} we explain why the outcome from the uniformly hyperbolic situation -- a unique equilibrium state, whose support is all of $X=T^1M$ -- cannot occur unless there is a \emph{pressure gap} $P(\Sing,\ph) < P(\ph)$. In \S\ref{sec:pressure-gap} we formulate the main results on uniqueness given a pressure gap, ergodic properties of the unique equilibrium state, and how often the pressure gap occurs. In \S\ref{sec:geod-per} we describe how the notion of periodic orbit equidistribution from \S\ref{sec:periodic-points} is adapted to this setting. The proof of the uniqueness result uses Theorem \ref{thm:ESflows} and is outlined in \S\ref{sec:bcftproofidea}. The proofs regarding ergodic properties, particularly the Kolmogorov property, are described later in \S\ref{sec:Kproperty}, and the pressure gap itself is discussed in \S\ref{sec:entropygap}.

\subsubsection{Uniqueness can fail without a pressure gap}\label{sec:no-gap}

For uniformly hyperbolic flows and H\"older continuous potentials, there is a unique equilibrium state, and this equilibrium state gives positive weight to every open set; it is \emph{fully supported}. For geodesic flow in nonpositive curvature, this conclusion cannot hold unless there is a \emph{pressure gap}, which we now describe.

Since the singular set $\Sing$ is closed and flow-invariant, we can apply the variational principle to the restriction of the flow to $\Sing$, and obtain
\[
P(\Sing, \varphi)=\sup \Big\{ h_{\mu}(f_1) +\int \varphi \,d\mu: \mu\in \MMM_F(\Sing) \Big\}.
\]
As discussed in Remark \ref{rmk:usc}, the geodesic flow is entropy-expansive and thus the entropy map $\mu \mapsto h_\mu(f_1)$ is upper semi-continuous.
This guarantees that there exists $\nu \in \MMM_F(\Sing)$ with $h_\nu(f_1) + \int \ph\,d\mu = P(\Sing,\ph)$.

If $P(\Sing,\ph) = P(\ph)$, then $\nu$ is an equilibrium state for $(T^1M, F, \ph)$, and even if it happens that
$\nu$ is the unique equilibrium state (which can be arranged, but is not generally expected), it is not fully supported. Thus in order to obtain the classical conclusion of unique equilibrium state and full support, we require a \emph{pressure gap} $P(\Sing,\ph) < P(\ph)$.

To see that the case $P(\Sing,\ph) = P(\ph)$ can actually occur, we observe that there is a natural $(f_t)$-invariant volume measure $\mu_L$ on $X = T^1M$ called the \emph{Liouville measure}. Locally, $\mu_L$ is the product of the Riemannian volume on $M$ and Haar measure on the unit sphere of dimension $n-1$. 
Using the Ruelle--Margulis inequality, the Pesin entropy formula, and the fact that $-\int \varphi^u d\mu$ is the sum of the positive Lyapunov exponents for $\mu$ (where $\ph^u$ is the geometric potential), one can show that $P(\ph^u) = 0$ and that $\mu_L$ is an equilibrium state for $\ph^u$. 

In negative curvature, $\ph^u$ is H\"older and $\mu_L$ is the unique equilibrium state. In non-positive curvature, however, $\mu_L$ often fails to be the unique equilibrium state.\footnote{We mention that $\mu_L(\Reg) >0$ and that $\mu_L|_{\Reg}$ is known to be ergodic. Ergodicity of $\mu_L$, which is a major open problem, is thus equivalent to the question of whether $\mu_L(\Sing)=0$.} For example, in the surface case, it is easily checked that $P(\Sing,\ph^u) = P(\ph^u) = 0$, and any closed geodesic in $\Sing$ defines two equilibrium states for $\ph^u$ (one for each direction of travel around the geodesic). 

Since a general uniqueness result for $\ph^u$ is impossible, we often turn our attention to the one-parameter family of potentials $q \ph^u$, where $q \in \RR$. Equilibrium states for these potentials are geometrically relevant, and a natural question is to identify the range of values for $q$ so that uniqueness holds.

\subsubsection{Uniqueness given a pressure gap}\label{sec:pressure-gap}

Our main result on uniqueness of equilibrium states for geodesic flow in non-positive curvature is the following.

\begin{theorem}[Uniqueness of equilibrium states for rank 1 geodesic flow \cite{BCFT18}]\label{BCFT}
Let $(f_t)$ be the geodesic flow over a closed rank 1 manifold $M$ and let $\varphi\colon T^1M\to \RR$ be $\varphi=q\ph^u$ or be H\"older continuous. If $\ph$ satisfies the \emph{pressure gap}
\begin{equation}\label{eqn:P-gap}
P(\Sing, \varphi)<P(\varphi),
\end{equation}
then $\varphi$ has a unique equilibrium state $\mu$. This equilibrium state 
is hyperbolic, fully supported, and is the weak$^\ast$ limit of weighted regular closed geodesics in the sense of \S\ref{sec:geod-per} below.\end{theorem}

\begin{remark}\label{rmk:Knieper}
Knieper used a Patterson--Sullivan type construction on the boundary at infinity to prove uniqueness of the MME (the case $\ph=0$) and deduce the entropy gap $h(\Sing) < h(T^1M)$ from this \cite{gK98}. This construction has recently been extended to manifolds with no focal points by Fei Liu, Fang Wang, and Weisheng Wu \cite{LWW}. We work in the other direction: we need to first establish the gap (see Theorem \ref{thm:gap} below), and then use this to prove uniqueness.
\end{remark}

In \S\ref{sec:Kproperty} we discuss the following result on strengthened ergodic properties for the equilibrium states in Theorem \ref{BCFT}, due to Ben Call and the second-named author.\
 
\begin{theorem}[K and Bernoulli properties \cite{CT19}]\label{thm:K}
Any unique equilibrium state provided by Theorem \ref{BCFT} has the K-property. The unique MME has the Bernoulli property.
\end{theorem}

In dimension 2, the Margulis--Ruelle inequality gives $h(\Sing)=0$, from which the pressure gap \eqref{eqn:P-gap} follows when $\sup \ph - \inf \ph < h(X)$, via a soft argument based on the variational principle. In higher dimensions we may have $h(\Sing)>0$ (see the Gromov example in \S\ref{sec:entropygap}), and the entropy gap $h(\Sing) < h(X)$ established by Knieper is nontrivial. In \S\ref{sec:entropygap} we outline a direct proof of this gap that uses the specification property, and that generalizes to some nonzero potentials as follows.

\begin{theorem}[Direct proof of entropy/pressure gap]\label{thm:gap}
For geodesic flow on a closed rank 1 manifold $M$, every continuous potential $\ph$ that is locally constant on a neighbourhood of $\Sing$ satisfies the pressure gap condition \eqref{eqn:P-gap}.
\end{theorem}

\begin{remark}
When $\Sing$ is a finite union of periodic orbits, which is the case for real analytic surfaces of non-positive curvature, Theorem \ref{thm:gap} can be used to prove that the pressure gap holds for a $C^0$-open and dense set of potential functions.
\end{remark}

For surfaces, the fact that $\ph^u|_\Sing = 0$ and $h(\Sing)=0$ implies that $P(\Sing,q\ph^u)=0$ for all $q\in \RR$. It is an easy consequence of the Margulis--Ruelle inequality and Pesin's entropy formula that
\[
P(q \varphi^u)> 0 \text{ for } q<1,
\]
and thus $q\varphi^u$ has a unique equilibrium state for all $q<1$. We obtain the classic picture of the pressure function in non-uniform hyperbolicity, shown in Figure \ref{fig:Pqphu}. This is analogous to the familiar picture in the case of non-uniformly expanding interval maps with indifferent fixed points, e.g., the Manneville-Pomeau map \cite{PS92, mU96, oS01}.
 
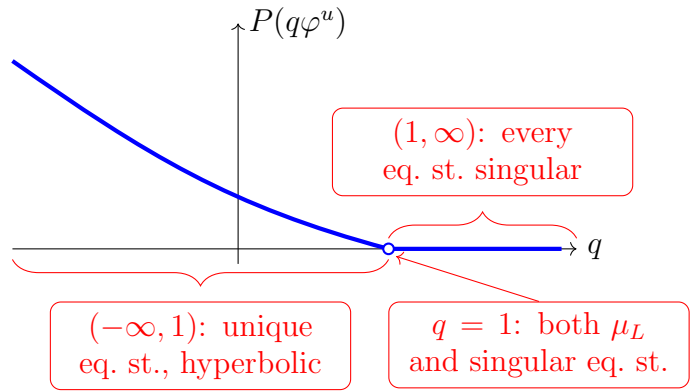
\begin{figure}[htbp]
\begin{tikzpicture}[block/.style= {rectangle,text width=9em, text = red, draw=red,align=center, rounded corners}]
\draw[->] (-3,0)--(4.5,0) node[right]{$q$};
\draw[->] (0,-0.2)--(0,3) node[right]{$P(q\varphi^u)$};
\draw[blue,ultra thick] (2,0)--(4.3,0);
\draw[blue,ultra thick] (2,0) to[out=165, in=-35] (-3,2.5);
\filldraw[fill=white,draw=blue,thick] (2,0) circle (2pt);
\begin{scope}[decoration={brace,amplitude=10,raise=4pt}]
\draw[decorate,red] (2,0)--(-3,0);
\draw[decorate,red] (2,0)--(4.5,0);
\end{scope}
\node[block,below] at (-.5,-.7){$(-\infty,1)$: unique eq.\ st., hyperbolic};
\node[block,above] at (3.25,.7){$(1,\infty)$: every eq.\ st.\ singular};
\draw[<-,red] (2.1,-.1) -- (4, -.7) node[block,below]{$q=1$: both $\mu_L$ and singular eq.\ st.};
\end{tikzpicture}
\caption{Pressure for surfaces with non-positive curvature.}
\label{fig:Pqphu}
\end{figure}

\subsubsection{Pressure and periodic orbits}\label{sec:geod-per}
We describe the sense in which the unique equilibrium state is the limit of periodic orbits, analogously to \S\ref{sec:periodic-points}.
For $a<b$, let $\Per_{R}(a, b]$ denote the set of closed regular geodesics with length in the interval $(a,b]$.\footnote{Here, we are following a notation convention of Katok: when we say a geodesic, we mean oriented geodesic, and we are considering $\gamma$ as a periodic orbit living in $T^1M$.}
For each such geodesic $\gamma$, let
$\Phi(\gamma)$ be the value given by integrating $\ph$ around $\gamma$; that is, $\Phi(\gamma):=\Phi(v, |\gamma|) = \int_0^{|\gamma|} \ph(f_t v)\,dt$, where $v\in T^1M$ is tangent to $\gamma$ and $|\gamma|$ is the length of $\gamma$. 
Given $T,\delta>0$, let
\[
\Lambda^\ast_{\Reg} (\ph, T,\delta) = \sum_{\gamma\in \Per_R(T-\delta,T]} e^{\Phi(\gamma)}.
\]
For a closed geodesic $\gamma$, let $\mu_\gamma$ be the normalized Lebesgue measure around the orbit.   We consider the measures
\[
\mureg_{T,\delta} = \frac 1{\Lambda^\ast_{\Reg} (\ph, T,\delta)} \sum_{\gamma\in \Per_R(T-\delta,T]} e^{\Phi(\gamma)} \mu_\gamma.
\]
We say that
\emph{regular closed geodesics weighted by $\varphi$ equidistribute} to a measure $\mu$ if $\lim_{T\to\infty} \mureg_{T,\delta} = \mu$ in the weak* topology for every $\delta>0$.

\subsubsection{Main ideas of the proof of uniqueness} \label{sec:bcftproofidea}

Theorem \ref{BCFT} is proved using the general result in Theorem \ref{thm:ESflows}. As observed in Remark \ref{rmk:Pexp-sing}, we have $\Pexp(\ph) \leq P(\Sing,\ph)$, so the condition $\Pexp(\ph) < P(\ph)$ follows immediately from the pressure gap assumption \eqref{eqn:P-gap}, and it remains to find a decomposition of the space of orbit segments satisfying \ref{cond:spec}--\ref{cond:P-gap}. We will do this using a function $\lambda\colon X \to [0, \infty)$ that measures `hyperbolicity'. We want this function to be such that:
\begin{enumerate}
\item $\lambda$ vanishes on $\Sing$;
\item $\lambda$ uniformly positive implies uniform hyperbolicity estimates.
\end{enumerate}
There is a convenient geometrically-defined function which has the desired properties, whose definition in dimension $2$ is simple: we let $\lambda(v)$ be the minimum of the curvature of the stable horosphere $H^s(v)$ and the unstable horosphere $H^u(v)$.\footnote{For manifolds $M$ with $\Dim(M)\geq 2$, we define $\lambda\colon T^1M\to [0,\infty)$ as follows.  Let $H^s, H^u$ be the stable and unstable horospheres for $v$. Let $\UUU^s_v \colon T_{\pi v} H^s \to T_{\pi v} H^s$ be the symmetric linear operator defined by $\UUU(v)=\nabla_vN$, where $N$ is the field of unit vectors normal to $H$ on the same side as $v$.  This determines the second fundamental form of the stable horosphere $H^s$. We define $\UUU^u_v \colon T_{\pi v} H^u \to T_{\pi v} H^u$ analogously. Then $\UUU_v^u$ and $\UUU_v^s$ depend continuously on $v$, $\UUU^u$ is positive semidefinite, $\UUU^s$ is negative semidefinite, and $\UUU^u_{-v}=-\UUU^s_v$.  For $v \in T^1M$, let $\lambda^u(v)$ be the minimum eigenvalue of $\UUU^u_v$ and let $\lambda^s(v) = \lambda^u(-v)$. Let $\lambda(v) = \min ( \lambda^u(v), \lambda^s(v))$.

The functions $\lambda^u$, $\lambda^s$, and $\lambda$ are continuous since the map $v\mapsto \UUU^{u,s}_v$ is continuous, and we have $\lambda^{u,s} \geq 0$. When $M$ is a surface, the quantities $\lambda^{u,s}(v)$ are just the curvatures at $\pi v$ of the stable and unstable horocycles, and we recover the definition of $\lambda$ stated above.}

If $v \in \Sing$, then $\lambda(v)=0$ due to the presence of a parallel orthogonal Jacobi field. The set $\{v \in \Reg: \lambda(v)=0\}$ may be non-empty, but it has zero measure for any invariant measure \cite[Corollary 3.6]{BCFT18}.

If $\lambda(v) \geq \eta >0$, then we have various uniform estimates at the point $v$, for example on the angle between $E^u(v)$ and $E^s(v)$, and on the growth of Jacobi fields at $v$.  Thus, the function $\lambda$ serves as a useful  `measure of hyperbolicity'. In particular, we get the following distance estimates:
given $\eta >0$ and $\delta= \delta(\eta)>0$ sufficiently small, $v\in T^1M$, and $w,w'\in W_\delta^s(v)$, we have 
\begin{equation}\label{eqn:ds-lambda}
d^s(f_t w, f_t w') \leq d^s(w,w') e^{-\int_0^t (\lambda(f_\tau v)-\eta/2)\,d\tau}
\text{ for all } t \geq 0,
\end{equation}
where $d^s$ is the distance on $W^s$. We get similar estimates for $w,w'\in W_\delta^u(v)$.

Now we use $\lambda$ to define a decomposition. We give a general definition since the procedure here applies not just to geodesic flows, but to other examples including the partially hyperbolic systems in \S\ref{sec:DA} and \S\ref{sec:ph} (indeed, the decomposition in \S\ref{sec:ph-good} is of this type); see \cite{bC20}.

\begin{definition}\label{def:lambda-decomposition}
Let $X$ be a compact metric space and $F=(f_t)$ a continuous flow on $X$. Let $\lambda \colon X\to [0,\infty)$ be a bounded lower semicontinuous function\footnote{This allows us to use indicator functions of open sets, which is helpful in some applications.} and fix $\eta>0$. The \emph{$\lambda$-decomposition} (with constant $\eta$) of $X\times [0,\infty)$ is given by defining
\begin{align*}
\BBB(\eta) &= \Big\{(x,t)\mid \frac{1}{t}\int_{0}^{t}\lambda(f_s(x))\,ds < \eta \Big\}, \\
\GGG(\eta) &= \Big\{(x,t)\mid \frac{1}{\rho}\int_{0}^{\rho}\lambda(f_s(x))\,ds\geq \eta \\
&\qquad\qquad\qquad \text{ and }\frac{1}{\rho}\int_{0}^{\rho}\lambda(f_{-s}f_t(x))\,ds\geq \eta \text{ for all }\rho \in [0, t] \Big\}
\end{align*}
and then putting $\PPP = \SSS = \BBB(\eta)$ and $\GGG = \GGG(\eta)$.  We decompose an orbit segment $(x,t)$ by taking the longest initial segment in $\PPP$ as the prefix, and the longest terminal segment in $\SSS$ as the suffix\footnote{We could also define the class of \emph{one-sided $\lambda$-decompositions} by taking the longest initial segment in $\BBB(\eta)$, declaring what is left over to be good, and setting $\SSS=\emptyset$, or conversely by putting $\SSS = \BBB(\eta)$ and $\PPP=\emptyset$. This formalism is defined in \cite{bC20}: the decompositions in \S \ref{sec:ph-good} are examples of one-sided $\lambda$-decompositions.}: that is,
\[
p(x,t) = \sup \{ p\geq 0 : (x,p) \in \PPP \}
\quad\text{and}\quad
s(x,t) = \sup \{ s\geq 0 : (f_{t-s} x, s) \in \SSS \}.
\]
The good core is what is left over; see Figure \ref{fig:lambda-decomposition}.
\end{definition}

\begin{figure}[htbp]
\includegraphics[width=.7\textwidth]{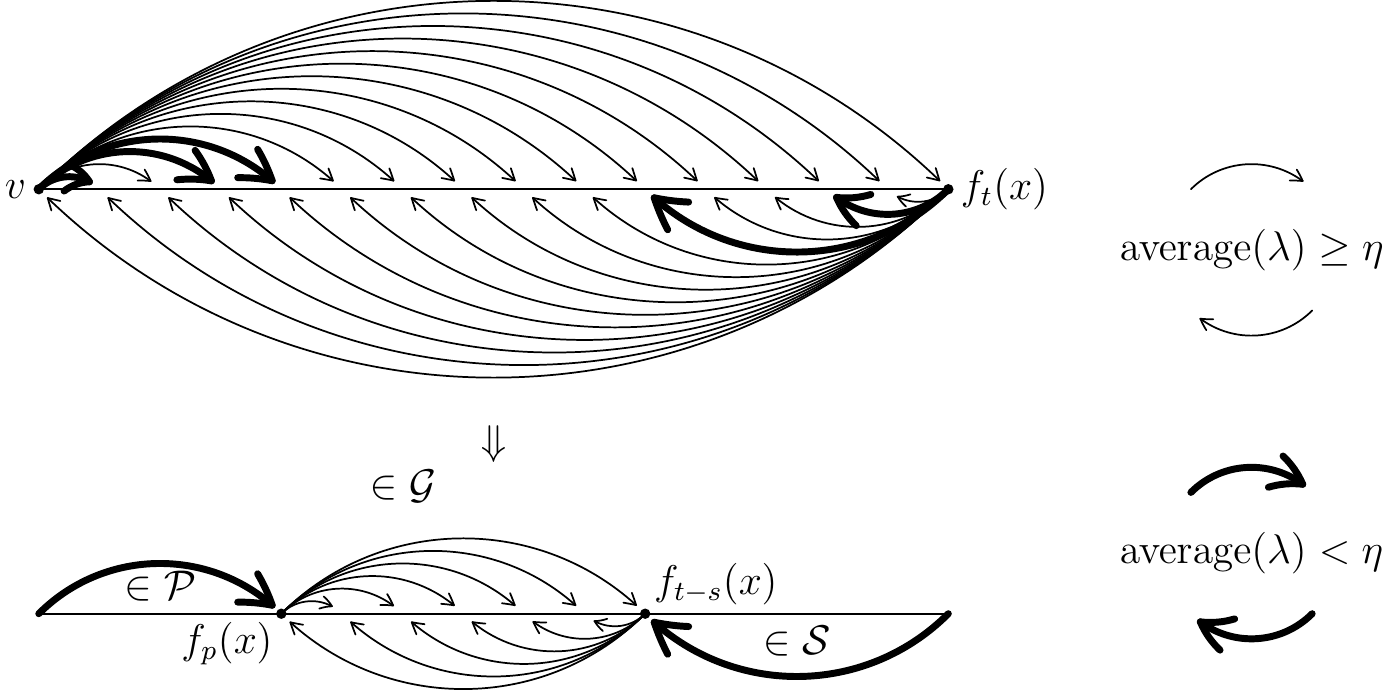}
\caption{A $\lambda$-decomposition.}
\label{fig:lambda-decomposition}
\end{figure}
 
For rank 1 geodesic flow, the decompositions associated to the horosphere curvature function $\lambda$ have the following useful properties:
\begin{enumerate}
\item we can relate $P([\PPP] \cup [\SSS],\ph)$ to $P(\Sing,\ph)$;
\item the specification and Bowen properties hold for $\GGG$ and $\varphi$.
\end{enumerate}

For the first of these, one can show that when $\eta>0$ is small, $P(\PPP \cup \SSS, \ph)$ is close to the pressure of the set of orbit segments along which the integral of $\lambda$ vanishes; this in turn can be shown to equal $P(\Sing, \ph)$. Thus the pressure gap assumption \eqref{eqn:P-gap} gives us $P([\PPP] \cup [\SSS], \ph)< P(X, \ph)$ for sufficiently small $\eta$, which is \ref{cond:P-gap} in Theorem \ref{thm:ESflows}.

For the second of these, one can in fact prove the specification property for the larger collection
\begin{equation}\label{eqn:C-eta}
\CCC(\eta)= \{ (v,t) : \lambda(v)> \eta, \lambda(f_tv) > \eta\};
\end{equation}
this will be useful in \S\ref{sec:entropygap}. Observe that $\GGG(\eta) \subset \CCC(\eta)$. The proof of the specification property is essentially the one from the uniformly hyperbolic case, as described in \S\ref{sec:specification}. See particularly Remark \ref{rmk:Wcs-spec}, and we refer to \cite[\S 4]{BCFT18} for the full proof. The key ingredient is uniformity of the local product structure at the end points of the orbit segments. This is provided by the condition that $\lambda$ is uniformly positive at these points. Then we use uniform density of unstable leaves to transition between orbit segments. We additionally need some definite expansion along the unstable of each orbit segment, which follows from the uniformity of $\lambda$ at the endpoints.

\begin{remark}
In fact, $\CCC(\eta)$ satisfies a stronger version of specification than the one formulated in Definition \ref{def:spec-flow}: one can replace the conclusion that the shadowing can be accomplished
\[
\text{for \emph{some} $0 = T_1 < T_2 < \cdots < T_k$ satisfying $T_{i+1} - T_i - t_i \in [0,\tau]$}
\] 
with the stronger conclusion that it can be accomplished
\[
\text{for \emph{every} $T_1 < T_2 < \cdots < T_k$ satisfying $T_{i+1} \geq T_i + \tau$}.
\]
That is, we are able to take all the transition times to be exactly $\tau$, or any length at least $\tau$ that we choose. This stronger conclusion is important in both the $K$-property result in \S \ref{sec:Kproperty} and the entropy gap result discussed in \S\ref{sec:entropygap}.
\end{remark}

Finally, for the Bowen property, the key is to use the distance estimate \eqref{eqn:ds-lambda} to deduce that for every $(v,t) \in \GGG(\eta)$ and $w,w' \in W_\delta^s(v)$, we have
\[
d^s(f_\tau w, f_\tau w') \leq d^s(w,w') e^{-\tau\eta/2}
\text{ for all } \tau \in [0,t],
\]
with a similar estimate along the unstables (going backwards from the end of the orbit segment).  Together with the local product structure, this allows the Bowen property on $\GGG$ for H\"older continuous potentials to be deduced from the same argument used in Proposition \ref{prop:holder-bowen}.

\begin{remark}
Since it is not known whether the geometric potential $\ph^u$ is H\"older continuous, an alternate proof is required to show that it satisfies the Bowen property on $\GGG$. This is one of the hardest parts of the analysis of \cite{BCFT18}, and relies on detailed estimates involving the Riccati equation.
\end{remark}

Combining the ideas described above verifies the hypotheses of the abstract result in Theorem \ref{thm:ESflows}, so that the pressure gap \eqref{eqn:P-gap} yields a unique equilibrium state.

\subsection{Unique MMEs for surfaces without conjugate points}\label{sec:NCP}

When $M$ is merely assumed to have no conjugate points, life is substantially harder because many of the geometric tools used in the previous section are no long available, such as convexity of horospheres, monotonicity of the distance function, and continuity of the stable and unstable foliations of $T^1M$ (cf.\ the ``dinosaur'' example of Ballmann, Brin, and Burns \cite{BBB87}).

Under the additional (strong) assumption that the flow is expansive, uniqueness of the MME was proved by Aur\'elien Bosch\'e, a student of Knieper, in his Ph.D.\ thesis \cite{aB18}.  The following result says that at least in dimension 2, we can remove the assumption of expansivity.

\begin{theorem}[{\cite{CKW}}]
\label{thm:CKW}
Let $M$ be a closed manifold of dimension 2, with genus $\geq 2$, equipped with a smooth Riemannian metric without conjugate points.  Then the geodesic flow on $T^1M$ has a unique measure of maximal entropy.
\end{theorem}

\begin{remark}
A higher-dimensional version of Theorem \ref{thm:CKW} is available \cite{CKW}, but requires additional assumptions on $M$: existence of a `background' metric with negative curvature; the divergence property; residually finite fundamental group; and a certain `entropy gap' condition.  All of these can be verified for every metric without conjugate points on a surface of genus 2.
\end{remark}

Theorem \ref{thm:CKW} is proved using a coarse-scale expansivity and specification result. Issues of coarse scale did not arise in our non-positive curvature result, where we obtained the specification property at arbitrarily small scales. This removed a great deal of technicality from the analysis. We will not discuss the general coarse-scale analogue of Theorem \ref{thm:ESflows}, since we do not use it. Instead, we state the special case where $\ph=0$ and $\GGG = X\times [0,\infty)$, which suffices for Theorem \ref{thm:CKW}. This is the continuous-time analogue of Theorem \ref{thm:hexp}.

\begin{theorem}[\cite{CT16}]\label{thm:coarse-flows}
Let $X$ be a compact metric space and $(f_t) \colon X\to X$ a continuous flow.  Suppose that $\eps>40\delta>0$ are such that $\hexp(X,(f_t),\eps) < h(X,(f_t))$, and that the system has the specification property at scale $\delta$.
Then $(X,(f_t))$ has a unique measure of maximal entropy.
\end{theorem}

Note that Theorem \ref{thm:coarse-flows} is stated using the hypothesis of specification for the entire system, without passing to a subcollection of orbit segments.  The key tool in proving this fact for surfaces without conjugate points is the \emph{Morse Lemma}, which states that if $g,g_0$ are two metrics on $M$ such that $g$ has no conjugate points and $g_0$ has negative curvature, then there is a constant $R>0$ such that if $c,\alpha$ are geodesic segments w.r.t.\ $g,g_0$, respectively, in the universal cover $\wM$ that agree at their endpoints, then they remain within a distance $R$ for along their entire length.  

Since $M$ is a surface of genus $\geq 2$, it admits a metric of negative curvature.  Given an orbit segment $(v,t) \in T^1M\times (0,\infty)$ for the $g$-geodesic flow, let $p,q$ be the start and end points of some lift of the corresponding $g$-geodesic segment to the universal cover.  Let $w \in T^1M\times (0,\infty)$ lift to the unique unit tangent vector that begins a $g_0$-geodesic segment starting at $p$ and ending at $q$, and let $s$ be the $g_0$-length of this segment.  Then $E\colon (v,t) \mapsto (w,s)$ defines a map from the space of $g$-orbit segments to the space of $g_0$-orbit segments with the property that $(v,t)$ and $E(v,t)$ remain within $R$ for their entire lengths.

Using this correspondence, one can take a finite sequence of $g$-orbit segments $(v_1,t_1),\dots, (v_k,t_k)$, find $g_0$-orbit segments $E(v_i,t_i)$ that remain within $R$, and use the specification property for the (Anosov) $g_0$-geodesic flow to shadow these (w.r.t.\ $g_0$) by a single orbit segment $(y,T)$.  Then $E^{-1}(y,T)$ is a shadowing orbit (w.r.t.\ $g$) for the original segments $(x_i,t_i)$, for which the transition times are uniformly bounded.  

Writing down the details of the scales involved, one finds that the geodesic flow for $g$, has specification at scale\footnote{In fact one can improve this estimate, but the formula is more complicated \cite{CKW}.}
$\delta = 100 A^3 R$, where $A\geq 1$ is such that $A^{-1} \leq \|v\|_g / \|v\|_{g_0} \leq A$ for all $v\in TM$. (Existence of $A$ follows from compactness.)

To apply Theorem \ref{thm:coarse-flows}, it remains to prove that obstructions to expansivity at some scale $\eps > 40\delta$ have small entropy.  The problem with this is that $R$ itself, and especially $40\delta = 4000 A^3 R$, is likely much larger than the diameter of $M$.  So at this point, it looks like the previous paragraph is completely vacuous -- \emph{any} orbit segment of the appropriate length shadows the $(v_i,t_i)$ segments to within $\delta$.

The solution is to pass to a finite cover.  By gluing together enough copies of a fundamental domain for $M$,\footnote{Formally, one needs to take a finite index subgroup of $\pi_1(M)$ that avoids all non-identity elements corresponding to a large ball in $\wM$; this is possible because $\pi_1(M)$ is \emph{residually finite}.} one can find a finite covering manifold $N$ whose injectivity radius is $>3\eps$.  Observe that
\begin{itemize}
\item the geodesic flow on $T^1M$ is a finite-to-1 factor of the geodesic flow on $T^1N$, so there is an entropy-preserving bijection between their spaces of invariant measures, and in particular there is a unique MME for the geodesic flow over $M$ if and only if there is a unique MME over $N$;
\item the argument for specification that we gave above still works for the geodesic flow on $N$, with the same scale, because this scale comes from the Morse Lemma and is given at the level of the universal cover.
\end{itemize}
So it only remains to argue that $\hexp(\eps)<\htop$ for the geodesic flow on $N$.  This is done by observing that if $d(f_tv,f_t w) < \eps$ for all $t\in\RR$ but $w$ does not lie on the orbit of $v$, then lifting to geodesics on $\wM$ and using the fact that we are below the injectivity radius of $N$ allows us to conclude that the lifts of $v,w$ are tangent to distinct geodesics between the same pair of points on the ideal boundary $\partial\wM$.  Thus if $\mu$ is any ergodic invariant measure that is \emph{not} almost expansive at scale $\eps$, then $\mu$ gives full weight to the set of vectors tangent to such ``non-unique geodesics''.

On the other hand, if $h_\mu > 0$, then $\mu$ is a hyperbolic measure by the Margulis--Ruelle inequality, and thus by Pesin theory, $\mu$-a.e.\ $v$ has transverse stable and unstable leaves.  These leaves are the normal vector fields to the stable and unstable horospheres, and thus these horospheres meet at a single point, meaning that the geodesic through $v$ is the \emph{unique} geodesic between its endpoints on the ideal boundary.  By the previous paragraph, this means that $\mu$ is almost expansive.  It follows that $\hexp(\eps) = 0 < \htop$, and so there is a unique MME by the coarse-scale result Theorem \ref{thm:coarse-flows}.

We remark that the proof technique sketched here does not extend to non-zero potentials, and a theory of equilibrium states for surfaces with no conjugate points beyond the MME case is currently not available.

\subsection{Geodesic flows on metric spaces} \label{s.cat} 
Another natural direction to extend the classical case of geodesic flow on a negative curvature manifold is to generalize beyond the Riemannian case.  The geodesic flow on a compact locally $\CAT(-1)$ metric space is one such generalization. Here, a geodesic is a curve that locally minimizes distance, and the flow acts on the space of bi-infinite geodesics parametrized with unit speed. In the Riemannian case this space is naturally identified with $T^1M$.
The $\CAT(-1)$ property is a negative curvature condition which roughly says that a geodesic triangle is thinner than a comparison geodesic triangle in the model hyperbolic space with curvature $-1$. While one expects these flows to exhibit similar behavior to the classical case, branching phenomena
and the lack of smooth structure are obstructions to some of the usual techniques.

More generally, one can study geodesic flow on a compact locally $\CAT(0)$ metric space, in which geodesic triangles are thinner than Euclidean triangles. This is a generalization of geodesic flow in Riemannian non-positive curvature.

We survey some recent results in this direction. In the $\CAT(-1)$ case (allowing cusps), the MME has been well-studied using the boundary at infinity approach, see \cite{tR03}.
Constantine, Lafont and the second-named author studied the compact locally $\CAT(-1)$ case using the specification approach \cite{CLT20a}, and later using a symbolic dynamics approach \cite{CLT20b}, proving that every H\"older continuous potential has a unique equilibrium state, and obtaining many of the strong stochastic properties one expects from the classical case (e.g., Central Limit Theorem, Bernoullicity, Large Deviations).  Broise-Alamichel, Paulin and Parkonnen \cite{BPP19} have extended the equilibrium state constructions and results of Paulin, Pollicott and Schapira \cite{PPS} to the $\CAT(-1)$ case for a restricted class of potentials which includes the locally constant ones. (See \S2.4 and \S3.2 of \cite{BPP19} for a description of this class -- in the compact case treated in \cite{CLT20a}, no such restrictions are required, as described in the introduction of \cite{CLT20a}.) The results of \cite{BPP19} give detailed information in the MME case for non-compact $\CAT(-1)$ spaces, and particularly for trees, which is the focus of their work.

The $\CAT(0)$ case has seen substantial recent advances in the MME case, notably by Ricks \cite{rR19}, who has proved uniqueness of the MME by extending Knieper's construction. A theory of equilibrium states for translation surfaces, which is an important class of $\CAT(0)$ examples, is currently being developed by Call, Constantine, Erchenko, Sawyer and Work. A theory of equilibrium states for the general $\CAT(0)$ setting is currently open.

\section{Kolmogorov property for equilibrium states} \label{sec:Kproperty}

\subsection{Moving up the mixing hierarchy}

We describe results of Ben Call and the second-named author on the Kolmogorov and Bernoulli properties \cite{CT19}.

A flow-invariant measure $\mu$ is said to have the \emph{Kolmogorov property}, or \emph{$K$-property}, if every time-$t$ map has positive entropy with respect to any non-trivial partition $\xi$: that is, for every partition $\xi$ that does not contain a set of full measure, and for every $t\neq 0$, we have $h_\mu(f_t, \xi) > 0$.\footnote{This can also be formulated in terms of the \emph{Pinsker $\sigma$-algebra} for $\mu$, which can be thought of as the biggest $\sigma$-algebra with entropy $0$: the measure $\mu$ has the $K$-property if and only if the Pinsker $\sigma$-algebra for $\mu$ is trivial.}

\begin{theorem} \label{Kproperty}
Let $F = (f_t)$ be the geodesic flow over a closed rank 1 manifold $M$ and let $\varphi\colon T^1M\to \RR$ be $\varphi=q \varphi^u$ or be H\"older continuous. If $P(\Sing, \varphi)<P(\varphi)$, then the unique equilibrium state $\mu_\varphi$ has the Kolmogorov property.
\end{theorem}

In the case $\ph=0$, the mixing property for the unique MME was known due to work of Babillot \cite{Bab}. Theorem \ref{Kproperty} strengthens this. We recall the hierarchy of mixing properties (this is an ``express train'' version of the hierarchy):
\[ 
\text{Bernoulli $\Rightarrow$ K $\Rightarrow$ mixing of all orders $\Rightarrow$ mixing $\Rightarrow$ weak mixing $\Rightarrow$ ergodic}.
\]
When $\dim(M)=2$, it was shown by Ledrappier, Lima, and Sarig \cite{LLS} that equilibrium states are Bernoulli; their proof uses countable-state symbolic dynamics for 3-dimensional flows. In higher dimensions, Theorem \ref{Kproperty} gives the strongest known results.

The implications in the mixing hierarchy are not ``if and only if''s in general.
However, in smooth settings with some hyperbolicity, a classic strategy for proving the Bernoulli property is to move \emph{up} the hierarchy, establishing $K$, and then proving that $K$ implies Bernoulli. This approach was notably carried out by Ornstein and Weiss \cite{OW73, OW98}, Pesin \cite{yP77}, and Chernov and Haskell \cite{CH96}.
In particular, a major success of Pesin theory is his proof that the Liouville measure restricted to the regular set is Bernoulli. We refer to the recent book of Ponce and Var\~{a}o \cite{PV19} for more details on this process. 
Here we simply mention that this approach can be carried out for the unique MME of rank 1 geodesic flow, and this is done in \cite{CT19}.

\begin{theorem}[Bernoulli property \cite{CT19}]
Let $(f_t)$ be the geodesic flow over a closed rank 1 manifold $M$. The unique measure of maximal entropy is Bernoulli.
\end{theorem}

\subsection{Ledrappier's approach}

The main tool in the proof of Theorem \ref{Kproperty} is a fantastic result of Ledrappier \cite{L}, which 
deserves to be more widely known. Ledrappier's proof is about one page long, and gives criteria for the $K$-property in terms of thermodynamic formalism. The original result is for discrete-time systems. We state here a version of it for flows; the proof is  given in \cite{CT19}, and in more detail in  \cite{bC20}.

Given a flow $F = (f_t)$ on a compact metric space $X$, the idea is to consider the product flow $(X\times X,F \times F)$, i.e., the flow $(f_s  \times  f_s)_{s \in \RR}$ given by 
\begin{equation}\label{eqn:prod-flow}
(f_s\times f_s)(x,y) = (f_sx, f_sy) \text{ for }s \in \RR.
\end{equation}

\begin{theorem}[Criteria for $K$-property]
Let $(X,F)$ be a flow such that $f_t$ is asymptotically entropy expansive for all $t\neq 0$, and let $\ph$ be a continuous function on $X$. Let $(X \times  X,F \times F)$ be the product flow \eqref{eqn:prod-flow}, and define $\Phi\colon X \times  X \to \RR$ by $\Phi(x_1,x_2) = \ph(x_1) + \ph(x_2)$. 

If $\Phi$ has a unique equilibrium measure in $\MMM_{F\times F}(X \times  X)$, then the unique equilibrium state for $\ph$ in $\MMM_F(X)$ has the Kolmogorov property.
\end{theorem}
	
The fact that $(X,F,\ph)$ has a unique equilibrium state when $(X\times X, F\times F, \Phi)$ does is a consequence of the following simple lemma.

\begin{lemma}\label{ProductES}
Let $\mu$ be an equilibrium state for $(X,F,\ph)$. Then $\mu \times \mu$ is an equilibrium state for $(X \times  X,F \times F,\Phi)$.
\end{lemma}

\begin{proof}
	Observe that
	$$h_{\mu \times \mu}(f_1 \times  f_1) = h_\mu(f_1) + h_\mu(f_1)$$
	and
	$$\int \Phi\,d(\mu \times \mu) = \int\ph\,d\mu + \int\ph\,d\mu.$$
	Therefore,
	$h_{\mu \times \mu}(f_1 \times  f_1) + \int\Phi\,d(\mu \times \mu) = 2P(X,F,\ph) = P(X \times  X,F \times F,\Phi).$
\end{proof}

From Lemma \ref{ProductES} we see that if $\mu, \nu$ are distinct equilibrium states for $(X,\FFF,\ph)$, then $\mu  \times  \mu$ and $\nu  \times  \nu$ are both equilibrium states for $\Phi$. If $\Phi$ has a unique equilibrium state, then this means that  $\mu  \times  \mu = \nu  \times  \nu$ and hence $\mu= \nu$; thus, we get uniqueness of the equilibrium state downstairs, and we see that if $\Phi$ has a unique equilibrium state, it must have the form $\mu  \times  \mu$ where $\mu$ is the unique equilibrium state for $\ph$. 

Now the main idea of Ledrappier's argument can be stated quite quickly:
\emph{By the argument above, if $\Phi$ has a unique equilibrium state, then so does $\ph$. Write $\mu$ for this measure; then $\mu  \times  \mu$ is the unique equilibrium state for $\Phi$. Now assume that $\mu$ is not $K$. Then $\mu$ has a non-trivial Pinsker $\sigma$-algebra. This can be used to define another equilibrium state for $\Phi$. Contradiction.}

\subsection{Decompositions for products}\label{sec:proving-K}

Given Ledrappier's result, our strategy for proving the $K$ property in Theorem \ref{Kproperty} is now clear. We want to show that the product system of two copies of the geodesic flow has a unique equilibrium state for the class of potentials under consideration.

So let's find a decomposition for the product system.

{\bf Problem:} Lifting decompositions to products in general does not work well. One fact we do have in our favor is that if $\GGG$ has good properties, then so does $\GGG  \times  \GGG$. However, we need $\GGG  \times  \GGG$ to arise in a decomposition for $(X  \times  X, F  \times  F)$. In general this does not look at all promising: for example, the reader may try to do it for the $S$-gap shifts as studied in \cite{CT12}, and will quickly see the issue.

{\bf Idea:} Work with a nice class of decompositions that \emph{does} behave well under products. We claim that the $\lambda$-decompositions from Definition \ref{def:lambda-decomposition} form such a class. To see this, suppose we have a $\lambda$-decomposition $(\PPP,\GGG,\SSS)$ for a flow $(X,F)$, and define $\tilde \lambda \colon X  \times  X \to [0, \infty)$ by
\begin{equation}\label{eqn:tilde-lambda}
\tilde{\lambda}(x,y) = \lambda(x)\lambda(y).
\end{equation}
This function inherits lower semicontinuity from $\lambda$, and we can consider the $\tilde\lambda$-decomposition $(\tilde\PPP,\tilde\GGG,\tilde\SSS)$ for $(X\times X, F\times F)$. 

Given $((x,t),(y,t)) \in \tilde\GGG$, it follows from \eqref{eqn:tilde-lambda} and boundedness of $\lambda$ that we have $(x,t), (y,t) \in \GGG$ (with an appropriate choice of $\eta$), and thus $\tilde \GGG \subset \GGG \times \GGG$. This means that specification and the Bowen property for $\tilde\GGG$ can be deduced from the corresponding properties for $\GGG$.

But how big are $\tilde \PPP$ and $\tilde \SSS$? If $\lambda=0$ on one of the coordinates, then anything is allowed on the other. Roughly, we can show that:
\[
P(\tilde \PPP \cup \tilde \SSS, \Phi) \approx P(\ph) + P(\PPP \cup \SSS, \ph).
\]
Recall that $P(\Phi)=2P(\ph)$. Thus, if we have $P(\PPP \cup \SSS, \ph)< P(\ph)$, then we expect to be able to obtain the estimate $P(\tilde \PPP \cup \tilde \SSS, \Phi) < P(\Phi)$. This is the strategy carried out in \cite{CT19,bC20}.

\subsection{Expansivity issues}

Specification and regularity are not the whole story; in fact, dealing with continuous time and related expansivity issues is the most difficult point in our analysis.

Recall from \eqref{eqn:NE-flow} that for flows we define
\[
\NE(\epsilon, F):=\{ x\in X \mid \Gamma_\epsilon(x)\not\subset  f_{[-s,s]}(x) \text{ for any }s>0 \}.
\]
For a product flow as in \eqref{eqn:prod-flow}, the set $\Gamma_\epsilon(x, y)$ always contains $f_{[-s,s]}x  \times  f_{[-s,s]}y$. That is, we are considering a flow with a 2-dimensional center. The theory in \S\ref{sec:flow-unique} does not apply directly because $\NE(\eps,F\times F)$ as defined for a flow is the whole space! We have to build a new theory that uses information about
\begin{equation}\label{eqn:NE-prod}
\operatorname{NE^{\times} }(\epsilon) := \{(x,y)\in X \times  X\mid \Gamma_\epsilon(x,y)\not\subset f_{[-s,s]}(x) \times  f_{[-s,s]}(y)\text{ for any }s > 0\}.
\end{equation}

There are no new difficulties with counting estimates, but serious issues arise when we build adapted partitions. In the discrete time case, our adapted partition elements look like pixels and can be used to approximate sets. In the flow case, our adapted partition elements approach a small piece of orbit, so look like thin cigars. Collections of partition elements can thus be used to approximate flow-invariant sets. In the `product of flows' case, the best we can do is approximate sets invariant under $f_s  \times  f_t$ for \emph{all} $s, t \in \RR$. This creates new technical obstacles that must be overcome in our uniqueness proof. In particular, to run our ergodicity proof, we need to be able to approximate sets which are invariant only under $f_s  \times  f_s$ for all $s \in \RR$. This disconnect is a fundamental  additional difficulty.

\
In \cite{CT19}, this difficulty is overcome by proving weak mixing for $\mu$ using a lower joint Gibbs estimate  which gives a kind of partial mixing for sets that are flowed out by a small time interval. This can be used to prove weak mixing of $\mu$ by a spectral argument. This is equivalent to the desired ergodicity of $\mu  \times  \mu$.

\section{Knieper's entropy gap} \label{sec:entropygap}

\subsection{Entropy in the singular set}\label{sec:sing-entropy}

For the geodesic flow on a rank 1 non-positive curvature manifold, we have stated and discussed our main results on uniqueness of equilibrium states, and the K property for these equilibrium states. Our results hold under the hypothesis of the pressure gap $P(\Sing, \ph)< P(\ph)$. Thus, being able to verify the pressure gap is of central importance for our results. In this section we outline the proof that the gap holds for $\ph=0$, when it reduces to the \emph{entropy gap} $h(\Sing) < h(X)$. The argument extends easily to potentials that are locally constant on a neighbourhood of $\Sing$, as claimed in Theorem \ref{thm:gap}.

Our introduction of rank 1 manifolds in \S\ref{sec:surfaces} focused on examples where $\Sing$ contains only periodic orbits and has $0$ entropy, and indeed for any surface of nonpositive curvature, one can observe that every $\mu\in \MMM^e_F(\Sing)$ has $h_\mu(f_1) \leq \lambda^+(\mu) = \int -\ph^u\,d\mu = 0$ by the Margulis--Ruelle inequality, where the last equality uses the fact that $\ph^u|_{\Sing} \equiv 0$ for surfaces. Then the variational principle give $h(\Sing)=0$, and since $h(X)>0$ for all surfaces of genus at least $2$, the entropy gap holds.

In higher dimensions, however, $\Sing$ can be more complicated\footnote{In dimension 2, it is in fact an open problem whether $\Sing$ can contain non-periodic orbits \cite{BM19}, but this does not affect the argument that $h(\Sing)=0$.}
and it is not at all clear a priori that the entropy gap should always hold. The Gromov example described in \cite[\S6]{gK98} 
demonstrates that starting in dimension $3$, we may have $h(\Sing)>0$. To construct this example, let $M_0$ be a surface of constant negative curvature with one infinite cusp.  Now cut off the cusp and flatten the end so that it is isometric to a flat cylinder with radius $r$.  Take the product $M_1=M_0\times S$, where $S$ is the circle of radius $r$.  This defines a non-positive curvature 3-manifold with boundary, where the boundary is a flat torus $\partial M_1 = \partial M_0 \times S$.  Now let $M_2=S\times M_0$ so that $\partial M_2 = S \times \partial M_0$. Glue $M_1$ and $M_2$ along the boundaries (note that the order of the factors is reversed) to obtain a 3-manifold $M$. 

One can show that the regular set in $T^1M$ consists of all vectors in $T^1M$ whose geodesic enters the non-flat part of both $M_1$ and $M_2$. The singular set is then the set of vectors whose geodesics stay entirely on one side (or in the flat cylinder).  It is not hard to see that $h(\Sing)> 0$.
In fact, by defining $M_0$ using a cut arbitrarily high up the cusp, one can make $h(X) - h(\Sing)$ arbitrarily close to $0$, and indeed it is not immediately obvious that this difference is non-zero. Why should there be an entropy gap at all?

Knieper's work in \cite{gK98} proved that there is a unique MME for rank 1 geodesic flow, and that this measure is fully supported on $T^1M$. This in turn implies the entropy gap, as explained in \S\ref{sec:no-gap}. 

Our argument in this section differs from Knieper's by being constructive, suitable for generalization, and (hopefully) shedding light on the mechanism that drives the `entropy gap' phenomenon. In \S\ref{swarmup} we present the basic idea behind using the specification property to produce entropy in the symbolic setting, and then in \S\ref{sec:geod-flow-gap} we discuss how this approach can be extended to geodesic flow in non-positive curvature. Full details of the argument are in \cite{BCFT18}.

\subsection{Warm-up: shifts with specification} \label{swarmup}

The basic mechanism for using specification to produce entropy is simply to construct exponentially many orbit segments ``by hand''. This idea can be seen in its simplest form in the following result, which has been known since the 70's, see \cite{DGS76}.

\begin{theorem}\label{thm:spec-entropy}
Let $(X, \sigma)$ be a shift space with the following \emph{strong specification} property:
there is $\tau\in \NN$ such that for all $v,w\in\LLL=\LLL(X)$, there is $u\in \LLL_\tau$ such that $vuw\in \LLL$.
If $X$ has more than one point, then the strong specification property has positive entropy.
\end{theorem}
\begin{proof}
Fix $n\in \NN$ such that there are $w^1, w^2 \in \LLL_n$ with $w^1 \neq w^2$. For each $k \geq 1$, define a map $\Phi\colon \{1,2\}^k \to \LLL_{k(n+\tau)}$ by
\[
\Phi(\underline i) = w^{i_1} v^1 w^{i_2} v^2 \cdots v^{k-1} w^{i_k} v^k,
\]
where all the $v^j$ have length $\tau$ and the expression on the right hand side is chosen to be in the language of $X$. The existence of such a word is guaranteed by the strong specification property.

Since $w^1 \neq w^2$, we can see that $\Phi$ is injective on $\{1,2\}^k$, so $\# \LLL_{k(n+\tau)}(X) \geq 2^k$. Taking logs, dividing by $k(n+\tau)$, and sending $k\to\infty$ gives
\[
h(X) \geq \lim_{k \to \infty} \frac{1}{k(n+\tau)} \log 2^k = \frac 1 {n+\tau} \log 2 >0.\qedhere
\]
\end{proof}

We take this basic idea further, and sketch a proof of the following result about shifts with specification. The interest here is not so much in the statement, but rather in the fact that the proof contains the main entropy production idea that we will use for geodesic flow in the next section.

\begin{theorem}\label{thm:hYhX}
Consider a shift space $(X, \sigma)$ with the strong specification property. Let $Y \subset X$ be a compact invariant proper subset. Then $h(Y)< h(X)$.
\end{theorem}
\begin{proof}
We use the specification property, words in $\LLL(Y)$ and a single word $w \notin \LLL(Y)$ to construct at least  $e^{n(h(Y)+\eps)}$ words in $\LLL_n(X)$ for large $n$, giving the desired result.

Since $Y \neq X$, we can fix $w \notin \LLL(Y)$. Let $t$ be the length of $w$, and $\tau$ the gap size in the strong specification property. We fix a ``window size'' $n> t + 2\tau$; given $N\in\NN$, we divide the indices $\{1,2,\dots, nN\}$ into $N$ ``windows'' of the form $\{kn+1, kn+2, \dots, (k+1)n\}$ for $1\leq k\leq N$. In particular, given $y\in \LLL_{nN}(Y)$, we consider the subwords of $y$ that appear in each window, which have the form $u^k := y_{[kn+1,(k+1)n]}$ for $1\leq k\leq n$.

Within each window, we can perform the following `surgery' to replace $u^k$ with a word that is in $\LLL_n(X)$ but not $\LLL(Y)$:
\[
u^k \mapsto u^k_{[1, n-t-2\tau]} v^1wv^2,
\]
where the words $v^1, v^2$ of length $\tau$ are chosen as needed for the specification property.

In each of the $N$ windows of length $n$, we can decide whether to do surgery or not.  Given this choice, we use the specification property to create a new word of length $nN$; as long as we performed at least $1$ surgery, this new word lies in $\LLL(X)$ but not in $\LLL(Y)$. In this way, from a single word $y_{[1, nN]}$, we can create $2^N - 1$ new words of length $nN$ in $\LLL(X) \setminus \LLL(Y)$ by varying over all the possible choices of windows for doing this surgery procedure. Note that these words are all distinct because within each window, we can determine whether or not we did surgery by checking whether the word $w$ appears.

This looks promising; however, it is too naive: we have to be careful as we vary over $y_{[1, nN]} \in \LLL(Y)$. In any window we selected for surgery, we are losing all the information on the last $t+2\tau$ entries in the window. This means that up to $\# \LLL_{t+2\tau}$ distinct words could be mapped to the same word for \emph{each} window we select for surgery. If we select too many windows, the gain in new words is far outweighed by the loss coming from this multiplicity estimate.

{\bf Fix:} \emph{Carry out surgery on a small proportion of the windows, and argue that the number of new words created beats the loss of multiplicity.}

More precisely, fix $\alpha>0$ small. Each surgery takes place at the boundary between two windows, so we consider the $N-1$ internal boundary points of the $N$ windows, i.e., the set
\[
A = \{n, 2n, 3n, \ldots, (N-1)n\}.
\]
Assuming for convenience that $\alpha N \in \NN$, we declare $\alpha N - 1$ of the points in $A$ to be ``on'',\footnote{The idea is that we want to split a word $y_{[1,nN]}$ into $\alpha N$ subwords and perform surgeries near the points where it was split; these are the ``on'' points in $A$.} and denote the set of ``on'' points by $J$.
Let $\underline J^\alpha_N$ be the set of all such $J$, that is:
\[
\underline J^\alpha_N = \{ J \subset A : \# J= \alpha N-1\}.
\] 
Note that since $\frac{N-k}{\alpha N-k} \geq \frac 1\alpha$ for all $1\leq k < \alpha N$, we have
\[
\# \underline J^\alpha_N = \binom {N-1}{\alpha N-1} 
= \prod_{k=1}^{\alpha N-1} \frac{N-k}{\alpha N-k} \geq \Big( \frac 1\alpha \Big)^{\alpha N-1}
= \alpha e^{(-\alpha \log \alpha)N}.
\] 
Fix $y=y_{[1, nN]} \in \LLL_{nN}(Y)$. Given $J \in \underline J^\alpha_N$, we carry out our surgery procedure on the windows whose boundaries are determined by $J$.\footnote{Each such window determined by the set $J$ has length some multiple of $n$. The surgery procedure is to remove the last $t+2\tau$ symbols from each window and replace with a word of the form $v^1wv^2$ where the words $v^j$ are provided by the specification property to ensure that this procedure creates a word in $\LLL_{nN}(X)$.}
We obtain a new word $\Phi_J(y) \in \LLL_{nN}(X)$ which is definitely not in $\LLL(Y)$.

The set $\{ \Phi_J(y) : J \in \underline J^\alpha_N\}$ is disjoint because we can recover $J$ from $\Phi_J(y)$ by looking at which windows contain the ``marker'' $w$. Given $J$, the maximum number of words $y\in \LLL_{nN}(Y)$ that can have the same image $\Phi_J(y)$ is $C^{\alpha N-1}$, where $C = \#\LLL_{t+2\tau}(Y)$ is independent of $\alpha$ and $N$.
Thus if we carry out this procedure for each word in $\# \LLL_{nN}(Y)$ and each $J \in \underline J^\alpha_N$, we obtain
\[
\# \Big ( \bigcup_{ y_{[1, nN]} \in \LLL_{nN}(Y)} \bigcup_J \Phi_J(y) \Big) \geq (C^{-1})^{\alpha N-1} \binom {N-1}{\alpha N-1} \# \LLL_{nN}(Y),
\]
which gives
\[
\# \LLL_{nN}(X) \geq \alpha e^{(-\alpha \log \alpha)N}e^{-\alpha N \log C}  \# \LLL_{nN}(Y).
\]
Taking logs, dividing by $N$, and sending $N\to\infty$, we see that
\[
h(X) \geq h(Y) + \frac{\alpha}{n} (-\log \alpha - \log C).
\]
If $\alpha>0$ is chosen small enough, the quantity in brackets is positive, and thus $h(X) > h(Y)$.
\end{proof}

\subsection{Entropy gap for geodesic flow}\label{sec:geod-flow-gap}

Now we return our attention to the geodesic flow on $X=T^1M$ for a closed rank 1 non-positive curvature manifold $M$ and outline the proof of the entropy gap $h(X)> h(\Sing)$.

We follow the same entropy production strategy described in the previous section. The singular set $\Sing \subset X$ is a compact invariant proper subset. But how should we construct orbits? We do not expect that orbit segments contained in $\Sing$ will have the specification property. For example, orbit segments which are contained in the interior of a flat strip definitely do not have the specification property because of the flat geometry. If we stay $\epsilon$-close inside the flat strip on the time interval $[0,t]$, the amount of additional time needed to escape the flat strip grows with $t$.

So we want to use a specification argument on orbit segments without specification, which does not immediately look promising. Let us recall what kind of orbits \emph{do} have specification: it suffices to know that both the start and end of the orbit segment are `uniformly' in the regular set. 

More precisely, for any $\eta>0$, we have the specification property on the collection
\[
\CCC(\eta) = \{ (x,t): x, f_tx \in \Reg(\eta) \},
\]
where $\Reg(\eta)= \{ x : \lambda (v) \geq \eta\}$. See \S \ref{sec:bcftproofidea} for the definition of $\lambda$ and discussion of why the specification property holds on $\CCC(\eta)$. 

In order to make use of this fact, we require a reasonable way to approximate orbit segments in $\Sing$ by orbit segments in $\CCC(\eta)$. This will be given by a map $\Pi_t\colon \Sing\to\Reg$, which can be roughly summarized by the following slogan (which doesn't make sense as a rigorous statement):
\begin{quote}
\emph{Move the start of $(v,t)$ along its stable into $\Reg(\eta)$. Move the end along an unstable into $\Reg(\eta)$.}
\end{quote}
We now explain the construction that makes this idea precise. In our approximation of $(v,t)$, we ask that:
\begin{enumerate}
\item $\Pi_t(v), \Pi_t(f_tv) \in \Reg(\eta)$.
\item there exists $L$ so $f_s(\Pi_t v)$ and $\Sing$ are close for $s \in [L, t-L]$.
\end{enumerate}
In the second property, one might hope to find $L$ so $f_s(\Pi_tv)$ and $f_s v$ are close for $s \in [L, t-L]$; however, this is too much to ask for. We can see the issue if $(v, t)$ is in the middle of a flat strip; the best we can hope for  is that the orbit of $\Pi_t(v)$ approaches the \emph{edge} of the flat strip; see Figure \ref{fig:regularize}, which also illustrates the following ``regularizing'' procedure.

We fix $\eta_0$ so $\Reg(\eta_0)$ has nonempty interior. Then using density of stable and unstable leaves, together with a compactness argument, we show the following:
There exists $R>0$ such that for every $v\in T^1M$ we have both $W_R^s(v) \cap \Reg(\eta_0)\neq\emptyset$ and $W_R^u(v)\cap \Reg(\eta_0)\neq\emptyset$. 

Using this fact, given $v \in \Sing$, choose $v' \in W_R^s(v) \cap \Reg(\eta_0)$. Then for $f_t(v')$, choose $f_t(w) \in W_R^u(f_t v')\cap \Reg(\eta_0)$. Define $\Pi_t(v):=w$. 

By continuity of $\lambda$, we have $\lambda(w) \geq \eta$ for an $\eta$ slightly smaller than $\eta_0$. We can argue that the function $\lambda^u(f_tw)$ is small along all of the orbit segment except for an initial and terminal run of uniformly bounded length. This in turn implies that $d(f_tw, \Sing)$ is small, giving us condition (2).   The reason $\lambda^u(f_tw)$ must be small away from the ends of the orbit segment is that otherwise small local stable and unstable manifolds centered here would get big too fast, contradicting that the endpoints of the orbit segment are in stable and unstable manifolds of size $R$. This is made precise by Proposition 3.13 of \cite{BCFT18}, which tells us that on a compact part of the regular set, for fixed $\eps$ and $R$, an $\eps$-stable/unstable manifold grows in a uniform amount of time to cover a $R$-stable/unstable manifold. 
\begin{figure}[htbp]
\includegraphics[width=\textwidth]{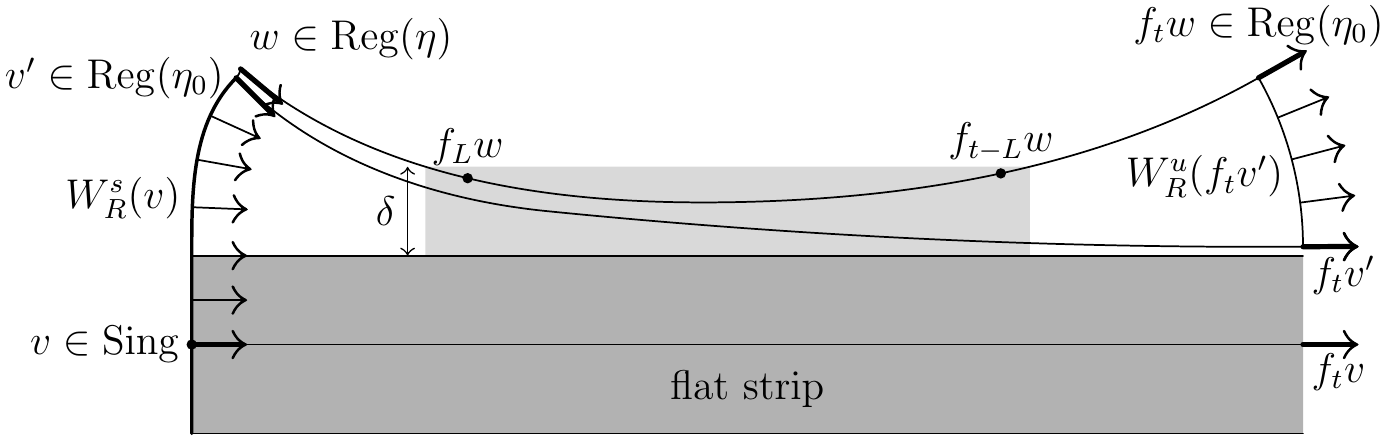}
\caption{The regularizing function $\Pi_t \colon v \mapsto w$.}
\label{fig:regularize}
\end{figure}

In conclusion, we obtain the  following properties:

\begin{theorem}\label{thm:sing-to-reg}
For every $\delta>0$ and $\eta \in (0,\eta_0)$, there exists $L>0$ such that for every $v\in \Sing$ and $t\geq 2L$, the image $w = \Pi_t(v)$ has the following properties:
\begin{enumerate}[label=\textup{(\arabic{*})}]
\item\label{wfw}
 $w, f_t(w) \in \Reg(\eta)$;
\item\label{nearSing}
$d(f_s(w),\Sing) < \delta$ for all $s\in [L,t-L]$;
\item \label{component}
for every $s\in [L,t-L]$, $f_s(w)$ and $v$ lie in the same connected component of $B(\Sing,\delta) := \{w\in T^1M : d(w,\Sing)<\delta)\}$.
\end{enumerate}
\end{theorem}

This result is found in \cite[Theorem 8.1]{BCFT18}, where the proof of \ref{nearSing} contains some typos: we take this opportunity to correct these typos by providing a complete proof here. (Most of this proof is word-for-word identical to the one in \cite{BCFT18}.)

\begin{proof}[Proof of Theorem \ref{thm:sing-to-reg}]
Let $\delta,\eta,\eta_0$ be as in the statement of the theorem. For property \ref{wfw}, it is immediate from the definition of $\Pi_t$ that $\lambda(f_tw) \geq \eta$.  By uniform continuity of $\lambda$, we can take $\eps_0$ sufficiently small such that if $v_2 \in W_{\eps_0}^u(v_1)$ and $\lambda(v_1) \geq \eta_0$, then $\lambda(v_2) \geq \eta$.  By \cite[Corollary 3.14]{BCFT18},
there exists $T_0>0$ such that if $t \geq T_0$ and $f_{t}(w) \in W_R^u(f_{t}v')$, then $w\in W_{\eps_0}^u(v')$. Thus, if $\lambda(v')\geq \eta_0$, then $\lambda(w) \geq \eta$. Thus, item \ref{wfw} of the theorem holds for any $t\geq T_0$.

We turn our attention to item \ref{nearSing}. \cite[Proposition 3.4]{BCFT18} tells us that there are $\eta',T_1>0$ such that
\begin{equation}\label{eqn:T1}
\text{if }\lambda^u(f_s v)\leq \eta' \text{ for all } |s|\leq T_1, \text{ then } d(v,\Sing) < \delta.
\end{equation}
Given $v\in \Sing$, we have $\Pi^s(v) = v' \in W_R^s(v)$, and $\lambda(f_s v) = 0$ for all $s$.  

By continuity of $\lambda^u$, we can take $\eps_1$ sufficiently small such that if $v_2 \in W_{\eps_1}^s(v_1)$, then $|\lambda^u(v_1)-\lambda^u(v_2)| <\eta'/2$. 
Applying \cite[Proposition 3.13]{BCFT18} to the compact set
$\{v : \lambda^u(v) \geq \eta'/2\} \subset \Reg$ gives $T_2>0$ such that if $\lambda^u(v_1) \geq \eta'/2$ and $\tau\geq T_2$, then $f_{-\tau} W_{\eps_1}^s(v_1) \supset W_R^s(f_{-\tau} v_1)$
and $f_\tau W^u_{\epsilon_1}(v_1) \supset W_R^u(f_\tau v_1)$.

Suppose for a contradiction that $\lambda^u(f_s v') \geq \eta'/2$ for some $s \geq T_2$.  Applying the previous paragraph with $v_1 = f_s v'$ gives $f_s v \in f_s W_R^s(f_s v') \subset W_{\eps_1}^s(f_s v')$.  By our choice of $\eps_1$, this gives $\lambda^u(f_s v)>0$, contradicting the fact that $v\in \Sing$, and we conclude that $\lambda^u(f_s v') < \eta'/2$ for  $s\geq T_2$.

Similarly, if there is $s\in [T_2,t-T_2]$ such that $\lambda^u(f_s w) \geq \eta'$, then the same argument with $v_1 = f_s w$ and $\tau=t-s$ gives $f_sv' \in f_{-(t-s)} W_R^u(f_t w) \subset W_{\eps_1}^u(f_s w)$, and our choice of $\eps_1$ gives $\lambda^u(f_s v') \geq \lambda^u(f_s w) - \eta'/2 \geq \eta'/2$, a contradiction since $\lambda^u(f_s v') < \eta'/2$ for all $s\geq T_2$.  Thus $\lambda^u(f_s w) < \eta'$ for all $s\in [T_2,t-T_2]$.

Applying \eqref{eqn:T1} gives $d(f_sw,\Sing) < \delta$ for all $s\in [T_2 + T_1, t-T_2 - T_1]$. Thus, taking $L = \max(T_0, T_1 + T_2)$, assertions \ref{wfw} and \ref{nearSing} follow for $s\geq 2L$.

For item \ref{component} of the theorem, we observe that $v$ and $w$ can be connected by a path $u(r)$ that follows first $W_R^s(v)$, then $f_{-t}(W_R^u(f_tv'))$ (see Figure), and that the arguments giving $d(f_sw,\Sing)<\delta$ also give $d(f_su(r),\Sing)<\delta$ for every $s\in [L,t-L]$ and every $r$.  We conclude that $f_sv$ and $f_sw$ lie in the same connected component of $B(\Sing,\delta)$ for every such $s$.  
\end{proof}

The collection $\{ (\Pi_t(v), t) : v \in \Sing\}$ has the specification property. This is because an orbit segment $(\Pi_t(v), t)$ both starts and ends in $\Reg(\eta)$. As discussed, the collection $\CCC(\eta)$ of such orbit segments has the specification property.

We certainly do not expect the map $\Pi_t$ to preserve separation of orbits. For example, in Figure \ref{fig:regularize}, we would expect a $v_2 \in \Sing$ defining a geodesic parallel to $\gamma_v$ (for example the arrow just above $v$ in the picture) to be mapped to the same (or similar) point. However, using estimates in the universal cover, which we omit here, we can argue that $\Pi_t$ has bounded multiplicity on a $(t, \epsilon)$ separated set, independent of $t$, in the following sense.

\begin{proposition}
For every $\eps>0$, there exists $C>0$ such that if $E_t \subset \Sing$ is a $(t,2\eps)$-separated set for some $t>0$, then for every $w\in T^1M$, we have $\#\{v\in E_t \mid d_t(w,\Pi_t v) < \eps\} \leq C$.
\end{proposition}

Now let us return to our entropy production argument.  It is basically the argument we saw in \S \ref{swarmup}, except that we need to apply the regularizing map $\Pi_t$ before applying the specification property, as shown in Figure \ref{fig:glue-sing}.

\begin{figure}[htbp]

\begin{tikzpicture}
\def\h{.3}
\draw (0,0)--(8.5,0) (9.5,0)--(12,0);
\draw (9,0) node {$\dots$};
\draw (0,0)--(0,\h) node[above] {$0$};
\draw (1,0)--(1,\h) node[above] {$n$};
\draw (2,0)--(2,\h) node[above] {$2n$};
\draw (3,0)--(3,\h) node[above] {$3n$};
\draw (4,0)--(4,\h) node[above] {$4n$};
\draw (5,0)--(5,\h) node[above] {$5n$};
\draw (6,0)--(6,\h) node[above] {$6n$};
\draw (7,0)--(7,\h) node[above] {$7n$};
\draw (8,0)--(8,\h) node[above] {$8n$};
\draw (10,0)--(10,\h) node[above] {$(N-2)n$};
\draw (11,0) node[below] {$(N-1)n$} --(11,\h);
\draw (12,0)--(12,\h) node[above] {$Nn$};
\draw[thick,blue, ->] (5,-.2) to[out=-50,in=50] (5,-.8);
\node[blue,left] at (5,-.5) {cut at elements of $J$ (circled)};
\draw[thick,red] (2,\h) ellipse (.25 and 2*\h);
\draw[thick,red] (7,\h) ellipse (.25 and 2*\h);
\draw[thick,red] (8,\h) ellipse (.25 and 2*\h);
\def\T{.2};
\draw (0,-1)--(2-\T,-1) (2,-1)--(7-\T,-1) (7,-1)--(8-\T,-1) (8,-1)--(8.5,-1) (9.5,-1)--(12,-1);
\draw (9,-1) node {$\dots$};
\draw[thick,blue,->] (0,-1.1) -- (0,-1.9);
\node[blue,right] at (0,-1.5) {$\Pi_{\ell_1-T}$};
\draw[thick,blue,->] (2,-1.1) -- (2,-1.9);
\node[blue,right] at (2,-1.5) {$\Pi_{\ell_2-T}$};
\draw[thick,blue,->] (7,-1.1) -- (7,-1.9);
\draw[thick,blue,->] (8,-1.1) -- (8,-1.9);
\node[blue] at (5,-1.5) {regularize};
\node[red,right] at (12,-2) {$B(\mathrm{Sing},\delta)$};
\draw[red,dashed] (-.1,-2.2) -- (13,-2.2);
\draw[red,dotted] (-.1,-2.3) -- (13,-2.3);
\node[red,right] at (12,-2.5) {$\mathrm{Reg}(\eta)$};
\draw (0,-2.5) to[out=80,in=180] (.2,-2) -- (2-\T-.2,-2) to[out=0,in=100] (2-\T,-2.5);
\draw (2,-2.5) to[out=80,in=180] (2.2,-2) -- (7-\T-.2,-2) to[out=0,in=100] (7-\T,-2.5);
\draw (7,-2.5) to[out=80,in=180] (7.2,-2) -- (8-\T-.2,-2) to[out=0,in=100] (8-\T,-2.5);
\draw (8,-2.5) to[out=80,in=180] (8.2,-2) -- (8.5,-2);
\draw (9.5,-2) -- (12-\T-.2,-2) to[out=0,in=100] (12-\T,-2.5);
\draw[thick,blue, ->] (5,-2.5) to[out=-50,in=50] (5,-3.3);
\node[blue,left] at (5,-2.9) {``glue'' with specification};
\begin{scope}[shift={(0,-1.5)}]
\node[red,right] at (12,-2) {$B(\mathrm{Sing},\delta)$};
\draw[red,dashed] (-.1,-2.2) -- (13,-2.2);
\draw[red,dotted] (-.1,-2.3) -- (13,-2.3);
\node[red,right] at (12,-2.5) {$\mathrm{Reg}(\eta)$};
\draw (0,-2.5) to[out=80,in=180] (.2,-2) -- (2-\T-.2,-2) to[out=0,in=100] (2-\T,-2.5);
\draw (2,-2.5) to[out=80,in=180] (2.2,-2) -- (7-\T-.2,-2) to[out=0,in=100] (7-\T,-2.5);
\draw (7,-2.5) to[out=80,in=180] (7.2,-2) -- (8-\T-.2,-2) to[out=0,in=100] (8-\T,-2.5);
\draw (8,-2.5) to[out=80,in=180] (8.2,-2) -- (8.5,-2);
\draw (9.5,-2) -- (12-\T-.2,-2) to[out=0,in=100] (12-\T,-2.5);
\draw[blue,thick] 
(0,-2.6) to[out=0,in=180] (.2,-2.1) -- (2-\T-.2,-2.1) to[out=0,in=180] (2-\T,-2.6)
--(2,-2.6) to[out=0,in=180] (2.2,-2.1) -- (7-\T-.2,-2.1) to[out=0,in=180] (7-\T,-2.6)
--(7,-2.6) to[out=0,in=180] (7.2,-2.1) -- (8-\T-.2,-2.1) to[out=0,in=180] (8-\T,-2.6)
--(8,-2.6) to[out=0,in=180] (8.2,-2.1) -- (8.5,-2.1);
\node[blue] at (9,-2.1) {$\dots$};
\draw[blue,thick]
(9.5,-2.1) -- (12-\T-.2,-2.1) to[out=0,in=180] (12-\T,-2.6);
\node[red,below] at (5,-3) {recover $J$};
\draw[red,thick,->] (4,-3) -- (2.2,-2.6);
\draw[red,thick,->] (5.5,-3) -- (6.7,-2.7);
\draw[red,thick,->] (6,-3) -- (7.7,-2.7);
\end{scope}
\end{tikzpicture}

\caption{Gluing singular orbits.}
\label{fig:glue-sing}
\end{figure}
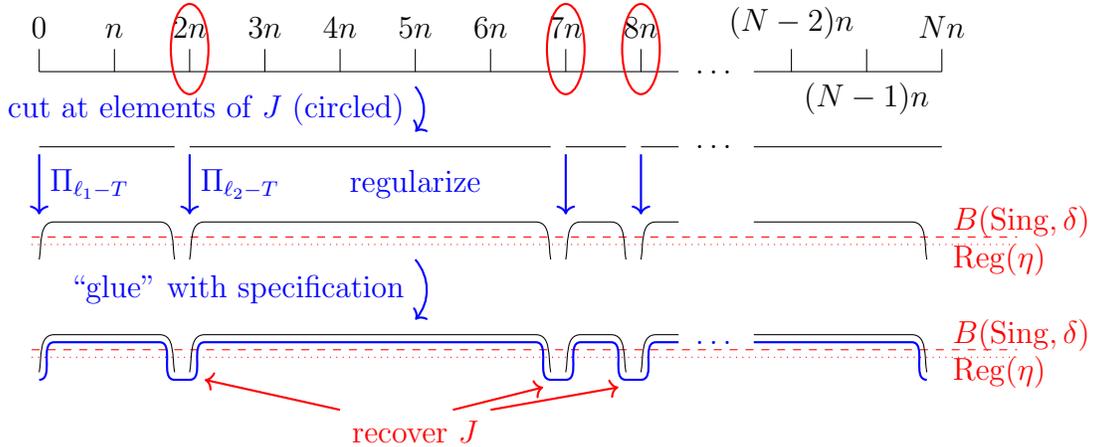

As before, consider a time window $[0, nN]$. Given a subset $J$ of $\alpha N-1$ elements from the set $\{ n, 2n, 3n, \ldots, (N-1)n \}$, we write $\ell_1, \ell_2, \dots, \ell_{\alpha N}$ for the lengths of the intervals (in order) whose endpoints are determined by $J$.

For $(v_1, v_2, \ldots, v_{\alpha N}) \in \Sing^{\alpha N}$, we apply the map $\Pi_{\ell_i-T}$ to each coordinate and glue the resulting orbit segments in $\CCC(\eta)$ using specification (where $T$ is the transition time in the specification property at a suitable scale).

Run this construction over $(g_{\ell_i-T}, \epsilon)$-separated sets for $\Sing$ in each coordinate, and for each choice of $J$, we construct exponentially more orbits than there are in $\Sing$. The argument is analogous to our previous entropy production argument: for $\alpha>0$ small, the growth from the $\binom {N-1}{\alpha N-1}$ term beats the loss coming from multiplicity in the construction.
In particular, we conclude that $h(X)> h(\Sing)$.

\subsection{Other applications of pressure production}

The argument for entropy and pressure production described above is quite flexible, and can be used in many other contexts. For example, in \cite{CT13} we used a variation on this argument to show that for a continuous potential $\varphi$ with the Bowen property on the $\beta$-shift $\Sigma_\beta$,
\[
\ulim_{n\to\infty} \frac 1 n \sum_{i=1}^{n-1} \ph(\sigma^i w^\beta) < P(\Sigma_\beta, \varphi),
\]
where $w^\beta$ is the lexicographically maximal sequence in $\Sigma_\beta$; this in turn established a pressure gap condition leading to a uniqueness result, similar to the procedure described above for geodesic flow.

Another variation of the argument can be used to prove that a unique equilibrium state $\mu_{\varphi}$ coming from Bowen's original theorem (i.e., from the assumptions of expansivity, specification and the Bowen property) satisfies
\[
P(\varphi)> \sup_{\mu \in \MMM_f(X)} \int \varphi \,d \mu,
\]
and thus that the entropy of $\mu_{\ph}$ is positive.\footnote{\url{https://vaughnclimenhaga.wordpress.com/2017/01/26/entropy-bounds-for-equilibrium-states/}} Such a potential is often called \emph{hyperbolic}. This idea was extended recently in the symbolic setting in \cite{CC}.

\bibliographystyle{amsalpha}
\bibliography{marseille}


\end{document}